\documentclass[11pt,letterpaper]{amsart}
\usepackage{amssymb,latexsym,amsmath,amsbsy}
\usepackage[mathscr]{eucal}
\usepackage[margin=1in]{geometry}

\AtBeginDocument{%
   \def\MR#1{}
}

\newtheorem{theorem}{Theorem}[section]
\newtheorem{lemma}[theorem]{Lemma}
\newtheorem{proposition}[theorem]{Proposition}
\newtheorem{corollary}[theorem]{Corollary}
\theoremstyle{definition}
\newtheorem{remark}[theorem]{Remark}

\numberwithin{equation}{section}

\newcommand{\RR}{\ensuremath{\mathbb{R}}}
\newcommand{\CC}{\ensuremath{\mathbb{C}}}
\newcommand{\TT}{\ensuremath{\mathbb{T}}}
\newcommand{\sph}{\ensuremath{\mathbb{S}}}
\newcommand{\prtl}{\ensuremath{\partial}}

\newcommand{\veps}{\ensuremath{\varepsilon}}

\newcommand{\sqrtl}{\sqrt{\Delta_g}}

\newcommand{\covt}{\nabla_{t}}

\newcommand{\expy}{\exp_{x_t}^{-1}(y)}
\newcommand{\expz}{\exp_{x_t}^{-1}(z)}
\newcommand{\expyno}{\exp_{x}^{-1}(y)}
\newcommand{\expynoz}{\exp_{x}^{-1}(z)}
\newcommand{\expynobz}{\exp_{z}^{-1}(y)}

\newcommand{\gammadot}{\dot{\gamma}}
\newcommand{\detexp}{\det\!}

\newcommand{\supp}{{\text{\rm supp}}}

\newcommand{\diag}{{\text{\rm diag}}}

\newcommand{\tr}{{\text{\rm tr}}}
\renewcommand{\Re}{\text{\rm Re}\,}
\renewcommand{\Im}{\text{\rm Im}\,}
\newcommand{\spec}{{\text{\rm spec}}}

\newcommand{\Op}{\text{\rm Op}}

\newcommand{\bg}{\bar{g}}
\newcommand{\Eb}{\mathbf{E}}
\newcommand{\Fb}{\mathbf{F}}
\newcommand{\Gb}{\mathbf{G}}
\newcommand{\Hb}{\mathbf{H}}

\newcommand{\Fsc}{\mathscr{F}}
\newcommand{\Gsc}{\mathscr{G}}
\newcommand{\Hsc}{\mathscr{H}}
\newcommand{\Ksc}{\mathscr{K}}
\newcommand{\Ssc}{\mathscr{S}}
\newcommand{\Sscwt}{\widetilde{\mathscr{S}}}
\newcommand{\Usc}{\mathscr{U}}
\newcommand{\Vsc}{\mathscr{V}}

\newcommand{\Jb}{\mathbb{J}}

\newcommand{\E}{\mathcal{E}}
\newcommand{\I}{\mathcal{I}}
\newcommand{\J}{\mathcal{J}}
\newcommand{\K}{\mathcal{K}}

\newcommand{\Lsc}{\mathscr{L}}
\newcommand{\Mno}{\mathscr{M}_{N_0}}
\newcommand{\Mnominus}{\mathscr{M}_{-N_0}}

\newcommand{\R}{\mathcal{R}}
\newcommand{\V}{\mathcal{V}}
\newcommand{\W}{\mathcal{W}}

\newcommand{\Rgt}{\mathcal{R}_{\dot\gamma(t)}}

\newcommand{\Cf}{\mathfrak{C}}
\newcommand{\Gf}{\mathfrak{G}}

\newcommand{\Spdr}{{\text{\rm \bf Sp}}(2(d-1))}
\newcommand{\Odr}{{\text{\rm \bf O}}(2(d-1))}
\newcommand{\Udr}{{\text{\rm \bf U}}(d-1)}

\newcommand{\pbf}{\mathbf{p}}

\renewcommand{\d}{d}

\newcommand{\vertiii}[1]{{\left\vert\kern-0.25ex\left\vert\kern-0.25ex\left\vert #1
    \right\vert\kern-0.25ex\right\vert\kern-0.25ex\right\vert}}
\newcommand{\vertiiinosize}[1]{{\vert\kern-0.25ex\vert\kern-0.25ex\vert #1
    \vert\kern-0.25ex\vert\kern-0.25ex\vert}}
\newcommand{\vertii}[1]{{\left\vert\kern-0.20ex\left\vert #1
    \right\vert\kern-0.20ex\right\vert}}

\title[Nonconcentration of eigenfunctions in Microlocal Kakeya-Nikodym norms]{Nonconcentration of eigenfunctions in Microlocal Kakeya-Nikodym norms: a phase space approach }

\author[M. D. Blair]{Matthew D. Blair}
\address{Department of Mathematics and Statistics, University of New Mexico, Albuquerque, NM, USA}
\email{blair@math.unm.edu}

\begin{document}
\begin{abstract}
Previous works of the author and Sogge \cite{BlairSoggeRefinedCMP}, \cite{BlairSoggeToponogov} showed the significance of microlocal Kakeya-Nikodym averages in improving $L^p$ bounds on (approximate) eigenfunctions of the Laplacian in the high frequency limit.  These averages are formed by taking the $L^2$ norm of an eigenfunction when localized in phase space to a small, frequency-dependent tube about a geodesic segment via a pseudodifferential operator. The former work showed that for values of $p$ beneath the Stein-Tomas exponent, $L^p$ norms are controlled by a supremum over these averages.  The latter work then showed that when $(M,g)$ has nonpositive sectional curvatures, there is a logarithmic gain in the averages. In combination, these two works improved the $L^p$ theory for eigenfunctions over the universal bounds of Sogge in this geometric setting.   

In the present work, we develop sufficient conditions for improving these averages which are more general than nonpositive curvature.  Instead our sufficient conditions are rooted in the dynamics of the geodesic flow on the tangent bundle, considering cases where the flow expands and contracts tangent vectors in at least some directions, e.g. partially hyperbolic flows.  We make use of Gaussian wave packet (phase space) transforms on the manifold in order to fully appreciate the gain these hypotheses impart on the microlocal averages.  In the process, we further develop Gaussian beam approximations to the wave equation in a coordinate invariant manner.

\end{abstract}
\maketitle
\centerline{ \bf In memoriam: David E. Blair (1940-2026)}
\section{Introduction}
Throughout this work, $(M,g)$ is a compact, closed, $C^\infty$ Riemannian manifold, of dimension $\d = \dim(M) \geq 2$. Let $\Delta_g$ denote the associated \emph{nonnegative} Laplace operator and let $L^p(M)$ denote the $L^p$ spaces on $M$ defined with respect to Riemannian measure.  It is well known that the compactness of $M$ means that the spectrum of $\Delta_g$ is discrete and nonnegative as a densely defined operator on $L^2(M)$.  In this work, we use $\varphi_\lambda$ to denote any $L^2$ normalized eigenfunction with  $\lambda \in \spec(\sqrtl )$, namely $\|\varphi_\lambda\|_{L^2(M)} = 1$ and 
\begin{equation*}
\Delta_g \varphi_\lambda = \lambda^2\varphi_\lambda, \quad \text{ or equivalently, } \quad \sqrtl \varphi_\lambda = \lambda \varphi_\lambda.
\end{equation*}

A celebrated result of Sogge \cite{Sogge88} showed that for $2 \leq p \leq \infty$, there is always an upper bound on the growth of the $L^p$ norms of eigenfunctions of the form
\begin{equation}\label{soggeefcn}
\|\varphi_\lambda\|_{L^p(M)} \lesssim_{M,p} \lambda^{\delta(p,\d)}, \quad
\delta(p,\d) = \begin{cases}
\frac{\d-1}{2}-\frac{\d}{p}, & \frac{2(d+1)}{d-1} \leq p \leq \infty, \\
\frac{\d-1}{2}(\frac 12- \frac 1p), & 2 \leq p \leq \frac{2(d+1)}{d-1}. 
\end{cases}
\end{equation}
These estimates result from much more general bound
\begin{equation}\label{soggecluster}
	\|\mathbf{1}_{[\lambda,\lambda +1]}(\sqrtl)\|_{L^2(M) \to L^p(M)} \lesssim \lambda^{\delta(p,\d)}
\end{equation}
where $\mathbf{1}_{[\lambda,\lambda +1]}(\sqrtl)$ projects functions $f$ onto the eigenspaces of $\sqrtl$ where the eigenvalue lies in the band $[\lambda,\lambda+1]$.  Indeed, since the operator is a projection, we have $\mathbf{1}_{[\lambda,\lambda +1]}(\sqrtl)\varphi_\lambda = \varphi_\lambda$ so \eqref{soggeefcn} follows from \eqref{soggecluster}.

In the same work \cite{Sogge88}, Sogge showed that exponent $\delta(p,\d)$ in \eqref{soggecluster} is sharp: there are families of nontrivial functions $f_\lambda$ such that $\mathbf{1}_{[\lambda,\lambda +1]}(\sqrtl) f_\lambda = f_\lambda$ which saturate the bound in that
\begin{equation}\label{qmsaturate}
	\|f_\lambda\|_{L^2(M)}=1 \quad \text{ and }\quad \|f_\lambda\|_{L^p(M)} \gtrsim_{M,p} \lambda^{\delta(p,\d)}.
\end{equation}
When $\frac{2(d+1)}{d-1} \leq p \leq \infty$, the family $f_\lambda$ has a profile very similar to the zonal functions on the canonical sphere $\sph^\d$.  In the other range $2 \leq p \leq \frac{2(d+1)}{d-1}$, the family $f_\lambda$ is highly concentrated in a $\lambda^{-\frac 12}$-neighborhood of any geodesic segment\footnote{We use $\tilde\gamma$ to denote geodesic segments, reserving $\gamma$ for the parametrized curve which expresses a geodesic.} $\tilde\gamma$ of sufficiently short length
\begin{equation}\label{tubeconc}
|f_\lambda(x)| \lesssim_N \lambda^{\frac{d-1}{4}}\big(1+\lambda^{\frac 12}d(x,\tilde\gamma)\big)^{-N}\quad \text{ and } \quad |f_\lambda(x)| \gtrsim \lambda^{\frac{d-1}{4}} \text { when } d(x,\tilde\gamma)\ll \lambda^{-\frac 12}.
\end{equation}
Indeed, this function can be taken so that $\|f_\lambda\|_{L^2(M)} = 1$ and $\|f_\lambda\|_{L^p(M)} \approx \lambda^{\frac{\d-1}{2}(\frac 12-\frac 1p)}$.

On the canonical sphere $\sph^\d$, $\spec(\sqrtl)= \{\sqrt{k(k+d-1)}: k = 0,1,2,\dots\}$; this is nearly an arithmetic progression in that $\sqrt{k(k+n-1)} = k +\mathcal{O}(\frac 1k)$ as $k \to \infty$.  Hence the projection onto unit width bands of frequencies $\mathbf{1}_{[\lambda,\lambda +1]}(\sqrtl)$ is essentially an eigenspace projection.  In this case, there are families of exact eigenfunctions (spherical harmonics) which saturate \eqref{soggeefcn}.  The zonal harmonics saturate the bound when $\frac{2(d+1)}{d-1} \leq p \leq \infty$.  If $\tilde\gamma$ is replaced by a great circle, the highest weight harmonics satisfy \eqref{qmsaturate} and saturate the bound when $2 \leq p \leq \frac{2(d+1)}{d-1}$.

However, it is somewhat uncommon for Riemannian manifolds to admit exact eigenfunctions which saturate the bounds \eqref{soggeefcn}. A series of works by Sogge and Zelditch \cite{SoggeZelditchMaximal}, \cite{SoggeZelditchFocal}, \cite{SoggeTothZelditch}, the latter also with Toth, characterized geometries for which the $p= \infty$ case is saturated.  In short, a necessary condition is that there is a full measure set of looping directions, a family of unit speed geodesics emanating from a point $x_0 \in M$ which return to $x_0$ at a common time.  On the other hand, if $(M,g)$ admits an elliptic closed geodesic\footnote{This means the eigenvalues of the Poincar\'e map lie on the unit circle in $\CC$, so the flow near the orbit is stable.}, then Gaussian beam constructions of Babi\v{c} and Buldyrev \cite{BabicBuldyrevBook} and Ralston \cite{RalstonQuasimode} show that there are highly accurate quasimodes $f_\lambda$ which satisfy \eqref{tubeconc} and $\|(\sqrtl-\lambda)f_\lambda\|_{L^2} = \mathcal{O}(\lambda^{-\infty}\|f_\lambda\|_{L^2})$.  Consequently, if a manifold does not admit such an orbit, then \eqref{soggeefcn} does not expect to be saturated for $2 \leq p \leq \frac{2(d+1)}{d-1}$.

In light of this, it it interesting to consider sufficient conditions which imply a gain in the bounds \eqref{soggeefcn} and generalizations of \eqref{soggecluster} which consider spectral windows of shrinking width.

\subsection{$L^p$ bounds for manifolds of nonpositive curvature}\label{SS:nonposintro} In the past 10-20 years there have many results improving upon \eqref{soggeefcn} and \eqref{soggecluster} assuming $(M,g)$ is a manifold with nonpositive sectional curvature.   Much of it has been built on the pioneering work of B\'erard \cite{Berard77}, which considered the remainder in pointwise Weyl law on such manifolds.   The estimates are of the form
\begin{equation}\label{loggainclus}
	\|\mathbf{1}_{[\lambda,\lambda +(\log\lambda)^{-1}]}(\sqrtl)\|_{L^2(M) \to L^p(M)} \lesssim_{p,M} \lambda^{\delta(p,\d)}(\log\lambda)^{-\varsigma(p,d)}, \quad \lambda > 1,
\end{equation}
where the operator here now projects on to a band of frequencies $[\lambda,\lambda +(\log\lambda)^{-1}]$ where the width shrinks logarithmically as $\lambda \to \infty$.  Since $\mathbf{1}_{[\lambda,\lambda +(\log\lambda)^{-1}]}(\sqrtl)\varphi_\lambda = \varphi_\lambda$, this implies a corresponding logarithmic gain in \eqref{soggeefcn}.  The $p=\infty$ case of \eqref{loggainclus} already is a consequence of the results in \cite{Berard77}.  Hassell and Tacy \cite{HassellTacyNonpos} then used B\'erard's methods to show that for any $\frac{2(d+1)}{d-1} < p < \infty$, the bound \eqref{loggainclus} is satisfied with with $\varsigma = \frac 12$.  For these cases, one can even relax the nonpositive curvature hypothesis to merely assume that $(M,g)$ has no conjugate points.  When $\d=2$, this was already observed by B\'erard, with the crucial estimates when $d \geq 3$ appearing in a comparatively recent work of Bonthonneau \cite{BonthonneauTheta}.

A more recent set of results concerns the bounds \eqref{loggainclus} when $p = \frac{2(d+1)}{d-1}$ and $(M,g)$ has negative or nonpositive curvature.  The first work in this direction is due to the author and Sogge \cite{BlairSoggeCriticalExp}, though this was influenced by an earlier work of Sogge \cite{sogge2015improved} (this had a weaker gain where the $\log\lambda$ in \eqref{loggainclus} was replaced by $\log\log\lambda$).  The exponent \cite{BlairSoggeCriticalExp} was somewhat unsatisfactory and it was later improved by the author with Huang and Sogge \cite{BHSimproved} and then even further by the last two authors \cite{HuangSoggeInventiones}.  In particular, this last work shows that if the curvatures are negative, then one can take $\varsigma = \frac 12$ in \eqref{loggainclus}.  

The case $p = \frac{2(d+1)}{d-1}$ is subtle since there are two families of quasimodes in \eqref{qmsaturate} which saturate that bound: those with profiles similar the zonal harmonics and those satisfying \eqref{tubeconc}.  Consequently, the key development there was to formulate strategies which rule out both types of concentration at the same time.  What preceded those results were bounds of the form \eqref{loggainclus} in the cases where $2<p<\frac{2(d+1)}{d-1}$, which used a different set of strategies when compared to \cite{HassellTacyNonpos}, but nevertheless used elements of B\'erard's approach \cite{Berard77}. 

To detail these strategies,  consider the family $\varPi$ of all geodesic segments of some fixed, sufficiently small length. Given $\tilde\gamma \in \varPi$, take Fermi coordinates\footnote{Fermi coordinates are reviewed in \S\ref{SS:singval} below.  They are a canonical way of flattening $\tilde\gamma$ as an embedded submanifold so that $d(y,\tilde\gamma) = |y'|$.} $(y',y_d) \in \RR^{d-1}\times \RR$ so that $\tilde\gamma$ is identified with 
$
\{ y' = 0, |y_\d| \ll 1\}.
$
In \cite{BlairSoggeRefinedCMP}, the author and Sogge constructed a family of pseudodifferential operators (PDOs) $Q_{\lambda,\tilde\gamma}$ defined by symbols $q_{\lambda,\tilde\gamma}$ which vanish unless $|y_\d| \ll 1$ and $|\zeta -\lambda e_d| \ll \lambda $ (where $e_d = (0,\dots,0,1)$ denotes the $d$-th standard basis vector in $\RR^d$), with 
\begin{equation}\label{symbolhypintro}
	\big| \prtl_{y'}^\alpha \prtl_{y_\d}^j \prtl_{\zeta'}^\beta \prtl_{\zeta_\d}^k q_{\lambda,\tilde\gamma}(y,\zeta) \big| \lesssim_{\alpha, j, \beta,k, N}  \lambda^{\frac{|\alpha|-|\beta|}{2}-k}\big(1 + \lambda^{\frac 12}|y'| + \lambda^{-\frac 12}|\zeta'| \big)^{-N}.
\end{equation}
These PDOs were then used to define a \emph{microlocal Kakeya-Nikodym norm} at frequency $\lambda$
\begin{equation*}
	\|f\|_{MKN(\lambda)} := \sup_{\tilde\gamma \in \varPi} \|Q_{\lambda,\tilde\gamma} f\|_{L^2(M)}.
\end{equation*}
The main result in \cite{BlairSoggeRefinedCMP} was that if 
\begin{equation}\label{unitqmcond}
	\|f_\lambda\|_{L^2(M)}=1 \quad \text{ and } \quad \|(\sqrtl - \lambda)^j f_\lambda\| \lesssim 1 \text{ for } j=1,\dots, \left\lceil \frac{d}{d+1}+\frac{d-1}{2}\right\rceil,
\end{equation}
then
\begin{equation}\label{LptoMKN}
	\|f_\lambda\|_{L^p(M)}  \lesssim_{p,M}    \lambda^{\frac{d-1}{2}(\frac 12-\frac 1p)} \|f_\lambda\|_{MKN(\lambda)}^{\frac{2(d+1)}{p(d-1)}-1} \quad \text{ for any } \quad \frac{2(d+2)}{d} <p <\frac{2(d+1)}{d-1}.
\end{equation}
The condition \eqref{unitqmcond} is technical, but is easily seen to be verified for any normalized eigenfunction $f_\lambda = \varphi_\lambda$ or any function $f_\lambda$ in the 
range of $\mathbf{1}_{[\lambda,\lambda +1]}(\sqrtl)$, or smaller window.  The significance of \eqref{LptoMKN} is that it means that if the $MKN$ bounds on a family of quasimodes $f_\lambda$ are $o(1)$ as $\lambda \to \infty$ or better, then this implies a corresponding bound on the $\|f_\lambda\|_{L^p(M)}$.

Estimates of the form \eqref{loggainclus} followed by proving that the right hand side of \eqref{LptoMKN} satisfies a logarithmic gain relative to the trivial $\|f_\lambda\|_{MKN(\lambda)} \lesssim 1$.  The author and Sogge \cite{BlairSoggeToponogov} showed that for $(M,g)$ of negative curvature, this actually followed by showing \emph{Kakeya-Nikodym} bounds
\begin{equation}\label{KNgain}
	\left(\int_{	\mathcal{T}_{\lambda^{-1/2}}(\tilde\gamma)} |f_\lambda(x)|^2dv_g(x) \right)^{\frac 12} \lesssim (\log\lambda)^{-\frac 12}, \qquad \mathcal{T}_{\lambda^{-1/2}}(\tilde\gamma) = \{x \in M: d_g(x,\tilde\gamma) \lesssim  \lambda^{-\frac 12} \},
\end{equation}
when $f_\lambda = \mathbf{1}_{[\lambda,\lambda +(\log\lambda)^{-1}]}(\sqrtl)f_\lambda $, $\|f_\lambda\|_{L^2(M)}=1$.
When $(M,g)$ is of nonpositive curvature, the same bounds hold when $d \geq 4$, but there are weaker right hand sides of $(\log\lambda)^{-\frac 14}$, $(\log\lambda)^{-\frac 12}\log\log\lambda$ when $d=2,3$ respectively.  As noted above, the approach of B\'erard was used to obtain these bounds.  Since it was seen in \cite{BlairSoggeRefinedCMP} that 
\begin{equation*}
	 \|f_\lambda\|_{MKN(\lambda)} \lesssim \sup_{\tilde\gamma \in\varPi} \left(\int_{	\mathcal{T}_{\lambda^{-1/2}}(\tilde\gamma)} |f_\lambda(x)|^2dv_g(x) \right)^{\frac 12},
\end{equation*}
this does indeed imply a gain of the form \eqref{loggainclus}.  This approach motivates the terminology ``Kakeya-Nikodym'' bounds as they are reminiscent of the one used by Bourgain \cite{BourgainBesicovitch} in his influential work on the Fourier restriction conjecture.

The set of works leading to of bound \eqref{LptoMKN} actually characterized the right hand side in terms of Kakeya-Nikodym averages alone rather than their microlocal counterparts.  The first work in this direction is due to Sogge \cite{SoggeKaknik}, with very similar results appearing at roughly the same time due to \cite{BourgainRestr}.  The author and Sogge further generalized these bounds in \cite{BlairSoggeKaknik} and \cite{BlairSoggeRefined}, the latter being the first to employ the microlocal variation.  Weaker versions of the Kakeya-Nikodym nonconcentration bounds \eqref{KNgain} were shown in \cite{SoggeZelditchL4}, \cite{BlairSoggeKaknik}.

\subsection{More general conditions} Recently, Canzani and Galkowski \cite{CanzaniGalkowkiLinfinity}, \cite{CanzaniGalkowskiLp}, have shown bounds of the form \eqref{loggainclus} when $\frac{2(d+1)}{d-1} < p \leq \infty$ under conditions which are more general than nonpositive curvature or no conjugate points.  We omit a precise statement of their hypotheses, but in short, it assumes that points in $M$ are never \emph{maximally conjugate} to one another. This essentially means that whenever $\gamma:\RR\to M$ is a nonconstant geodesic, then for each $t \neq 0$, there exists at least one Jacobi field $J$ along $\gamma$ such that $J(0)=0$, but $J(t) \neq 0$, even if other Jacobi fields may vanish at the two points.  Put another way, for any $t \neq 0$, the family of normal Jacobi fields $\tilde J$ for which $\tilde J(0) =0$ and $\tilde J(t) = 0$ is at most $n-2$ dimensional, strictly less than the maximal dimension of $n-1$.  A much earlier work on the remainder in the Weyl law which uses conditions similar to theirs is due to Volovoy \cite{Volovoy}.

The methods and assumptions in the work of Canzani and Galkowski are very effective in improving $L^p$ norms when $\frac{2(d+1)}{d-1} < p \leq \infty$, but do not extend to values of $p$ with $2 < p \leq \frac{2(d+1)}{d-1}$.  In short, it is unclear that the maximal conjugacy conditions rule out the possibility of elliptic orbits which admit highly accurate quasimodes satisfying \eqref{tubeconc}.  Indeed, the ellipticity of a closed geodesic means that the eigenvalues of its linear Poincar\'e map defined in phase space $T^*M$, lie on the unit circle within in the complex plane.  It is unclear how this relates to the conjugate point condition\footnote{ Put another way, if the differential of the flow is put into block matrix form as in \eqref{blocksymplectic}, then it is unclear that the condition on the eigenvalues relates to the invertibility of the  $B$ block (determined by $V\mapsto\exp_p(V)$).}. The author is unaware of any examples of Riemannian manifolds which possess an elliptic closed geodesic but do not the maximal conjugacy condition; this seems to be an interesting open problem.

The purpose of this work is to develop sufficient conditions, more general than nonpositive curvature, which imply nonconcentration bounds of the form 
\begin{equation}\label{MKNloggain}
	\|Q_{\lambda,\tilde\gamma} f_\lambda \|_{L^2(M)} \lesssim_M (\log \lambda)^{-\tilde\varsigma}
\end{equation}
for some $\tilde\varsigma>0$ when $\lambda >2$, which again is a marked improvement on the trivial $\mathcal{O}(1)$ bound as $\lambda \to\infty$.  
Given \eqref{LptoMKN}, this in turn yield improvements of the form \eqref{loggainclus} in the range $2 < p < \frac{2(d+1)}{d-1}$.  Our sufficient conditions use the full strength of the symbols $q_{\lambda,\tilde\gamma}$ defined in \eqref{symbolhypintro} being concentrated in a \emph{phase space} tube about the geodesic segment, in that there is localization in the spatial and momentum variables in $T^*M$ simultaneously.  This is in contrast to \eqref{KNgain} which instead shows nonconcentration in spatial tubes $\mathcal{T}_{\lambda^{-1/2}}(\tilde\gamma)$.  In this sense the bounds in this work are truly those on \emph{microlocal} Kakeya-Nikodym averages rather than their spatial counterparts.

Our sufficient conditions essentially assume that the differential of the geodesic flow\footnote{Even though microlocal analysis is largely rooted in the cotangent bundle $T^*M$, much of our treatment considers the geodesic flow on the tangent bundle $TM$.  The musical isomorphisms mean that we can equivalently work in either domain, but we prefer $TM$ in light of the history of the Sasaki metric as introduced in \S\ref{S:Sasaki}.} $\kappa_t: TM \to TM$ always has one or more directions which expand the length of tangent vectors in a uniform way.  If $\tilde\gamma$ was a closed geodesic instead of a geodesic segment, then one might expect the eigenvalues of the linear Poincar\'e map to be pertinent, and assume that there is at least one pair of them which does not lie on the unit circle.  However, since we do consider geodesic segments, we instead characterize this in terms of the singular values of $\kappa_t$.  

In \S\ref{S:examples}, we tour some examples of geometries $(M,g)$ which satisfy our condition.  To briefly preview one of them, we will consider manifolds with \emph{partially hyperbolic} geodesic flows.  In contrast to hyperbolic/Anosov flows, the differential of these flows are not assumed to exponentially expand or contract all possible tangent vectors, but instead only do so on invariant subspaces which are not of the maximal dimension.  Consequently, the differential of the geodesic flow expands the length of some tangent vectors, but not all the ones normal to the flow direction.
 
 Between this work and \cite{CanzaniGalkowskiLp}, it now remains to develop sufficient conditions which yield \eqref{loggainclus} at the critical exponent $p = \frac{2(d+1)}{d-1}$.  While this is still uncertain, the progress made for manifolds of nonpositive curvature suggests this is an important first step.
 
\subsection{Narrow spectral multipliers and the transition to semiclassical formalism}
Before stating our main results, we introduce the spectral multipliers under consideration.  We want these to favor the part of the spectrum localized to a $1/T$ neighborhood of $\lambda$, but instead of taking a rigid cutoff such as $\mathbf{1}_{[\lambda,\lambda+T^{-1}]}(\sqrtl)$, we regularize this and instead consider Schwartz class functions $\chi \in \mathcal{S}(\RR)$ such that $\supp(\hat{\chi}) \subset (-\frac 12,\frac 12)$.  We then take $\chi(T(\lambda-\sqrtl))$ as the regularized multiplier.

If we further assume that $\chi(\tau) >0$ for all $\tau \in [-\veps,\veps]$, then $\chi(T(\lambda-\sqrtl))$ is invertible on the range of $\mathbf{1}_{[\lambda,\lambda+\veps T^{-1}]}(\sqrtl)$ and Parseval's identity implies that 
\begin{equation*}
	\|\chi(T(\lambda-\sqrtl))^{-1}\circ \mathbf{1}_{[\lambda,\lambda+\veps T^{-1}]}(\sqrtl)\|_{L^2(M) \to L^2(M)} \lesssim 1.
\end{equation*}
Consequently, the following allows us to reduce to regularized multipliers, 
\begin{equation*}
	\begin{split}
	\|\mathbf{1}_{[\lambda,\lambda +\veps T^{-1}]}(\sqrtl)\|_{L^2(M) \to L^p(M)} &\lesssim_{p,M} \big\|\chi(T(\lambda-\sqrtl))\big\|_{L^2(M) \to L^p(M)} ,\\
	\|Q_{\lambda,\tilde\gamma} \circ \mathbf{1}_{[\lambda,\lambda +\veps T^{-1}]}(\sqrtl)\|_{L^2(M) \to L^2(M)} &\lesssim_{M} \big\|Q_{\lambda,\tilde\gamma} \circ \chi(T(\lambda-\sqrtl))\big\|_{L^2(M) \to L^2(M)} .
	\end{split}
\end{equation*}

In the rest of the work, we use semiclassical formalism.  For $T=T(h)$, to be defined later, we relate $h = \frac 1\lambda$ and consider $h>0$ sufficiently small.  The following characterizes the spectral multiplier in terms of the semiclassical operator $h\sqrtl$: 
\begin{equation*}
\chi(T(\lambda-\sqrtl)) = \chi(Th^{-1}(1-h\sqrtl)) = \chi_h(h\sqrtl), \quad \text{ where } \quad \chi_h(\tau) := \chi(Th^{-1}(1-\tau)).
\end{equation*}

Let $\Usc_t = e^{-it \sqrtl}$ be the half-wave evolution, that is, the unitary map on $L^2(M)$ defined by $\Usc_t f = u(t,\cdot)$, where $u:\RR \times M \to \CC$ is the unique solution to the pseudodifferential equation
\begin{equation}\label{halfwaveeqn}
	(D_t + \sqrtl) u = 0, \quad \text{ equivalently, } \quad (hD_t + h\sqrtl) u =0, \qquad \text{ where }u(0,\cdot) = f .
\end{equation}
This allows us to write $\chi_h(h\sqrtl)$ as the operator-valued integral
\begin{equation}\label{fourierintegralintro}
	\chi_h (h\sqrtl) = \frac{1}{2\pi T} \int_{-\frac T2}^{\frac T2} e^{\frac ih t} \Usc_t \,\widehat{\chi }\Big( \frac tT\Big)\,dt.
\end{equation}
This expression is the foundation of the previous progress for manifolds with no conjugate points recalled in \S\ref{SS:nonposintro} above.  These works made effective use of the method of B\'erard, which uses that the universal cover of $M$ is diffeomorphic to $\RR^d$. Therefore, the metric $g$ can be pulled back to $\RR^d$.  The evolution $\Usc_t$ on $M$ can thus be realized\footnote{Strictly speaking, the propagator $\cos(t\sqrtl)$ is used instead of $\Usc_t$ to take advantage of finite speed of propagation.} by periodizing the related corresponding operator on $\RR^d$ with respect to the group of deck transformations of the covering map, in the same way as the Poisson summation formula.  Moreover, the corresponding Riemannian wave equation on $\RR^d$ can be studied by taking a classical Hadamard parametrix.

A crucial feature of nonpositive curvature is that comparison theory for normal Jacobi fields along a geodesic $J(t)$ means that they generally satisfy good lower bounds.  Namely, Jacobi fields with $J(0)=0$ and $|\nabla_tJ(0)|_g =1$ satisfy $|J(t)| \geq |t|$ in the case of nonpositive curvature and $|J(t)| \gtrsim e^{c|t|}$ in the case of negative curvature.  It is well known that these Jacobi fields determine the differential of the exponential map $V \mapsto \exp_p V$.  This in turn has implications for the classical comparison theorems in Riemannian geometry, such as the Bishop-G\"unther volume comparison theorem.  A key observation in \cite{SoggeZelditchL4} and subsequent works is that such volume comparison means that the leading coefficient in the Hadamard parametrix decays either at a polynomial or exponential rate.  This decay played a crucial role in the author's aforementioned work with Sogge \cite{BlairSoggeToponogov} which led to the nonconcentration estimates \eqref{KNgain}.

Since we do not assume $(M,g)$ has no conjugate points, we instead consider a parametrix for $\Usc_t$ rooted in a wave packet transform on $(M,g)$ (sometimes \emph{FBI} or \emph{phase space} transform is used instead).  This will allow us to realize $\Usc_t$ as a superposition of Gaussian wave packets which are highly concentrated on scales saturating the uncertainty principle.  Since these concentration scales are the same as those for $(y',\zeta')$ in \eqref{symbolhypintro}, the wave packet analysis is well-suited for our purpose.

Given that the 2 equations in \eqref{halfwaveeqn} are identical, the operator $\Usc_t$ is not inherently a semiclassical one.  Hence the use of $h\sqrtl$ signals that we will use semiclassically defined operators moving forward.  In particular, from here on, we replace the symbols $q_{\lambda,\tilde\gamma}$ in \eqref{symbolhypintro} by ones obtained by rescaling $\eta = \zeta/h$.  We thus redefine $q_{h,\tilde\gamma}$ as symbols which satisfy 
\begin{equation}\label{hsymbolhypintro}
	\begin{split}
		\big| \prtl_{y'}^\alpha \prtl_{y_\d}^j \prtl_{\eta'}^\beta \prtl_{\eta_\d}^k q_{h,\tilde\gamma}(y,\eta) \big| 
		&\lesssim_{\alpha, j, \beta,k, N}  h^{-\frac{|\alpha|+|\beta|}{2}}\big(1 + h^{-\frac 12}|y'| + h^{-\frac 12}|\eta'| \big)^{-N},\\
		\supp(q_{h,\tilde\gamma}) &\subset \{(y,\eta): |y| \ll1, |\eta-e_d| \ll 1\}
	\end{split}	
\end{equation}
Now define the corresponding PDO as the operator $Q_{h,\tilde\gamma}$ given by the integral kernel
\begin{equation}\label{standardquantization}
	(Q_{h,\tilde\gamma} f)(y)  = \frac{1}{(2\pi h)^d} \iint e^{\frac ih (y-x)\cdot \eta} q_{h,\tilde\gamma}(y,\eta) f(x) \,d\eta \,dx.
\end{equation}
Reversing the change of variables $h\eta =\zeta$, we obtain the classically defined operator given by \eqref{symbolhypintro}.

The semiclassical formalism does have some advantages for us.  For one, our wave packet transform is rooted the work of Wunsch and Zworski \cite{WunschZworskiFBI}.  Moreover, it functions as a grand rescaling of $T^*M$ which carries regions $|\zeta | \approx \lambda$ to $|\zeta| \approx 1$.  This makes it easier to express our hypotheses as a dynamical condition satisfied in a neighborhood of the unit tangent bundle as well as apply stationary phase asymptotics.

\subsection{Hypotheses and main theorems}
Let $\kappa_t: TM \to TM$ denote the geodesic flow on $TM$, that is, $\kappa_t(p,V) = (\gamma_{p,V}(t), \dot\gamma_{p,V}(t))$, where $\gamma_{p,V}(t)$ is the geodesic determined by $(\gamma_{p,V}(0), \dot\gamma_{p,V}(0))=(p,V)$.  It is well-known that $\kappa_t$ preserves the fiber bundles\footnote{The $r=1$ case here is the \emph{unit tangent bundle.}} 
\begin{equation}\label{tgtspheresdef}
	S^{(r)}M := \{(p,V) \in TM: |V|_{g(p)} = r\}.
\end{equation} 
In Corollary \ref{C:singularvalues} below, we will show if we restrict $\kappa_t$ to any tangent sphere $S^{(r)}M$, then the differential of this restricted map $\kappa_t|_{S^{(r)}M}: S^{(r)}M \to S^{(r)}M$, at any point $(p,V)$ has singular values\footnote{For simplicity, we omit the dependence of $\sigma_j(t)$ on $(p,V)$ in the notation.} $\{\sigma_1 (t), \dots,  \sigma_{2d-2} (t), 1 \}$.  The $\sigma_j$ here are determined by tangent vectors which are normal to the flow and the singular value of $1$ not expressed by any $\sigma_j$ results from the vector field determining the flow direction (though it is possible that some of the $\sigma_j(t)$ happen to equal 1).  Singular values are well-defined for any linear map between finite dimensional inner product spaces; we use the Sasaki metric for this purpose, see \S\ref{S:Sasaki}. We always assume that the $\sigma_j(t)$ are arranged so that $\sigma_j(t) \geq \sigma_{j+1}(t)$, so in particular $\sigma_1(t)$ is the largest singular value, and hence gives the norm of the differential $\|d\kappa_t|_{S^{(r)}M}\|$ at any point.  Moreover, as shown in Proposition \ref{P:SVD} below, since $\kappa_t$ preserves the symplectic form on $TM$ (cf. \eqref{symptgtbundle}),  whenever $\sigma_j(t)$ is a singular value, then so is its reciprocal $\frac{1}{\sigma_j(t)}$.  Hence $\sigma_1(t) \geq 1$ for all $t \in \RR$.

Let $\Omega \subset TM$ be a flow-invariant subset: $\kappa_t(\Omega) \subset \Omega$ for all $t \in \RR$.  Suppose $0<\veps_0<1$ and set 
\begin{equation*}
	\Omega_0 = \Omega \cap \{(p,V) \in TM: |V|_{g(p)} \in [1-\veps_0 ,1+\veps_0 ] \}.
\end{equation*}
Expressing our results in terms of a flow-invariant subset allows for the greatest degree of generality, but we emphasize that $\Omega = TM$ is an important special case, and may be necessary to consider in ergodic and other cases.  We now define $\mu(t)$ and $\vartheta(t)$ be \emph{even}, positive, continuous functions on $\RR$ as follows. The function $\mu$ is increasing on $[0,\infty)$ and is an upper bound on the singular values, uniform for all points in $(p,V) \in \Omega_0$
\begin{equation}\label{muupperhyp}
	1 \leq \sigma_1(t) \leq \mu(t). 
\end{equation}
We emphasize that since $M$ is compact, there always exists some $\Lambda$ such that $\mu$ satisfies $\mu(t) \leq e^{\Lambda|t|}$ for all $t$, see Proposition \ref{P:gronwall} below. 
The function $\vartheta$ is defined to be a lower bound on the following product determined by the  singular values, uniform for all points $(p,V) \in \Omega_0$
\begin{equation}\label{varrholowerhyp}
1 \leq \vartheta(t) \leq \prod_{j=1}^{\d-1} \Big(\sigma_j (t)+ \frac{1}{\sigma_j(t)}\Big). 
\end{equation}
Since we assume that both $\mu,\vartheta$ are even, the inequalities in \eqref{muupperhyp}, \eqref{varrholowerhyp} are satisfied when the $\sigma_j(t)$ are replaced by $\sigma_j(-t)$.

Using $\vartheta$, we now define for $T \geq 0$
\begin{equation}\label{PTdef}
	\Theta(T) = \int_0^T \frac{1}{\sqrt{\vartheta(t)}} \,dt .
\end{equation}
Note that since $d\kappa_{t}|_{t=0}$ is the identity, $\sigma_j(t) \approx 1$ for $|t| \leq 1$, so $\Theta(T) \gtrsim 1$ for $|T| \geq 1$.  We always assume that $\Theta(T) = o(T)$.

\begin{theorem}\label{T:mainthmexp}
Suppose $\mu(t) \leq e^{\Lambda |t|}$ for all $t\in \RR$ and that $K \subset \Omega_0$ is a compact subset.  Then for any $0<c_0< \frac{1}{2\Lambda(2d+9)}$,
\begin{equation}\label{mainthmexp}
\|Q_{h,\tilde\gamma}  \circ \chi_h(h\sqrtl) \|_{L^2(M) \to L^2(M)} \lesssim \sqrt{\frac{\Theta(c_0|\log h|)}{|\log h|} } ,
\end{equation}
provided $(\tilde\gamma(t),\dot{\tilde\gamma}(t)) \subset K$ for all $t$ in the domain of the geodesic segment $\tilde\gamma$ (parameterized as unit speed) and the implicit constant can be taken to depend only on $K$.
\end{theorem}

 The best estimates in Theorem \ref{T:mainthmexp} are obtained when the function $\Theta$ in \eqref{PTdef} is uniformly bounded: $\sup_{T\geq 0} \Theta(T) <\infty$. 	One sufficient condition of interest for this is that the limit infimum at both $\pm \infty$ of the largest singular value of $d\kappa_{ t}$ satisfies 
	$
		\liminf\limits_{t\to \pm\infty }\frac{\log \sigma_1(t)}{|t|} >0.
	$
	In this case, $\sigma_1(t) \gtrsim e^{c|t|}$ for some $c>0$ and since the remaining singular values satisfy $\sigma_j(t) \geq 1$, we can take $\vartheta(t) \approx e^{c|t|}$.  As shown in \S\ref{SS:phyp}, this occurs in the case of hyperbolic and partially hyperbolic geodesic flows.
	
	Another circumstance is when the singular values satisfy $\sigma_j(t) \gtrsim 1+|t|$ for $j=1,\dots, d-1$, which means we can take $\vartheta(t) \approx (1+|t|)^{d-1}$.  This circumstance appeared in the prior work \cite{BlairSoggeToponogov} in the case of nonpositive curvature and will appear again in the example we present for integrable systems in Theorem \ref{T:integrable} below.  In this case, we have
	\begin{equation}\label{polycond}
		T \gg 1 \implies
		\frac{\Theta(T)}{T} \approx \frac 1T\int_1^T |t|^{-\frac{d-1}{2}}\,dt \lesssim
		\begin{cases}
			T^{-\frac 12}, & d=2,\\
			T^{-1}(\log T), & d=3,\\
			T^{-1}, & d \geq 4.
		\end{cases}
	\end{equation}
	
	Our next theorem allows for cases where $\mu(t)$ exhibits slower growth than exponential, such as polynomial growth.  As of now the only example we are able to provide of this is in the context of integrable systems (which is inclusive of the flat torus), see \S\ref{SS:integrable} below.  In these cases, it is possible to take $T(h)$ to be much larger than a logarithmic time scale, which leads to better estimates.
\begin{theorem}\label{T:mainthmpoly} For $h>0$ sufficiently small, small let $T = T(h)$ be the unique solution to 
	\begin{equation}\label{bigTeqn}
		\frac{\Theta(T)}{T} = h^{\frac 12}\mu(T)^{2d+9}.
	\end{equation}
	Let $K \subset \Omega_0$ be a compact subset. Then if $(\tilde\gamma(t),\dot{\tilde\gamma}(t)) \subset K$ for all $t$ in the domain of  $\tilde\gamma$ (parameterized as unit speed)
	\begin{equation}\label{mainthmpoly}
		\|Q_{h,\tilde\gamma}  \circ \chi_h(h\sqrtl) \|_{L^2(M) \to L^2(M)} \lesssim \sqrt{\frac{\Theta(T)}{T} },
	\end{equation}
	and the implicit constant can be taken to depend only on $K$.
\end{theorem}

Note that \eqref{bigTeqn} does indeed have a solution for $h>0$ sufficiently small due to the hypothesis that $\Theta(T) = o(T)$  and that $\mu(t)$ is increasing.

\begin{remark}\label{R:polyimprovement}
	To demonstrate why Theorem \ref{T:mainthmpoly} leads to better estimates, suppose we can take a bound of the form $\vartheta(t) \approx (1+|t|)^{d-1}$ as before in \eqref{polycond}, but this time assume the much stronger bound $\mu(t) \approx 1+|t|$.  This occurs in the case of integrable systems in Theorem \ref{T:integrable} below.  For simplicity, suppose $d=2$, limiting us to the first case in \eqref{polycond}.  For sufficiently small $h>0$ (leading to sufficiently large $T$), we will have that $T$ satisfies $T^{-\frac 12} \approx h^{\frac 12}T^{13}$ so that $T \approx h^{-\frac{1}{27}}$.  Hence the first case of \eqref{polycond} implies that the right hand side of \eqref{mainthmpoly} is $h^{\frac{1}{108}}$.  Similarly when $d\geq 3$, the right hand side of \eqref{mainthmpoly} is seen to be a power of $h$ depending on $d$.   When combined with \eqref{LptoMKN}, this yields a power improvement in the $L^p$ estimates relative to the universal bounds of Sogge.
\end{remark}

\begin{remark} 
An interesting recent work of Gao, Wu, and Xi \cite{gao2025sharp} extended the bound \eqref{LptoMKN} to the range $\frac{2(3d+1)}{3d-3}\leq p <\frac{2(d+1)}{d-1}$ in dimensions $d \geq 3$.  Their work also includes applications in the case of Hecke-Maass forms, which we do not consider here.  In cases where there is a logarithmic bound such as \eqref{KNgain} or  \eqref{MKNloggain}, this improves the exponent of $\log \lambda$ one obtains for values of $p$ with $\frac{2(3d+1)}{3d-3}\leq p \leq \frac{2(d+2)}{d}$.  However, this yields more meaningful gains in cases such as the one in Remark \ref{R:polyimprovement} as it further improves the power of $\lambda = h^{-1}$ over the universal bounds.
\end{remark}

An outline of the work will appear in \S\ref{SS:outline} below, after we have introduced the proof strategy.
\subsection{Notation}\label{SS:notation}. The notation $A \lesssim B$ means that $A \leq CB$ for some sufficiently large uniform implicit constant $C$ and $A \approx B$ means that $A \lesssim B$ and $B \lesssim A$.  On the other hand $A \ll B$ means that $A \leq \veps B$ for some uniform, sufficiently small $\veps>0$.  Throughout the work, we consistently use Einstein summation notation, that a sum should be taken over any repeated index.  However, in many cases we do not use subscripts and superscripts to indicate this repetition.  

We use $\langle \cdot,\cdot \rangle_g$ or $\langle \cdot,\cdot \rangle_{g(p)}$ to denote the inner product determined by the metric $g$, the latter emphasizing the dependence on the point $p\in M$.  Similarly, $|\cdot|_g$ denotes the length of a tangent vector (resp. covector) with respect to the metric (resp. cometric); if the subscript is omitted it should be interpreted as an absolute value or Euclidean length. As is typical, $\nabla$ denotes the Levi-Civita connection on $(M,g)$ with $\Gamma_{jk}^\ell$ typically denoting Christoffel symbols (of the second kind) determined by a coordinate frame.  We have many instances in which we take a covariant derivative along a curve $t \mapsto x(t) \in M$, which case, we use $\nabla_t = \nabla_{\dot x(t)}$ as short hand. 
We also take the convention that the Riemann curvature tensor is defined along vector fields $X,Y,Z$ as
\begin{equation*}
	R(X,Y)Z = \nabla_X \nabla_Y Z - \nabla_Y \nabla_X Z-\nabla_{[X,Y]}Z.
\end{equation*}
This defines a $(1,3)$-tensor, but by lowering an index, we can also define the $(0,4)$-tensor $Rm$ as
\begin{equation*}
	Rm(X,Y,Z,W) = \langle R(X,Y)Z, W\rangle_g.
\end{equation*}
In many cases, we use the symmetries of the latter, see e.g. \cite[Proposition 7.12]{LeeRiemannian}.

Given a square matrix $\omega$, we abbreviate $\omega+iI$ as $\omega+i$.

For pseudodifferential operators, we use standard quantization as in \eqref{standardquantization} in all sections except for \S\ref{S:WPansatz}, where we opt for Weyl quantization instead.  We also use Zworski's notation for $S(1)$ symbols \cite[(4.4.4)]{ZworskiSemiclassicalAnalysis} that for a symbol $a(y,\eta)$, we have $a \in S(1)$ if $|\prtl_{y,\eta}^\alpha a| \lesssim_\alpha 1$.  In cases such as \eqref{hsymbolhypintro} we typically treat support as the same as essential support (as defined in \cite{ZworskiSemiclassicalAnalysis}) because in all cases, the $\mathcal{O}(h^\infty)$ differences which arise are negliglible.

\subsection*{Acknowledgement} It is with deep gratitude that I dedicate this work in loving memory of my father David E. Blair. He was an exemplary role model throughout my life and gave me valuable insight on many aspects of the present work.  When I was stuck on proving Proposition \ref{P:covariantlogexpn}, he suggested the classical approach to normal coordinates appearing in Eisenhart's text.  He also provided many references on the Sasaki metric, including his own text \cite{BlairRiemannianBook}, which were invaluable in preparing \S\ref{S:Sasaki}.

\section{Outline of the proofs of the main theorems}\label{S:Outline}
We begin by defining $T=T(h)$ which governs the time scale over which we will propagate the half-wave evolution $\Usc_t$.  In Theorem \ref{T:mainthmpoly}, we define $T$ as in \eqref{bigTeqn}.  In Theorem \ref{T:mainthmexp}, it is easier to let
\begin{equation*}
	T = c_0|\log h|, \text{ for some } 0<c_0< \frac{1}{2\Lambda(2d+9)} .
\end{equation*}
We emphasize that $\mu(t)$ is always bounded by a small power of $h^{-1}$, namely,
\begin{equation}\label{mubddh}
	1 \leq \mu(t) \leq h^{-\frac{1}{26}} \qquad \text{ for all } 0 \leq  t \leq T.
\end{equation}
This follows from the given definitions of $T$, using that $d \geq 2$ and that $\Theta(T) = o(T)$ in the case of Theorem \ref{T:mainthmpoly}.

A routine ``$TT^*$" duality argument reduces \eqref{mainthmexp}, \eqref{mainthmpoly} to showing that 
\begin{equation}\label{initialTT}
	\|Q_{h,\tilde\gamma} \circ |\chi_h|^2(h\sqrtl)  \circ Q_{h,\tilde\gamma}^*\|_{L^2(M) \to L^2(M)} \lesssim \frac{\Theta(T)}{T} ,
\end{equation}
since the functional calculus gives that $|\chi_h|^2(h\sqrtl) = \chi_h(h\sqrtl) \circ  \chi_h(h\sqrtl)^*$.  Instead of \eqref{fourierintegralintro}, we thus make use of
\begin{equation}\label{fourierintegral}
	|\chi_h|^2 (h\sqrtl) = \frac{1}{2\pi T} \int_{-T}^T e^{\frac ih t} \Usc_t \,\widehat{|\chi|^2 }\Big( \frac tT\Big)\,dt.
\end{equation}
Since $\widehat{|\chi |^2 } (s)= (\widehat{\chi } * \widehat{\bar\chi })(s)$, this function is supported in $ (-1,1)$, only slightly larger than the support of $\widehat\chi$ which we took to be contained in $(-\frac 12, \frac 12)$.

Throughout this section we routinely (and implicitly) use the following uniform bounds:
\begin{equation*}
	\|Q_{h,\tilde\gamma}\|_{L^2(M)\to L^2(M)} = \|Q_{h,\tilde\gamma}^*\|_{L^2(M)\to L^2(M)} \lesssim 1.
\end{equation*}
Let $\beta \in C_c^\infty (0,\infty)$ be identically 1 in $[1-\frac{\veps_0}{4}, 1+ \frac{\veps_0}{4}]$ and vanishing outside of $(1-\frac{\veps_0}{2}, 1+ \frac{\veps_0}{2})$.  Consequently, 
\begin{equation*}
	(1-\beta^4)(\tau) |\chi_h|^2(\tau) \lesssim_N (1-\beta^4)(\tau)(1+Th^{-1}|\tau-1|)^{-N} \lesssim h^NT^{-N} 
\end{equation*}
and hence by Parseval's identity
\begin{equation*}
	\|(1-\beta^4)(h \sqrtl) |\chi_h|^2(h\sqrtl)\|_{L^2(M) \to L^2(M)} \lesssim_N h^NT^{-N} .
\end{equation*}
It thus suffices to show that
\begin{equation*}
	\|Q_{h,\tilde\gamma} \circ \beta^2(h\sqrtl) \circ |\chi_h|^2(h \sqrtl)  \circ \beta^2(h\sqrtl) \circ Q_{h,\tilde\gamma}^*\|_{L^2(M) \to L^2(M)} \lesssim \frac{\Theta(T)}{T} ,
\end{equation*}
as the error between this operator and the one in \eqref{initialTT} satisfies $\mathcal{O}((hT^{-1})^{\infty})$ bounds on $L^2(M)$.

Given a bump function $\tilde\beta \in C_c^\infty(0,\infty)$, we define a \emph{wave packet transform adapted to }$\tilde\beta$ as the integral operator $\Ssc: C^\infty(M)\to C^\infty (T^*M)$ defined at each $ (x,\xi) \in T^*M$ by
\begin{equation}\label{wptdef1}
	\Ssc f (x,\xi)= 2^{-\frac d2} (\pi h)^{-\frac{3d}{4}}\frac{\tilde\beta\big(|\xi|_{g(x)}\big) }{\detexp^{\frac 14}(g_{jk}(x))}\int_{M} e^{\frac ih(- \xi(\expyno) + \frac i2 d^2(x,y))} \psi(x,y)  f(y)\,dv_g(y), 
\end{equation}
where $dv_g$ denotes Riemannian volume.  Here $\expyno$ is the \emph{Riemannian logarithm}, the locally defined inverse of the exponential map $V\mapsto \exp_x(V)$ (cf. \S3 below) and $\xi(\expyno)$ expresses the value obtained by pairing the covector $\xi \in T_x^*M$ with $\expyno$.  Here $\psi$ is a bump function identically one in a neighborhood of the diagonal $x=y$, supported where $\expyno$ is well-defined.  Note that \eqref{wptdef1} is a coordinate invariant expression.  Transforms of this type on $C^\infty$ Riemannian manifolds\footnote{The $C^\infty$ category is signficant here as many works on the FBI transform assume $M$ is analytic, defining the transformation in terms of a holomorphic continuation of the heat kernel: see \cite[\S1]{WunschZworskiFBI} and references therein. } were studied by Wunsch and Zworski \cite{WunschZworskiFBI}.

In what follows, we take $\tilde\beta$ in \eqref{wptdef1} so that it is identically 1 on $\supp(\beta)$ as defined above, but vanishing outside of $(1-\veps_0,1+\veps_0)$. 	By \cite[Proposition 3.1]{WunschZworskiFBI}, $\Ssc^* \circ \Ssc$ is a semiclassical PDO whose symbol is in $S(1)$ (uniformly bounded in $C^\infty$) and the principal symbol of $\Ssc^* \circ \Ssc$ is  $\tilde\beta^2(|\xi|_{g})$.  We shall give a complete proof of this in Appendix \ref{A:wptpdo} below, though the only small difference is that the compact support of $\tilde\beta$ simplifies matters.  Consequently, $\Ssc$ extends to a bounded linear map $L^2(M) \to L^2(T^*M)$ (where $T^*M$ is endowed with Liouville measure) such that 
	\begin{equation}\label{wptL2bds}
		\| \Ssc \|_{L^2(M) \to L^2(T^*M)} = \| \Ssc^* \|_{L^2(T^*M) \to L^2(M) }\lesssim 1.
	\end{equation}

In \S\ref{S:wpt} below, we will show the following\footnote{The proof will reveal why we want the bump function in the definition of $\Ssc$ to be $\tilde\beta$ and not $\beta$.}:
\begin{proposition}\label{P:invwpt1}
	There exists a semiclassical PDO $P$ such that $\beta^2(h\sqrtl) = (\Ssc^* \circ \Ssc) \circ P + \mathcal{O}(h^\infty)$, that is, for any $N$
	\begin{equation*}
		\big\|\beta^2(h\sqrtl)-(\Ssc^* \circ \Ssc) \circ P\big\|_{L^2(M) \to L^2(M)} \lesssim_N h^N
	\end{equation*}
	The symbol of $p$ can be taken to satisfy $\supp(p) \subset \{(x,\xi) \in T^*M: |\xi|_{g(x)} \in \supp(\beta)\}$ and lies $S(1)$.
\end{proposition}

Note  that since $\beta^2(h\sqrtl)$ is self-adjoint, $\beta^2(h\sqrtl) =  P^* \circ (\Ssc^* \circ \Ssc) + \mathcal{O}(h^\infty)$ as well.  Finally, we are reduced to 
\begin{equation}\label{intromainreduction}
	\|Q_{h,\tilde\gamma} \circ P^* \circ \Ssc^* \circ (\Ssc \circ |\chi_h|^2(h \sqrtl) \circ  \Ssc^*) \circ \Ssc \circ P\circ Q_{h,\tilde\gamma}^*\|_{L^2(M) \to L^2(M)} \lesssim \frac{\Theta(T)}{T} .
\end{equation}
The expression in parentheses leads us to consider
\begin{equation}\label{Uintegral}
	\frac{1}{2\pi T} \int_{-T}^T e^{\frac ih t} \big(\Ssc \circ \Usc_t \circ \Ssc^*\big) \, \widehat{|\chi|^2 }\Big( \frac tT\Big)\,dt .
\end{equation}

In \S\ref{SS:Uint} below, we give a more detailed discussion of the signficance of the integral in \eqref{Uintegral}.  For now, we continue outlining our approach to Theorems \ref{T:mainthmexp} and \ref{T:mainthmpoly}.  

\begin{theorem}\label{T:Vtthm}
	Suppose $\Ssc$ is a wave packet transform as defined in \eqref{wptdef1}. There exists an operator $\Vsc_t:L^2(T^* M) \to L^2(M) $ defined for $|t| \leq T$ such that for some implicit constant independent of $t$,
	\begin{equation}\label{Vproperties}
		\Vsc_0 = \Ssc^* \quad \text{ and }\quad \|(hD_t + h\sqrtl) \Vsc_t\|_{L^2(T^* M) \to L^2(M) } \lesssim h^{\frac 32} \mu(t)^{2d+9} \quad\text{ for all } |t|\leq T.
	\end{equation}
\end{theorem}

We are thus led to examine the error resulting from replacing $\Usc_t \circ  \Ssc^*$ in \eqref{Uintegral} by $\Vsc_t$,
\begin{equation*}
	\frac{1}{2\pi T} \int_{-T}^T e^{\frac ih t} \Ssc \circ (\Usc_t \circ \Ssc^* -\Vsc_t) \,\widehat{|\chi|^2 }\Big( \frac tT\Big)\,dt .
\end{equation*}
Standard $L^2$ bounds for inhomogeneous equations in $hD_t + h\sqrtl$ and \eqref{Vproperties} imply that
\begin{equation*}
\|\Usc_t \circ \Ssc^* -\Vsc_t\|_{L^2(T^*M) \to L^2(M)} \lesssim h^{\frac 12} \mu(t)^{2d+9}  \quad\text{ for all } \quad|t|\leq T.
\end{equation*}
Consequently, since $\mu$ is increasing
\begin{equation*}
	\begin{split}
	\left\|\frac{1}{2\pi T} \int_{-T}^T e^{\frac ih t}  (\Usc_t\circ\Ssc^* -\Vsc_t) \,\widehat{|\chi|^2 }\Big( \frac tT\Big)\,dt \right\|_{L^2(T^* M) \to L^2(M)} 
	&\lesssim  \frac{h^{\frac 12}}{T} \int_{-T}^T  \mu(t)^{2d+9} \widehat{|\chi|^2 }\Big( \frac tT\Big) \,dt \\& \lesssim h^{\frac 12} \mu(T)^{2d+9}. 
	\end{split}
\end{equation*}
Given the $L^2$ bounds \eqref{wptL2bds}, we now see that it suffices to replace $\Usc_t$ by $\Vsc_t$ in \eqref{Uintegral}.  In the case of Theorem \ref{T:mainthmpoly}, $T$ was defined so that the right hand side is comparable to the desired right hand side in \eqref{initialTT}.  Similarly, in the case of Theorem \ref{T:mainthmexp}, the right hand side here is at least $\mathcal{O}(h^{\frac 12(1-c_0)})$, which is stronger than needed.

\begin{theorem}\label{T:tdepreduction}
For $|t| \leq T$, we have the following bound
\begin{equation}\label{tdepreduction}
	\|Q_{h,\tilde\gamma} \circ P^* \circ \Ssc^* \circ (\Ssc \circ \Vsc_t) \circ \Ssc \circ P\circ Q_{h,\tilde\gamma}^*\|_{L^2(M) \to L^2(M)} \lesssim \frac{1}{\sqrt{\vartheta(t)}} .
\end{equation}
\end{theorem}
Given this theorem, the proof of \eqref{intromainreduction} and hence the proofs of Theorems \ref{T:mainthmexp}, \ref{T:mainthmpoly} then follow. 
Indeed, given \eqref{tdepreduction}, then by Minkowski's inequality for integrals, it follows that
\begin{multline*}
	\left\|\frac{1}{2\pi T} \int_{-T}^T  e^{\frac ih t} Q_{h,\tilde\gamma} \circ P^* \circ \Ssc^* \circ (\Ssc \circ \Vsc_t) \circ \Ssc \circ P\circ Q_{h,\tilde\gamma}^* \,\widehat{|\chi|^2 }\Big( \frac tT\Big)\,dt \right\|_{L^2(M) \to L^2(M)} 
	\\
	\lesssim  \frac 1T \int_{-T}^T \frac{dt}{\sqrt{\vartheta(t)}} = \frac{\Theta(T)}{T}.
\end{multline*}

\subsection{The significance of the integral \eqref{Uintegral}: Gaussian beams and phase space kernels}\label{SS:Uint}
The operator valued integral \eqref{Uintegral} is fundamental in this work. The role of the wave packet transform is that it allows us to view functions as a superposition of generalized Gaussians of the form
\begin{equation}\label{gaussianinvariant}
	g_{x,\xi}(y) = (\pi h)^{-\frac d4} \detexp^{-\frac 14}(g_{jk}(x)) e^{\frac ih(\xi(\expyno) + \frac i2 d^2(x,y))}\psi(x,y),
\end{equation}  
for each $(x,\xi) \in T^*M$ with $\tilde\beta(|\xi|_{g(x)})\neq 0$.  Indeed, in \S\ref{S:RiemLog} we shall see that in coordinates we have $\expyno = (y-x) + \mathcal{O}(|y-x|^2)$. Hence this behaves similarly to the usual Gaussian $e^{\frac ih (\xi\cdot(y-x)+\frac i2|y-x|^2)}$ which are ``centered'' at $(x,\xi)\in T^*M$ in that they are concentrated in $h^{\frac 12}$-neighborhoods of $x,\xi$ in space and momentum respectively (the latter defined with respect to the semiclassical Fourier transform).  In particular, $\Ssc^*$ acts on functions $G(x,\xi) \in L^2(T^*M)$, via 
\begin{equation*}
	(\Ssc^*G)(y) = 2^{-\frac d2} (\pi h)^{-\frac{3d}{4}}\int_{T^* M} e^{\frac ih(\xi(\expyno) + \frac i2 d^2(x,y))} \frac{ \psi(x,y) \tilde\beta\big(|\xi|_{g(x)}\big) }{\detexp^{\frac 14}(g_{jk}(x))}G(x,\xi) \,dx d\xi,
\end{equation*}
where $dx d\xi$ is Liouville measure on $T^*M$.  The integral \eqref{Uintegral} thus allows us to construct a parametrix for $\Usc_t \circ \Ssc^*$ by approximating the $\Usc_t$ evolution of Gaussians in \eqref{gaussianinvariant}, then taking the corresponding integral superposition of them.

The construction of approximate solutions to hyperbolic PDE such as \eqref{halfwaveeqn} with Gaussian initial data has a rich history, typically under the name of ``Gaussian beams''.  We cannot mention all the work done here, though Babi\v{c} and Buldyrev \cite{BabicBuldyrevBook} and Ralston \cite{RalstonGaussianBeamsSurvey} are significant milestones in consolidating the theory.  These references (and others) show how to construct highly accurate approximations ($\mathcal{O}(h^N)$ on $L^2$ for any $N$) over bounded time intervals which are independent of in the frequency parameter.  For semiclassical Schr\"odinger equations, there is also the analogous work of Hagedorn \cite{HagedornI}, \cite{HagedornIII}, \cite{HagedornIV}, \cite{HagedornRaisingLowering}.  Work of Combescure and Robert \cite{CombescureRobertWavePackets} (see also their book \cite{CombescureRobertBook}) and also Hagedorn and Joye \cite{HagedornJoyeCMP}, \cite{HagedornJoyeChapter}, \cite{HagedornJoye} consider these approximations over the ``Ehrenfest'' time intervals $T = T(h)$ considered here.  

The references here take $\RR^\d$ as the configuration space instead of Riemannian manifolds $(M,g)$ considered here.  One of our objectives is to show that the $\Usc_t$-evolution of the Gaussians \eqref{gaussianinvariant} can be approximated by transformed Gaussians with phases that are also coordinate invariant.  This is done in the same spirit as in the Euclidean case, instead considering Gaussians with phases
\begin{equation}\label{phaseintro}
	\xi_t(\expy) + \frac i2\omega_t(\expy,\expy),
\end{equation}
where $\omega$ is a $t$-dependent symmetric $(0,2)$ tensor on the complexification of $T_{x_t}M$.   Here $(x_t,\xi_t)$ are integral curves of the Hamiltonian flow generated by the cometric $|\xi|_{g}$ with $(x_t,\xi_t)|_{t=0} = (x,\xi)$.  By taking $-i\omega_0 $ to be the Riemannian metric at $t=0$, the second term is just $\frac i2 |\expyno |_g^2 =\frac i2 d^2(x,y)$, recovering the phase of the original Gaussian.  After raising an index to form a $(1,1)$ tensor, $\omega_t$ solves a \emph{complex Riccati equation} (cf. \S\ref{S:riccati}) analogous to the classical Gaussian beams constructions.  Previous work on coordinate-free interpretations of the Gaussian beam phase is due to Katchalov and Lassas \cite{KatchalovLassas} (see also \cite{KatchalovKurylevLassasBook}) as well as Dahl \cite{DahlGeometricInterpretation}, \cite{DahlLeadingTerm} (see \S\ref{SSS:classicalbeams} below for more).  Related work on wave packet propagation in a Riemannian manifold is due to Paul and Uribe \cite{PaulUribePointwise}; Guillemin, Uribe, and Wang \cite{GuilleminUribeWang}; and Eswarathasan and Nonnenmacher \cite{EswarathasanNonnenmacher}.  However, none of the works referenced here appear to use the phase \eqref{phaseintro}.

It is also noteworthy that our parametrix $\Vsc_t$ is reminiscent of those in the work of Laptev, Safarov, Vassiliev \cite{LaptevSafarovVassiliev} and Capoferri, Levitin, and Vassiliev \cite{CapoferriLevitinVassiliev} which use FIOs with complex phase.  Indeed, both works seek to use complex phases to formulate parametrices given by a single, coordinate invariant integral.  However, their work is not rooted in Gaussian wave packets in the same way ours is, leading to differences in the oscillatory integrals.  

\begin{remark}\label{R:baderror}
The aforementioned works  \cite{CombescureRobertWavePackets}, \cite{HagedornJoyeCMP}, \cite{HagedornJoyeChapter}, \cite{HagedornJoye}, suggest that we should be able to obtain Gaussian beam solutions which are $\mathcal{O}(h^N)$ for any fixed $N$ provided $\mu(t) \leq h^{\frac 16-\veps}$ for some $\veps>0$. This would yield error estimates for $\Vsc_t$ that are much stronger than \eqref{Vproperties} (even with the somewhat crude approach to error bounds in \S\ref{SS:errorestimates}).  However, since these works consider propagation in $\RR^d$, their arguments crucially use that the phase of the Gaussian is quadratic. In contrast, the use of $\expy$ in \eqref{phaseintro} means that our phases cannot be treated as quadratic in coordinate systems: even in a normal coordinate system, the $t$ derivatives do not behave in the same way as the experience in $\RR^d$ suggests as per \eqref{firsttderiv} in Lemma \ref{L:firsttderiv} below.  This presents a technical hurdle in reiterating their arguments, which we do not try to resolve here.  Improving the error bound in \eqref{Vproperties} would translate to notable improvements in the bound in Theorem \ref{T:mainthmpoly}, but the implications for Theorem \ref{T:mainthmexp} do not appear to be significant.
\end{remark}

A second aspect of \eqref{Uintegral} is that it involves the conjugation $ \Ssc \circ \Usc_t \circ \Ssc^*$.  In the words of Tataru \cite{TataruParametricesSchrod} its kernel is the \emph{phase space kernel} of the evolution $\Usc_t$.  This can also be viewed as the function of $(z,\zeta,t,x,\xi)$ determined by the matrix elements 
\begin{equation*}
\int_M (\Usc_t g_{x,\xi})(y) \overline{ g_{z,\zeta}(y)}\,dv_g(y) ,
\end{equation*}
up to smooth cutoffs determined by $\tilde\beta$.  The influential work of Cord\'oba and Fefferman \cite{CordobaFeffermanWP} studied the virtues of this approach, examining how operators such as $\Usc_t$ behave under the conjugation.  In his work on rough metrics, Tataru \cite{TataruPhaseSpaceTransforms} made similar observations that the phase space kernel of $ \Ssc \circ \Usc_t \circ \Ssc^*$, denoted $\K(t, z,\zeta;x,\xi) : T^*M \times T^*M \to \CC$, satisfies the following bounds for $t$ in  \emph{uniformly bounded, $h$-independent}, time intervals $I$
\begin{equation}\label{shorttimepskernel}
	|\K(t, z,\zeta;x,\xi)| \lesssim_{N,I} h^{-\d}\big( 1 + h^{-\frac 12}|z-x_t| + h^{-\frac 12}|\zeta-\xi_t| \big)^{-N},
\end{equation}
where again $(x_t,\xi_t)$ is the image of $(x,\xi)$ under the bicharacteristic flow.  One of the features of the present work is to characterize this bound over the intervals $[-T(h),T(h)]$ considered above.  Indeed, over these time scales, the expansion and contraction effects of the flow significantly alters the estimates \eqref{shorttimepskernel}, but by considering the singular values of the differential, we shall be able to give a precise description of this in \S\ref{S:matrixelements} below.   This in turn relies an a precise analysis of the tensor $\omega_t$ which is presented in \S\ref{S:riccati}, in particular \S\ref{SS:SiegelPlaneDisk}.

Moreover, as we shall see in \S\ref{S:wpt}, the phase space kernel of the PDO $Q_{h,\tilde\gamma}$ is concentrated in a region where $|(x',\xi')| \lesssim h^{\frac 12}$, $|(x_d,\xi_d)| \ll 1$ in Fermi coordinates, using that the Gaussians decay on the same scale as the symbol of $Q_{h,\tilde\gamma}$.  This essentially allows us to limit attention to Gaussians centered at points $(x,\xi)$ emanating from this region, and also times $t$ for which $(x_t,\xi_t)$ returns to the same region.  This will play a key role in proving \eqref{tdepreduction}.  

All together, our approach makes full use of the microlocal Kakeya-Nikodym norms as opposed to their spatial counterparts such as \eqref{KNgain}.  The use of phase space kernels allow us to see how the expansion and contraction effects of the singular values of the differential of the geodesic flow lead to the decay estimate in Theorem \ref{T:tdepreduction}.  In contrast, the previous work \cite{BlairSoggeToponogov} which developed \eqref{KNgain} used expansion effects of the geodesic flow in the spatial domain only\footnote{Put another way, viewing the differential as a block matrix as in \eqref{blocksymplectic}, \cite{BlairSoggeToponogov} and its predecessors used that nonpositive curvature means that the growth of the determinant of the ``$B$" block of this matrix can be quantified, implying decay in the leading coefficient of the Hadamard parametrix, an inverse power of this determinant.  The present work instead seeks to appreciate the structure of this differential in full.}.

\subsection{Outline of the work}\label{SS:outline} Before we can get to the heart of our work, we have several preliminary matters to handle. In \S\ref{S:RiemLog}, we discuss the Riemannian logarithm $\expyno$ in detail as we shall need it both to define the wave packet transform and to be able to discuss \eqref{phaseintro}.  In \S\ref{S:Sasaki}, we introduce the Sasaki metric on the manifold $TM$, which turns the canonical projection from $TM$ to $M$ into a Riemannian submersion.  This section also discusses the differential of the geodesic flow on $TM$, and defines its singular values.  In \S\ref{S:examples}, we take a short pause from preliminary matters to consider examples of geometries where the hypotheses of Theorems \ref{T:mainthmexp} and \ref{T:mainthmpoly}; the preceding section will have set the stage for discussing singular values in detail. In \S\ref{S:riccati}, we introduce and solve the complex Riccati equation for the tensor $\omega_t$ in \eqref{phaseintro}.  In particular, \S\ref{SS:SiegelPlaneDisk} will prepare us to study the phase space kernel of our parametrix while \S\ref{SS:RiccatiReg} establishes the needed bounds on the tensor $\omega_t$.  One last technical matter to handle is introduce some needed stationary phase arguments for complex phases, which is done in \S\ref{S:statphase}.  

The heart of the work begins in \S\ref{S:wpt}, where we examine the wave packet transform, in particular proving Proposition \ref{P:invwpt1}. In \S\ref{S:WPansatz}, we construct the parametrix $\Vsc_t$ in Theorem \ref{T:Vtthm} and consider its error bounds.  The conclusion of the proof, namely Theorem \ref{T:tdepreduction} and some remaining error analysis, is finally handled in \S\ref{S:matrixelements}, which entails a thorough treatment of the phase space kernel of our approximation to $ \Ssc \circ \Usc_t \circ \Ssc^*$.

\section{The Riemannian Logarithm}\label{S:RiemLog}
Given a point $(p,V) \in TM$, expressed in a local trivialization, the map $(p,V) \mapsto \exp_p V$ yields a diffeomorphism from a neighborhood of the zero section in $T^*M$ to a neighborhood of the diagonal in $M \times M$ (cf. \cite[p.33]{SakaiRiemannianGeometry}). Wherever an inverse is defined, we denote it by $(x,y) \mapsto\expyno$, so that $\expyno$ is the tangent vector in $T_xM$ whose image under $\exp_x$ is $y$.  

Our wave packet ansatz entails a somewhat tricky calculation involving $\expyno$.  We need to calculate covariant derivatives of the vector field $x \to \expyno$ at a point $x_0$, then examine the Taylor expansion of $\nabla\expyno|_{x_0}$ with respect to $y$.  Choosing coordinates is a challenge as we are pulled between taking normal coordinates centered at $y$ and at $x_0$.  All the while the calculations in our ansatz in \S\ref{S:WPansatz} below benefits from the latter.  We benefit from a classical approach to normal coordinates we learned of in \cite[\S17-18]{EisenhartBook}.  The expansion \eqref{logexpn} below is already known, but we do prove it here as the work involved sets the stage for \eqref{covariantlogexpn}, which we have not seen elsewhere.

\begin{proposition}\label{P:covariantlogexpn}
	Let $x,y$ denote variables in a convex coordinate chart close enough so that $\expyno =V^j(x,y)\frac{\prtl }{\prtl x_j}|_x \in T_xM$ is defined.
	\begin{enumerate}
		\item The quadratic Taylor approximation of $V^j(x,\cdot)$ takes the form
		\begin{equation}\label{logexpn}
			V^j(x,y) = (y-x)^j + \frac 12 \Gamma_{k\ell}^j(x)(y-x)^k (y-x)^\ell + \mathcal{O}(|y-x|^3).
		\end{equation}
		\item Suppose further that all Christoffel symbols vanish at some $x_0$ in the chart: $\Gamma^k_{ij}(x_0) =0$, but not necessarily elsewhere.  The total covariant derivative\footnote{In this case, the total covariant derivative defines a $(1,1)$ tensor, cf. \cite[Proposition 4.17]{LeeRiemannian}.} of the vector field $x\mapsto \expyno$ at $x_0$ denoted as $\nabla \expyno|_{x_0}=V^j_{;m}(y)\frac{\prtl}{\prtl x_j}\Big|_{x_0}dx_m|_{x_0}$ admits a quadratic approximation
			 \begin{equation}\label{covariantlogexpn}
			V^j_{;m}(y) = - \delta_m^j + \frac{1}{6}\big(R_{mk\ell}^j (x_0) +R_{m\ell k}^j (x_0)\big)y_ky_\ell+ \mathcal{O}(|y|^3).
		\end{equation}
		where $R_{mk\ell}^j $ denote the components of the Riemann curvature tensor with respect to the coordinate frame.  Consequently, if $W \in T_{x_0}M$,
		\begin{equation}\label{covariantlogexpndir}
			\nabla_W \expyno|_{x_0} = \big( -W^j + \frac{1}{6}\big(R_{mk\ell}^j (x_0) +R_{m\ell k}^j (x_0)\big)y_ky_\ell  W^m + \mathcal{O}(|y|^3)\big)\frac{\prtl}{\prtl x_j}\Big|_{x_0}
		\end{equation}

	\end{enumerate}
\end{proposition}

\begin{proof}
	Let $x(t)$ parameterize a geodesic in the coordinates such that $x(0) = x$ and $\frac{dx}{dt}(0)=V^j \prtl_{x_j}|_x$.
	Recall that
	\begin{equation*}
		\frac{d^2 x_j}{dt^2} (t) = -\Gamma_{k\ell}^j(x(t)) \frac{dx_k}{dt}(t)\frac{dx_\ell}{dt}(t).
	\end{equation*}
	Differentiating both sides yields and evaluating at $t=0$ yields
	\begin{equation}\label{geodthird}
		\frac{d^3 x_j}{dt^3} (0)= 
		\big(2\Gamma_{kn}^j(x)\Gamma_{\ell m}^n(x(t)) -\prtl_m \Gamma_{k\ell}^j(x(t)) \big) V^kV^\ell V^m.
	\end{equation}
	To simplify the notation, we introduce $S_{k\ell m}^j(x)$ defined by
	\begin{equation*}
		S_{k\ell m}^j := \frac 23 \big(\Gamma_{kn}^j(x)\Gamma_{\ell m}^n(x) + \Gamma_{\ell n}^j(x)\Gamma_{ m k}^n(x) + \Gamma_{mn}^j(x)\Gamma_{k\ell}^n(x)\big)-\frac 13\big(\prtl_m \Gamma_{k\ell}^j(x) +\prtl_k \Gamma_{\ell m}^j(x) +\prtl_\ell \Gamma_{mk}^j(x) \big) .
	\end{equation*}
	Since $\Gamma_{k\ell}^j = \Gamma_{\ell k}^j$ this gives the symmetrization of the expression in parentheses in \eqref{geodthird} with respect to all 3 lower indices.  Consequently, near $V=0$, the $j$-th component of $\exp_x(V)$ is
	\begin{equation*}
		x_j(1) = x_j + V^j  - \frac 12 \Gamma_{k\ell}^j(x) V^k V^\ell + \frac 16 S_{k\ell m}^j (x)V^kV^\ell V^m + \mathcal{O}(|V|^4).
	\end{equation*}
	
	We now define $F_1, \dots, F_\d$ so that $F_j +y_j$ is the $j$-th coordinate of $\exp_x(V)$.  In other words, with the same quartic term,
	\begin{equation*}
		F_j(x,y,V) := -y_j + x_j + V^j  - \frac 12 \Gamma_{k\ell}^j(x) V^k V^\ell + \frac 16 S_{k\ell m}^j (x)V^kV^\ell V^m + \mathcal{O}(|V|^4) 
	\end{equation*}
	so that locally $y = \exp_x(V)$ if and only if $F_j(x,y,V)=0$ for all $j=1,\dots, \d$.  Since $\frac{\prtl F_j}{\prtl V^k}|_{V=0} = \delta^j_k$, the implicit function theorem implies that near points where $x=y$ and $V=0$, the zero set of $F$ is locally parameterized as the graph of $(x,y)\mapsto V(x,y)$ where $V(x,x)=0$ and $V^j(x,y)$ is the $j$-th component of $\expyno$.  
	
	The approximations \eqref{logexpn}, \eqref{covariantlogexpn} result from the usual method for calculating partials of $V$ along the diagonal $x=y$ and in the latter case $x=y=x_0$. The chain rule gives
	\begin{equation}\label{chained}
		0=\frac{\prtl F_j}{\prtl x_m } + \frac{\prtl F_j}{\prtl V^a } \frac{\prtl V^a}{\prtl x_m} \quad \text{ and } \quad 
		0=\frac{\prtl F_j}{\prtl y_k } + \frac{\prtl F_j}{\prtl V^a } \frac{\prtl V^a}{\prtl y_k} . 
	\end{equation}
	Restricting these to $x=y$ (so that $V=0$) we obtain  
	\begin{equation*}
		\frac{\prtl V^j}{\prtl x_m}\Big|_{y=x} = \delta^j_m \quad \text{ and } \quad \frac{\prtl V^j}{\prtl y_k}\Big|_{y=x} =- \delta^j_k
	\end{equation*}
	
	Calculating the second derivatives of $V$ is naturally more involved, but we can simplify it by noting that any partial derivative of $F_j$ of order 2 or more involving at least one $y$ variable vanishes.  Hence differentiating the first identity in \eqref{chained} with respect to $y_k$ yields
	\begin{equation}\label{chained2}
		0=\frac{\prtl^2 F_j}{\prtl V^a \prtl x_m }\frac{\prtl V^a}{\prtl y_k} + \frac{\prtl^2 F_j}{\prtl V^a \prtl V^b }\frac{\prtl V^b}{\prtl y_k} \frac{\prtl V^a}{\prtl x_m} + \frac{\prtl F_j}{\prtl V^a } \frac{\prtl^2 V^a}{\prtl y_k \prtl x_m } .
	\end{equation}
	Along $y=x$ the first term is seen to vanish since $ \frac{\prtl^2 F_j}{\prtl V^a \prtl x_m }|_{V=0}=0$ while the second gives the Christoffel symbols.  We can take a similar approach differentiating the second identity in \eqref{chained} in $y_\ell$ and we now have
	\begin{equation}\label{vprtl}
		\frac{\prtl^2 V^j}{\prtl x_m \prtl y_k}\Big|_{y=x} = -\Gamma^j_{km}(x)
		\quad \text{ and } \quad
		\frac{\prtl^2 V^j}{\prtl y_\ell \prtl y_k}\Big|_{y=x} = \Gamma^j_{k\ell}(x)
	\end{equation} 
	The latter identity completes the derivation of \eqref{logexpn}.
	
	The derivation of \eqref{covariantlogexpn} is simplified by the hypothesis that the Christoffel symbols vanish at $x_0$, meaning that  $V^j_{;m}(y) = \frac{\prtl V^j}{\prtl x_m}(y)$ and the identities in \eqref{vprtl} vanish at $y=x=x_0$.  Differentiating \eqref{chained2} with respect to $y_\ell$ in the same way as before and incorporating the observations above gives
	\begin{equation*}
		\frac{\prtl^3 V^j}{\prtl y_\ell \prtl y_k \prtl x_m }(x_0,x_0)= \prtl_m \Gamma_{k\ell}^j(x_0) + S_{k\ell m}^j(x_0) = \frac 13 \big(R_{mk\ell}^j(x_0) +R_{m\ell k}^j(x_0) \big)
	\end{equation*}
	where the last equation follows from the typical formula for the curvature tensor in coordinates (e.g. \cite[Proposition 7.4]{LeeRiemannian}).
\end{proof}

\section{Sasaki Metric and Differential of the Geodesic Flow}\label{S:Sasaki}
\subsection{The Sasaki Metric and Symplectic Structure on the Tangent Bundle}\label{SS:SasakiTM}
 In this section, we review the Sasaki metric $\bg$ on $TM$ drawing on material from \cite[\S9.1]{BlairRiemannianBook}, \cite[\S5.8-9]{BurnsGideaBook}, \cite[Ch.1]{PaternainBook}, \cite[\S II.4]{SakaiRiemannianGeometry}. This metric defines $(TM, \bg)$ as a Riemannian manifold such that the canonical projection $\pi :TM \to M$ is a Riemannian submersion.  It gives us a framework for understanding the differential of the geodesic flow on $TM$ in terms of Jacobi fields.

In what follows, $\theta = (p,V) \in TM$ denotes an arbitrary point expressed in terms of a local trivialization, and when it is desirable to take a coordinate system on $M$ nearby $p$, we shall denote these by $(x_1,\dots,x_\d)$ with $(x_1, \dots, x_d, X^1, \dots, X^d)$ denoting the induced coordinates on $TM$.  The \emph{vertical subspace} of $T_\theta TM$ is defined to be $V_\theta = \ker(d\pi|_\theta)$.  Given a vector field $Y$ on $M$, we define its \emph{vertical lift} $Y^v$ to $TM$ as the unique vector field defined by $Y^v \omega = \omega(Y)\circ \pi$ for any 1-form $\omega$ on $M$, viewed as a function on $TM$ on the left hand side of the identity\footnote{In other words, if $\omega = \omega_j (x)dx_j$ in coordinates, then the function on $TM$ is $\omega_j(x) X^j$, cf. \cite[Ch. 1 (2.4)]{YanoIshihara}.}.  Thus if $Y = Y^j(x) \frac{\prtl }{\prtl x_j}$ then $Y^v = Y^j(x) \frac{\prtl }{\prtl X^j}$. Since $V_\theta=$ span$\{\frac{\prtl }{\prtl X^1},\dots, \frac{\prtl }{\prtl X^d}\}$, the map $\iota_v: T_pM \to T_\theta TM$ defined by  $\iota_v (Y) = Y^v$ is an isomorphism from $T_pM $ to the vertical subspace in $T_\theta TM$.

The Levi-Civita connection $\nabla$ on $(M,g)$ means that for each $\theta \in TTM$, there is a well-defined \emph{horizontal subspace} $H_\theta$ complementary to each vertical subspace in $T_\theta TM$.  First define the \emph{horizontal lift} of a vector field $Y$ on $M$ as the unique vector field $Y^h$ on $TM$ such that $Y^h \omega = \nabla_Y \omega$ for any differential 1-form $\omega$ on $M$, again viewing $\omega$ as a function on $TM$ on the left hand side.  In coordinates, $Y^h = Y^j(x) \frac{\prtl}{\prtl x_j} - Y^j (x) X^k \Gamma_{jk}^\ell(x) \frac{\prtl}{\prtl X^\ell}$.  The horizontal subspace $H_\theta$ is now defined to be the $d$-dimensional subspace spanned by the horizontal lifts of vector fields on $M$, that is, $H_\theta$ is the image of $T_pM$ under  $\iota_h: T_pM \to T_\theta TM$ defined by $\iota_h(Y) = Y^h$.

The above shows that both the horizontal and vertical subspaces are locally defined as the span of smooth vector fields and hence define distributions on $TM$.    Geodesics $\gamma:\RR \to M$ are characterized by the property that the tangents to the curve $t \mapsto (\gamma(t), \dot\gamma(t)) \in TM$ always lie in the horizontal distribution.  More generally, given a smooth curve $(x(t), X(t)) \in TM$, then $X(t) = X^j(t)\frac{\prtl}{\prtl x_j}\big|_{x(t)}$ is parallel along $x(t)$ $(\nabla_t X = 0)$ if and only if $(\dot x(t),\dot X(t))$ lies in the horizontal distribution.

Any vector field on $TTM$ in coordinates $\Xi = Y^j \frac{\prtl}{\prtl x_j}+ Z^j \frac{\prtl}{\prtl X^j}$ decomposes in $T_\theta TM$ as 
\begin{equation}\label{horvertdecomp}
	 \Big( Y^j(\theta) \frac{\prtl}{\prtl x_j}\Big|_\theta - Y^j (\theta) X^k \Gamma_{jk}^\ell(p) \frac{\prtl}{\prtl X^\ell} \Big|_\theta\Big) + 
	 \big( Z^\ell(\theta) + Y^j (\theta) X^k \Gamma_{jk}^\ell(p)\big)  \frac{\prtl}{\prtl X^\ell}\Big|_\theta 
	 \in H_\theta \oplus V_\theta
\end{equation}
Equivalently, the horizontal component  is given by $(d\pi|_\theta( \Xi))^h$ and the vertical component is given by $\Xi - (d\pi|_\theta (\Xi))^h$.  The \emph{connection map} $K_\theta: T_\theta TM \to T_pM$ is defined in the following way: for any vector field $\Xi$, $K_\theta(\Xi) $ is the image in $T_p M$ of the vertical component of $\Xi$ under $\iota_v^{-1}$, where $\iota_v$ is the vertical isomorphism defined above.  In coordinates, 
\begin{equation}\label{connectionmap}
	K_\theta(\Xi) = \big( Z^\ell(\theta) + Y^j (\theta) X^k \Gamma_{jk}^\ell(x)\big)  \frac{\prtl}{\prtl x_\ell}\Big|_x .
\end{equation}
The map $K_\theta$ can be equivalently characterized by the following: take any curve $(x(t), X(t)) \in TTM$, defined on an open interval in $\RR$ containing 0, such that $(\dot x(0), \dot X(0)) = \Xi$, we then have $K_\theta (\Xi) = \nabla_t X|_{t=0} = \nabla_{\dot x(0)} X$ where $X(t) = X^j(t)\frac{\prtl}{\prtl x_j}\big|_{x(t)}$ is treated as a vector field along $x(t)$.  

The \emph{Sasaki metric} $\bg$ on $TM$, introduced by Sasaki in \cite{SasakiTangentBundlesI}, is defined as 
\begin{equation*}
	\langle \Xi_\theta, \Upsilon_\theta \rangle_{\bg(\theta)} = \langle d\pi|_\theta (\Xi_\theta), d\pi|_\theta (\Upsilon_\theta)\rangle_{g(\pi(\theta))} +
	 \langle K_\theta (\Xi_\theta) ,K_\theta (\Upsilon_\theta)\rangle_{g(\pi(\theta))} \quad \text{ for } \Xi_\theta, \Upsilon_\theta \in T_\theta TM.
\end{equation*}
Equivalently, in terms of horizontal and vertical lifts, we have
\begin{equation*}
	\langle X^h + \tilde X^v, Y^h + \tilde Y^v\rangle_{\bg(\theta)} = \langle X,Y \rangle_{g(p)} + \langle \tilde X, \tilde Y\rangle_{g(p)} , \quad p = \pi(\theta), \;X,\tilde X, Y, \tilde Y \in T_p M.
\end{equation*}
In coordinates, we have \cite[(3.2)]{SasakiTangentBundlesI}
\begin{equation}\label{sasakicoord}
	\bar g = g_{jk}(x) dx_j dx_k + g_{jk}(x)\pmb{\delta} X^j \pmb{\delta} X^k, \qquad \pmb{\delta} X^j := dX^j + \Gamma_{\ell m}^j(x) X^\ell dx_m,
\end{equation}
where $\pmb{\delta} X^j$ is the \emph{covariant differential}.  We thus obtain the simpler expression if the Christoffel symbols vanish at a point $x_0$:
\begin{equation*}
	\Gamma_{\ell m}^j(x_0) =0\quad \forall j,\ell,m \implies \bar g|_{x_0} =  g_{jk}(x_0) dx_j dx_k + g_{jk}(x_0)d X^j d X^k.
\end{equation*}

There is an almost complex structure $\mathbf{J}_\theta$ on  $T_\theta TM$ (a $(1,1)$ tensor satisfying $\mathbf{J}_\theta^2 = -I$) introduced by Dombrowski \cite{DombrowskiTgt} characterized by 
\begin{equation}\label{acstructure}
	\mathbf{J}_\theta X^h|_\theta = X^v|_\theta \quad \text{ and } \quad \mathbf{J}_\theta X^v|_\theta = -X^h|_\theta,
\end{equation}
for any vector field $X$ defined near $\pi(\theta)$.  Equivalently, $\mathbf{J}_\theta$ is characterized by $d\pi|_\theta \circ \mathbf{J}_\theta = -K_\theta $ and $K_\theta \circ \mathbf{J}_\theta = d\pi|_\theta$. Consequently, 
\begin{equation}\label{symptgtbundle}
	\nu_\theta(\Xi_\theta, \Upsilon_\theta) = \langle \mathbf{J}_\theta \Xi_\theta, \Upsilon_\theta\rangle_{\bg(\theta)} 
	=  \langle d\pi|_\theta \Xi_\theta, K_\theta(\Upsilon_\theta)\rangle_{g(\pi(\theta))} - \langle K_\theta(\Xi_\theta), d\pi|_\theta \Upsilon_\theta\rangle_{g(\pi(\theta))} 
\end{equation}
defines a symplectic structure on $TM$.  In fact, a routine calculation shows this is the pull back of the canonical 2-form on $T^*M$ under the musical isomorphism from $TM$ to $T^*M$, $(p,V) \mapsto (p,V^\flat)$. 

\subsection{Jacobi Fields and the Differential of the Geodesic Flow}\label{SS:diffgeodflow}
Let $\kappa_t: TM \to TM$ denote the geodesic flow on $TM$, that is, $\kappa_t(p,V) = (\gamma_{p,V}(t), \dot\gamma_{p,V}(t))$, where $\gamma_{p,V}(t)$ is the geodesic determined by $(\gamma_{p,V}(0), \dot\gamma_{p,V}(0))=(p,V)$.  Equivalently, it is the Hamiltonian flow on $TM$ determined by the function on $TM$: $(p,V) \mapsto \frac 12 |V|_{g(p)}^2$. Hence $\kappa_t$ is a symplectomorphism with respect to $\nu$.  In $(x,X)$ coordinates, the Hamiltonian vector field is given by 
\begin{equation*}
X^j \frac{\prtl}{\prtl x_j} - \Gamma_{jk}^\ell (x) X^jX^k\frac{\prtl}{\prtl X^\ell} .
\end{equation*}
Moreover, $\kappa_t$ preserves the level sets of  $\frac 12|V|_{g(p)}^2$, namely the fiber bundles $S^{(r)}M$ defined in \eqref{tgtspheresdef}. The homogeneity $ \frac 12 |cV|_{g(p)}^2 = \frac{c^2}{2} |V|_{g(p)}^2$ also gives the rescaling property for $c>0$
\begin{equation}\label{geodrescale}
	\big(\gamma_{p,cV}(t), \dot\gamma_{p,cV}(t)\big) = \big(\gamma_{p,V}(ct), c\dot\gamma_{p,V}(ct)\big).
\end{equation}
Equivalently, if we define the rescaling map $\tau_c(p,V) = (p,cV)$ for $c>0$, this can be expressed as
\begin{equation}\label{geodrescalemap}
	(\kappa_t \circ \tau_c) (p,V) = (\tau_c \circ \kappa_{ct}) (p,V).
\end{equation}

The differential of $\kappa_t$ can be expressed in terms of Jacobi fields.  Suppose $\theta \in TM$ and $\gamma: \RR \to M$ is the geodesic such that $(\gamma(0),\dot{\gamma}(0))=\theta$. Let $\R_{\dot\gamma(t)}$ denote the endomorphism on $T_{\gamma(t)}M$ formed by contracting the Riemann curvature tensor along $\dot\gamma$: 
\begin{equation}\label{tidalforce}
\R_{\dot\gamma(t)}(V) = R(V, \dot\gamma(t))\dot\gamma(t) .
\end{equation}
Recall that a Jacobi field is any solution to the Jacobi equation
\begin{equation}\label{JacobiEqn}
	\covt^2 J(t) +\R_{\dot\gamma(t)}(J(t))= \covt^2 J(t) + R(J(t),\dot\gamma(t))\dot\gamma (t) =0.
\end{equation}
Given $\Xi \in T_\theta TM$, let $J(t)$ be the Jacobi field along $\gamma(t)$ satisfying $J(0) = d\pi|_\theta (\Xi) $, $\covt J(0) = K_\theta (\Xi)$.  The differential of $\kappa_t$ can be calculated as 
\begin{equation}\label{diffgeodflow}
	d\kappa_t|_\theta (\Xi) = (J(t)^h, (\covt J(t))^v) \in H_{\kappa_t(\theta)} \oplus V_{\kappa_t(\theta)} ;
\end{equation}
proofs can be found in \cite[Prop. 5.9.2]{BurnsGideaBook}, \cite[Lemma 1.40]{PaternainBook}, \cite[Ch. II-Lemma 4.3]{SakaiRiemannianGeometry}.

The symmetries of the Riemann curvature tensor imply that if $J(t)$ is a Jacobi field, then $\langle \covt J(t),\dot\gamma(t)\rangle_{g}$ is constant in $t$ so that $\langle J(t),\dot\gamma(t)\rangle_{g}$ is an affine function of $t$.  The field $J(t)$ is said to be \emph{tangential} if $J(t) = (at+b) \dot\gamma(t)$ for all $t$, where $a,b \in \RR$, or equivalently if its projection onto the normal space of $\gamma$ vanishes. By uniqueness of solutions to \eqref{JacobiEqn}, a Jacobi field is tangential if and only if $J(t_0)$, $\covt J(t_0)$ are both scalar multiples of $\dot\gamma(t_0)$ for some $t_0$.  Moreover, by \eqref{diffgeodflow},
\begin{equation}\label{difftangential}
	d\kappa_t|_\theta (\dot\gamma(t)^h) = \big(\dot\gamma(t), 0) \in H_{\kappa_t(\theta)} \oplus V_{\kappa_t(\theta)}  , 
	\qquad d\kappa_t|_\theta (\dot\gamma(t)^v) = (t\dot\gamma(t),\dot\gamma(t)) \in H_{\kappa_t(\theta)} \oplus V_{\kappa_t(\theta)} .
\end{equation}

On the other hand, a Jacobi field $J$ is said to be \emph{normal} if it satisfies $J(t) \perp \dot\gamma(t)$ for all $t$.  A normal Jacobi field also satisfies $\covt J(t) \perp \dot\gamma(t)$ since otherwise it would have a nontrivial tangential component. Similarly, if  $J(t_0)\perp \dot\gamma(t_0)$ and $\covt J(t_0) \perp \dot\gamma(t_0)$ for some $t_0 \in \RR$, then $J$ is normal.  

We now review the standard method for expressing solutions to the Jacobi equation as a second order ODE.  Begin with any parallel orthonormal frame along $\gamma(t)$ denoted as 
\begin{equation}\label{parallelONframe}
	E_1(t),\dots, E_\d(t) \quad \text{ such that } \quad  E_\d(t) = \frac{\dot\gamma(t)}{| \dot\gamma(t)|_{g}} .
\end{equation}
We now write $J(t) = J^i(t) E_i(t)$ so that the coefficients $(J^1(t),\dots,J^{\d}(t))$ satisfy the linear, second-order ODE system
\begin{equation*}
	\begin{split}
		&\ddot{J}^i(t) + R_{k}^i(t) J^k(t) =0 , \\
		& R_{k}^i(t):= \langle R(E_k(t),\dot\gamma(t))\dot\gamma(t),E_i(t)\rangle_{g(\gamma(t))} = Rm\big(E_k(t),\dot\gamma(t), \dot\gamma(t),E_i(t)\big) 
	\end{split}
\end{equation*}
Introducing $L^i(t) = \dot J^i(t) = \langle\covt J(t), E_i(t)\rangle_g$, the system converts from second to first-order 
\begin{equation}\label{jacobifirstorder}
	\dot J^i (t) = L^i(t), \qquad \dot L^i(t) = -R_{k}^i(t) J^k(t), \qquad i= 1, \dots, \d.
\end{equation}
Since $M$ is compact, the following is finite
\begin{equation}\label{lyapunovdef}
	 \tilde\Lambda = \sup\big\{ \|\R_Y\| : p \in M,\; Y \in S_p^{(1)} M \big\}.
\end{equation}
Here $S_p^{(r)}M$ is the fiber of $S^{(r)} M$ over $p$ and $\|\R_Y\|$ is the norm of the endomorphism $\R_Y$ with respect to the Riemannian metric, cf. \eqref{tidalforce}. A routine application of Gronwall's inequality now implies the following standard result.
\begin{proposition}\label{P:gronwall}
	Let $\theta \in S^{(r)}M$, $\Xi \in T_\theta TM$.  Then with $\tilde\Lambda$ as in \eqref{lyapunovdef},  for all $t \in \RR$ it holds that
	\begin{equation*}
		 \big\langle d\kappa_t|_\theta (\Xi), d\kappa_t|_\theta (\Xi) \big\rangle_{\bg(\kappa_t(\theta))}  \leq 	\big\langle \Xi, \Xi \big\rangle_{\bg(\theta)} e^{(1+r^2\tilde\Lambda) |t|}.
	\end{equation*}
\end{proposition}

\subsection{Singular value decompositions}\label{SS:singval}
Let $\gamma: \RR \to M$ be a geodesic of nontrivial speed and $\{E_j(t)\}_{j=1}^\d$ an orthonormal frame as in \eqref{parallelONframe}.  The map 
\begin{equation*}
\RR^{\d-1} \times \RR \in(y_1,\dots,y_{\d-1}, t) \mapsto \exp_{\gamma(t)}(y_1E_1(t)+ \dots+ y_{d-1}E_{d-1}(t)),
\end{equation*}
defines a local diffeomorphism from a neighborhood of $\{0\} \times \RR$ to a tubular neighborhood of the image of $\gamma$ in $M$.  This is a consequence of the inverse function theorem since the differential of the map at any point $(0,t_0)$ maps $\frac{\prtl}{\prtl y_j}|_{(0,t_0)} \mapsto E_j(t_0)$, $\frac{\prtl }{\prtl t}|_{(0,t_0)} \mapsto |\dot\gamma(t)|_gE_\d(t_0)$.  The local coordinates furnished are known as \emph{Fermi} coordinates.  Along $\gamma$, the metric tensor satisfies $g(0,t) = \sum_{j=1}^{d-1} dy_j^2 +|\dot\gamma(t)|_g^2 dt^2$, a consequence of the orthonormality of $E_1,\dots,E_\d$.

Moreover, the Christoffel symbols in Fermi coordinates vanish along $\gamma$: $\Gamma_{ij}^k(0,t)\equiv 0$.    The latter follows by considering the cases $i=\d$ and $i,j\neq\d$ separately, viewing $t$ as the ``$d$-th'' coordinate.  In the first case, $\Gamma_{\d j}^k(0,t)=0$ for any $j$ since the orthonormal frame is parallel.  The second case does not use that $\gamma$ is a geodesic or that the frame is parallel: it follows since $s \mapsto (sv^1,\dots sv^{d-1},t_0)$ in the Fermi coordinates always defines a geodesic through the point $\gamma(t_0)$ which is orthogonal to $\gamma$, at which point the geodesic equations in coordinates implies the more general identity $\Gamma_{ij}^k(sv,t)v^iv^j=0$ for $s$ in a small open interval about 0. 

In the local Fermi coordinates, the vector fields $\frac{\prtl}{\prtl y_1}, \dots, \frac{\prtl}{\prtl y_{\d-1}}, \frac{\prtl}{\prtl t}$ satisfy $\frac{\prtl}{\prtl y_j}\big|_{\gamma(t_0)} =E_j(t_0)$, $\frac{\prtl}{\prtl t}\big|_{\gamma(t_0)} =|\dot\gamma(t_0)|_gE_\d(t_0)$.  Now take horizontal and vertical lifts of them to $(\gamma(t),\dot\gamma(t_0)) \in TM$
\begin{equation}\label{tangentframedef}
\begin{gathered}
 \Eb_j(t_0) := \Big(\frac{\prtl}{\prtl y_j}\Big)^h\Big|_{(\gamma(t_0), \dot{\gamma}(t_0))}, \quad 
  \Fb_j(t_0) := \Big(\frac{\prtl}{\prtl y_j}\Big)^v\Big|_{(\gamma(t_0), \dot{\gamma}(t_0))}, \quad j=1,\dots,d-1\\
  |\dot\gamma(t_0)|_g \Eb_\d(t_0) := \Big(\frac{\prtl}{\prtl t}\Big)^h\Big|_{(\gamma(t_0), \dot{\gamma}(t_0))}, 
  \quad |\dot\gamma(t_0)|_g\Fb_\d(t_0) := \Big(\frac{\prtl}{\prtl t}\Big)^v\Big|_{(\gamma(t_0), \dot{\gamma}(t_0))}.
 \end{gathered}
\end{equation}

\begin{proposition}\label{P:symporthoframe}
Let $\gamma:\RR \to M$ be a geodesic of speed $r=  |\dot\gamma(t)| >0$.
\begin{enumerate} 
\item The frame $(\Eb_1(t), \dots, \Eb_\d(t),\Fb_1(t), \dots, \Fb_\d(t))$ defined in \eqref{tangentframedef} is both $\bg$-orthonormal and $\nu$-symplectic at each $(\gamma(t),\dot\gamma(t)) \in TM$.
\item For any $t \in \RR$, $\Fb_\d(t)$ is normal to the tangent space $T_{(\gamma(t),\dot\gamma(t))} S^{(r)}M$, and 
\begin{equation}\label{tgtspaceSM}
	T_{(\gamma(t),\dot\gamma(t))} S^{(r)}M = \text{span}\{\Eb_1(t), \dots, \Eb_\d(t),\Fb_1(t), \dots, \Fb_{\d-1}(t)\}.
\end{equation}
\end{enumerate}
\end{proposition}
\begin{proof} Let $\theta = (\gamma(t), \dot\gamma(t))\in TM$ for some $t \in \RR$.  For the first claim, the frame is $\bg$-orthonormal due to the coordinate formulae \eqref{horvertdecomp}, \eqref{connectionmap} and the fact that the Christoffel symbols in local Fermi coordinates vanish along $\gamma$.  Moreover, the definition in \eqref{tangentframedef} and the characterization of the almost complex structure $\mathbf{J}$ in \eqref{acstructure} implies that 
\begin{equation*}
	\mathbf{J}_\theta \big( \Eb_j(t) \big)= \Fb_j(t), \qquad \mathbf{J}_\theta \big(\Fb_j(t)\big) =- \Eb_j(t).
\end{equation*}
Consequently, using that $H_\theta,V_\theta$ are $\bg$-orthogonal, we have that for all $j,k$
\begin{equation}\label{symporthoframe}
	\begin{split}
		\nu_\theta\big(\Eb_j(t), \Eb_k(t)\big) &= 0 =\nu_\theta\big(\Fb_j(t), \Fb_k(t)\big)\\
		\nu_\theta\big(\Eb_j(t), \Fb_k(t)\big) &= \bg_\theta\big(\mathbf{J}_\theta \big(\Eb_j(t)\big), \Fb_k(t)\big) = \bg_\theta\big(\Fb_j(t), \Fb_k(t)\big) = \delta_{jk}.
	\end{split}
\end{equation}
This is exactly what it means for the basis to be symplectic with respect to $\nu$.

For the second part of the proposition, we borrow the notation above and let $(x_1,\dots,x_\d) = (y_1,\dots, y_{\d-1}, t)$ denote local Fermi coordinates with $(x,X)$ induced coordinates on $TM$.  The differential of $\frac 12 g_{ij}(x)X^iX^j$ at any $\theta = (\gamma(t),\dot\gamma(t))$ is just 
\begin{equation*}
	d\Big( \frac 12 g_{ij}(x)X^iX^j\Big)\Big|_{\theta}= \frac 12 \frac{\prtl g_{ij}}{\prtl x_k} (x) X^iX^j dx_k|_\theta + g_{ij} (x) X^i dX^j|_\theta = |\dot\gamma(t)|^2 dX^\d|_\theta.
\end{equation*}
Here we have used that $X^i(\theta) = \delta^{i\d}$ in the local coordinates and that the  partials of the metric tensor vanish along $\gamma$ since the Christoffel symbols do.  Hence the kernel of $dX^\d|_\theta$ determines the tangent space to $S^{(r)} M$.
If we now raise $dX^\d|_\theta$ with respect to the Sasaki metric and normalize to obtain a unit vector in $T_\theta TM$, we get exactly $\Fb_\d(t)$ and the rest of the claim follows.
\end{proof}

Now observe that for all $t\in \RR$, 
\begin{equation}\label{sympsubspace}
\mathbf{W}_t := \text{span}\{\Eb_1(t), \dots, \Eb_{\d-1}(t),\Fb_1(t), \dots, \Fb_{\d-1}(t)\} \quad \text{ and } \quad \mathbf{W}_t^\perp =  \text{span}\{\Eb_\d(t),\Fb_\d(t)\}
\end{equation}
both define symplectic subspaces in that the restriction of $\nu$ to $\mathbf{W}_t, \mathbf{W}_t^\perp$ is nondegenerate.  Moreover, the observations above concerning tangential and normal Jacobi fields imply that  $\mathbf{W}_t, \mathbf{W}_t^\perp $ are invariant subspaces for $d\kappa_t$ and hence both $d\kappa_t|_{\mathbf{W}_0}$, $d\kappa_t|_{\mathbf{W}_0^\perp}$ define symplectic maps. 

\begin{proposition}\label{P:SVD}
The restricted map $d\kappa_t|_{\mathbf{W}_0}:\mathbf{W}_0 \to \mathbf{W}_t$ admits a singular value decomposition in the following sense: there exists $t$-dependent bases $\{\Gb_1, \dots, \Gb_{2\d-2} \}$ and $\{\Hb_1, \dots, \Hb_{2d-2} \}$ of $\mathbf{W}_0$ and $\mathbf{W}_t$ respectively which are both $\bg$-orthogonal and $\nu$-symplectic in their tangent spaces such that for some positive real numbers $\sigma_1(t) \geq \sigma_2(t) \geq \dots \geq \sigma_{2\d-2}(t)$
\begin{equation*}
d\kappa_t(\Gb_j) = \sigma_j\Hb_j \qquad j = 1,\dots, 2\d-2.
\end{equation*}
The ordered singular values come in reciprocal pairs $\sigma_1 = \sigma_{2\d-2}^{-1}, \sigma_2 = \sigma_{2\d-3}^{-1}, \dots, \sigma_{\d-1} = \sigma_{\d}^{-1}$.
\end{proposition}
\begin{proof}
The main idea behind the proof is to use the frame $E_1(t), \dots E_{\d-1}(t)$ for the normal bundle of $\gamma$ in \eqref{parallelONframe} to reduce the theorem to the level of matrices, where the theorem is known. Consider the fundamental matrix $\Fsc(t)$ of the system \eqref{jacobifirstorder} expressing the Jacobi equation, omitting $J_\d, L_\d$.  
The $j$-th column of $\Fsc(t)$ is thus given by $(J^1(t),\dots, J^{\d-1}(t), L^1(t),\dots,L^{\d-1}(t))$ where 
\begin{equation*}
	\begin{cases}
		J^i(0) = \delta^{ij} \text{ and } L^i(0)=0 & \text{if }j =1,\dots,\d-1,\\
		J^i(0) = 0 \text{ and } L^i(0) = \delta^{ij} & \text{if }j=\d,\dots,2\d-2.
	\end{cases}
\end{equation*}
Since $d\kappa_t|_{\mathbf{W}_0}$ is symplectic, and the frame in Proposition \ref{P:symporthoframe} is symplectic, it follows that\footnote{This introduces our notation for the space of $2(d-1)\times 2(d-1)$ real symplectic matrices.} $\mathscr{F}(t) \in \Spdr$ is a symplectic matrix in that 
\begin{equation}\label{fundsymp}
	\Fsc^T \Jb \Fsc = \Jb \quad \text{ and } \quad \Fsc^{-1} = \Jb \Fsc^T \Jb^{-1} \quad \text{ where } \quad 		\Jb = 
	\begin{bmatrix}
		0 & I_{\d-1} \\
		-I_{\d-1} & 0 
	\end{bmatrix}.
\end{equation}
It is known that symplectic matrices admit a singular value decomposition
\begin{equation}\label{symplecticsvd}
\mathscr{F}(t) = \mathscr{O}_2(t) \Sigma(t) \mathscr{O}_1(t)^T, \quad \text{ where } \quad 
\mathscr{O}_1,\mathscr{O}_2 \in  \Spdr \cap \Odr,
\end{equation}
where \textbf{O} denotes the space of orthogonal matrices satisfying  $\mathscr{O}^T = \mathscr{O}^{-1}$.  The fact that the matrices in the singular value decomposition can be taken to be symplectic matrices is known as a \emph{Bloch-Messiah} decomposition in the physics literature or an \emph{Euler} decomposition in mathematical literature.  The result now follows by using the entries of $\mathscr{O}_1^T, \mathscr{O}_2^T$ to determine the coefficients of $\Gb_j, \Hb_j$ with respect to the frame defining $\mathbf{W}_t$ in \eqref{sympsubspace}.

The final claim that $\sigma_1 = \sigma_{2\d-2}^{-1}, \sigma_2 = \sigma_{2\d-3}^{-1}, \dots, \sigma_{\d-1} = \sigma_{\d}^{-1}$ follows since these are the eigenvalues of  the positive definite symplectic matrix $\Fsc^T \Fsc$.  Whenever a symplectic matrix has a real eigenvalue $\varsigma$, then its reciprocal $\varsigma^{-1}$ is also an eigenvalue, so these identities are a consequence of the decreasing ordering $\sigma_j \geq \sigma_{j+1}$.
\end{proof}

Given Proposition \ref{P:symporthoframe} and the first part of \eqref{difftangential}, we obtain the following corollary.
\begin{corollary}\label{C:singularvalues}
	Let $\sigma_1(t), \dots, \sigma_{2\d-2}(t)$ denote the singular values defined in Proposition \ref{P:SVD}.  Restricting $d\kappa_t$ to the larger domain $T_{(\gamma(0),\dot\gamma(0))} S^{(r)}M $ yields a bijection with $T_{(\gamma(t),\dot\gamma(t))} S^{(r)}M $ with singular values $\{\sigma_1(t), \dots, \sigma_{2\d-2}(t),1\}$.  The last singular value is determined by $d\kappa_t( E_\d(0)) = E_\d(t)$.
\end{corollary}

\subsection{The Sasaki Metric on the Cotangent Bundle}
The musical isomorphism $\xi \mapsto \xi^\#$ which maps covectors to tangent vectors via the Riemannian metric $g$ on $M$ defines a diffeomorphism between the cotangent bundle $T^*M$ and the tangent bundle $TM$.   Consequently, by pulling back the Sasaki metric $\bg$ on $TM$ defined in \S\ref{SS:SasakiTM}, we obtain a metric $\tilde g$ on $T^*M$ for which the musical isomorphism is an isometry.  In \cite{SatoSasakiCotangent}, Sat\^{o} observed the following formula for the metric in coordinates $(x,\xi)$, analogous to \eqref{sasakicoord}:
\begin{equation*}
	\tilde g = g_{jk}(x)dx_j dx_k + g^{jk}(x)\pmb{\delta} \xi^j \pmb{\delta}\xi^k, \qquad \pmb{\delta}\xi^j := d\xi_j - \xi_\ell \Gamma_{jk}^\ell (x)dx_k.
\end{equation*}

It was shown by Tondeur \cite{TondeurSasakiCotangent} that the canonical 2-form and $\tilde g$ define an almost K\"ahler structure on $T^*M$, namely there exists an almost complex structure $\mathbf{J}$ such that $(d\xi\wedge dx)(\Xi,\Upsilon) =  \langle\mathbf J \Xi,\Upsilon\rangle_{\tilde g}$ for vectors $\Xi,\Upsilon \in T(T^*M)$.  Consequently, the Liouville volume form $\frac{1}{d!} (d\xi\wedge dx)^d$ and the Riemannian volume form determined by $\tilde g$ are identical.  Indeed, as in \eqref{symporthoframe} from the proof of Proposition \ref{P:SVD} above, we may take a basis  $\Eb_1, \dots, \Eb_\d, \Fb_1, \dots, \Fb_\d$ which is simultaneously $\tilde g$-orthonormal and symplectic with respect to the canonical 2-form.  Hence both volume forms evaluate to 1 along this ordered basis, which means they are identical.  This means that below, any integral expressed in coordinates $(x,\xi)$ on $T^*M$ can be integrated with respect to the usual Lebesgue measure $dx d\xi$ to achieve the same result as integrating with respect to either volume form on $T^*M$.

\section{Examples}\label{S:examples}
To study instances of where the hypotheses in Theorems \ref{T:mainthmexp} and \ref{T:mainthmpoly} are satisfied, we will use the following variational characterization of singular values of $d\kappa_t$ (see e.g. \cite[Theorem 7.3.8]{HornJohnsonMatrixAnalysis})
\begin{align}
		&\sigma_j(t) = \max_{\{S:\dim S = j\}} \min_{\{X: 0 \neq X \in S\}} \frac{|d\kappa_t(X)|_{\bar g}}{|X|_{\bar g}},\label{singvalvariation}\\
		&\text{ where }\sigma_1 (t)\geq \sigma_2 (t)\geq \cdots \geq \sigma_{2\d-2}(t) \text{ as assumed above.}			\label{singvaldecreasing}
\end{align}
where $S$ denotes subspaces $S$ in the subspace $\W_0$ defined in \eqref{sympsubspace}.
This follows as a consequence of the Courant-Fischer min-max theorem (see e.g. \cite[Theorem 4.2.6]{HornJohnsonMatrixAnalysis}) for eigenvalues applied to the positive definite transformation $(d\kappa_t)^Td\kappa_t$.

Recall in \eqref{geodrescalemap} we observed that if $\tau_r(p,V) = (p,rV)$ is the rescaling map for $r>0$, then the identity there can equivalently be expressed as $\kappa_t = \tau_r \circ \kappa_{rt} \circ \tau_{1/r}$.  Now consider $(p,V) \in S^{(r)}M $ where $r \in [1-\veps_0,1+\veps_0]$ for some $0 < \veps_0 <1$. The chain rule then bounds the   singular values $\tilde\sigma_j(t)$ for $d\kappa_t$ at $(p,V)$ of in terms of the  singular values $\sigma_j(t)$  of $d\kappa_t$ at $(p,\frac Vr) \in S^{(1)}M $ as 
\begin{equation}\label{singvaldilation}
	\tilde{\sigma_j}(t) \approx_{\veps_0} \sigma_j(rt).
\end{equation}
This allows us to limit attention to the unit tangent bundle $S^{(1)}M $ in checking the hypotheses of Theorems \ref{T:mainthmexp} and \ref{T:mainthmpoly} in what follows.  In particular, if the singular values on $S^{(1)}M$ satsify $\log\sigma_j(t) \gtrsim |t|$ or $\sigma_j(t) \gtrsim |t|^k$ for $|t| \geq 1$, then the same holds for the singular values on $S^{(r)}M$ up to a change in the implicit constant. 
\subsection{Hyperbolic and partially hyperbolic geodesic flows}\label{SS:phyp}
Let $\kappa_t$ denote the geodesic flow on $TM$, restricting attention to its action on the unit tangent bundle $S^{(1)}M $, and let $\Xi$ denote the Hamiltonian vector field of $H(p,V) = \frac 12 |V|_{g(p)}^2$ which generates the flow. The geodesic flow is said to be \emph{Anosov} or \emph{hyperbolic} if  there exists $\nu >0$ and a splitting of $S^{(1)}M $ into $\kappa_t$-invariant subspaces $T_{(x,\xi)} S^{(1)}M  = E^f(x,\xi)\oplus E^s(x,\xi) \oplus E^u(x,\xi) $ where $E^f = \text{span }\{\Xi\}$ and 
\begin{equation*}
	\begin{split}
		&\Upsilon \in E^u \implies \begin{cases}
			\big|d\kappa_{t}(\Upsilon )\big|_{\bg} \gtrsim e^{ \nu t}|\Upsilon |_{\bg}, \quad \text{ if }t > 0, \\
			\big|d\kappa_{t}(\Upsilon )\big|_{\bg} \lesssim e^{ \nu t}|\Upsilon |_{\bg}, \quad \text{ if }t < 0, 
		\end{cases}
		\\
		&\Upsilon  \in E^s \implies \begin{cases}
			\big|d\kappa_{t}(\Upsilon )\big|_{\bg} \lesssim e^{ -\nu t}|\Upsilon |_{\bg}, \quad \text{ if }t > 0, \\
			\big|d\kappa_{t}(\Upsilon )\big|_{\bg} \gtrsim e^{- \nu t}|\Upsilon |_{\bg}, \quad \text{ if }t < 0 .
		\end{cases}
	\end{split}
\end{equation*}

The author and Sogge \cite{BlairSoggeToponogov} showed the nonconcentration bound \eqref{KNgain} for Anosov flows. While not explicitly mentioned there, results of Klingenberg \cite{KlingenbergAnosov} and Ma\~n\'e \cite{ManeKlingenbergTheorem} imply that the leading factor in the Hadamard parametrix decays exponentially, see the discussion in \cite[\S5]{BlairSoggeCriticalExp}.

Theorem \ref{T:mainthmexp} allows us to relax this Anosov condition.  The geodesic flow is said to be \emph{partially hyperbolic} if there exists a splitting into nontrivial invariant subspaces
\begin{equation*}
	T_{(x,\xi)} S^{(1)}M = E^f(x,\xi)\oplus E^s(x,\xi) \oplus E^u(x,\xi) \oplus E^c(x,\xi)
\end{equation*}
where $E^f$, $E^s$, $E^u$ satisfy the same properties as before but there exists $0< \mu <\nu$ such that 
\begin{equation*}
	e^{-\mu t} \lesssim \big|d\kappa_t(X)\big|_{\bg} \lesssim e^{\mu t} \text{ for all } X \in E^c, t \in \RR.
\end{equation*}
The subspace $E^c$ thus serves as a ``center'' subspace in which any expansion or contraction is less pronounced relative to $E^s, E^u$.

Partially hyperbolic geodesic flows are thus an instance where the hypotheses of Theorem \ref{T:mainthmexp} are satisfied with $\sup_{T>0} \Theta(T) \lesssim 1$.  Indeed, by making use of \eqref{singvalvariation}, \eqref{singvaldilation}, see that the key hypotheses $\mu(t) \lesssim e^{\Lambda |t|}$ is satisfied by taking $\Lambda$ to be a constant multiple of $\nu$.  Moreover, it follows that $\sigma_1(t)  \gtrsim e^{c|t|}$ where $c$ is a small constant multiple of $\nu$, and hence we can take $\vartheta (t)\approx e^{c|t|}$. Manifolds with such flows were studied extensively by Carniero and Pujals \cite{CarneiroPujals} and we borrow their definition above.  In particular, they construct examples of Riemannian manifolds with partially hyperbolic geodesic flows which possess conjugate points.  Since the methods of \cite{Berard77} and hence \cite{BlairSoggeToponogov} rely crucially on the absence of conjugate points, Theorem \ref{T:mainthmexp} does present a significant relaxation of the hypotheses in previous works.

\subsection{Product manifolds}\label{SS:products}
Let $(M_1, g_1)$, $(M_2,g_2)$ be Riemannian manifolds of dimension $d_1,d_2 \geq 2$ and let $M = M_1 \times M_2$ denote the product manifold with $g = g_1 \oplus g_2$ denoting the product metric.  Given any $ X \in T_pM$, we write $X = X_1+ X_2$, under the natural identification of $TM$ with $TM_1 \oplus TM_2$.  We have the following formulas relating operations on $M$ to those on\footnote{The superscripts $(k)$ denote the corresponding operation on $M_k$, $k=1,2$.} $M_1, M_2$:
\begin{equation}\label{productconnex}
\begin{aligned}
	\nabla_{Y_1 + Y_2}(X_1 + X_2)& = \nabla_{Y_1 }^{(1)}(X_1 ) + \nabla_{Y_2 }^{(2)}(X_2 ), \\
	[X_1+X_2, Y_1+Y_2]& = [X_1,Y_1]^{(1)} + [X_2,Y_2]^{(2)}, \\
	R(X_1+X_2,Y_1+Y_2)(Z_1+Z_2) &= R^{(1)}(X_1,Y_1)(Z_1) +R^{(2)}(X_2,Y_2)(Z_2), \\
	\R_{Y_1+Y_2}(X_1+X_2) &= \R_{Y_1}^{(1)}(X_1) + \R_{Y_2}^{(2)}(X_2),
\end{aligned}
\end{equation}
where the right hand side in each case expresses the decomposition in $TM_1 \oplus TM_2$.
The first of these is in \cite[p.139]{DoCarmoRiemannian} and the second by direct calculation.  The remaining identities follow as a consequence.  As a consequence of the first identity, $\gamma(t) = (\gamma_1(t),\gamma_2(t))$ is a geodesic on $M$ if and only if $\gamma_k(t)$ is a (possibly constant) geodesic on $M_k$, $k=1,2$.  As a consequence of these identities, if $J(t)$ is a Jacobi field on $M$, then $J(t) = J_1(t)+J_2(t)$ where each $J_k$ is a Jacobi field on $M_k$ determined by initial data $J_k(0), \covt^{(k)}J_k(0) \in TM_k$.

\begin{theorem}
	Suppose $(M_1,g_1)$ has nonpositive sectional curvatures.  Then the lower bound in \eqref{varrholowerhyp} can be taken to satisfy $\vartheta(t) \gtrsim |t|^{d_1-1}$ for $|t| \gtrsim 1$.
\end{theorem}
\begin{proof} Without loss of generality we consider $t >0$ throughout the proof.  Moreover, given \eqref{singvaldilation}, it suffices to consider the differential of $\kappa_t$ restricted to cosphere bundle $S^{(1)}M$.  The notation below will not distinguish between lengths of tangent vectors in the $g_1$ or $g_2$ metric as this should be clear from the context.  
	
	Consider a geodesic $\gamma_1: \RR\to M_1$ of any nontrivial speed.  Let $J_1(t)$ be any Jacobi field satisfying $J_1(0) =0$, $|\covt J_1(0)| = 1$.  We first revisit the usual argument that  $|J_1(t)| \geq t$ for all $t \in (0,\infty)$.  There is a maximal interval $(0,R) \subset \RR$ such that $J_1(t) \neq 0$ for all $t \in (0,R)$.  Using properties of the Levi-Civita connection $\nabla^{(1)}$ we have for all $t\in (0,T)$,
	\begin{equation*}
		\frac{d^2}{dt^2} |J_1|= -\frac{Rm^{(1)}(J_1,\dot\gamma_1,\dot\gamma_1, J_1 )}{|J_1|_{g_1}} + \bigg(\frac{|\covt^{(1)} J_1|_{g_1}^2}{|J_1|_{g_1}} - \frac{\langle\covt^{(1)} J_1,J_1\rangle_{g_1}^2}{|J_1|_{g_1}^3}\bigg) \geq  0,
	\end{equation*}
	as the expression in parenthesis is nonnegative by Cauchy-Schwarz. Despite its appearance, $|J_1(t)|_{g_1}$ is differentiable from the right at $t=0$: since $J_1(t)=tW(t)$ (cf. \cite[Proposition 10.10]{LeeRiemannian}) for some smooth vector field $W(t)$ along $\gamma$, it follows that $\lim\limits_{t\to 0+} \frac{|J_1(t)|_{g_1}-|J_1(0)|_{g_1}}{t} = \lim\limits_{t\to 0+} |W(t)|=1$. Taylor's theorem now implies that for all $t \in [0,R)$,  $|J_1(t)| \geq t$ and hence $R=\infty$.
	
	We now show that  $|J_1(t)| = t$ in the case where $\gamma_1$ is constant: $\gamma_1(t) = p$ for all $t \in \RR$.  In this case, $J_1(t)$ should be treated as a map $J_1:\RR\to T_pM_1$ and the covariant derivative coincides with the ordinary derivative \cite[p. 57-8]{DoCarmoRiemannian}.  Moreover, $\R_{\dot\gamma(t)}^{(1)} \equiv 0$, so the Jacobi equation with initial data  $J_1(0) =0$, $|\covt J_1(0)| = 1$ is solved by $J_1(t) = t\covt J_1(0)$.  It follows that $|J_1(t)| =t$ for $t >0$.  
	
	Now consider any unit speed geodesic in $M$, $\gamma(t) = (\gamma_1(t),\gamma_2(t)) \in M_1 \times M_2$, $|\dot\gamma_1|^2+|\dot\gamma_2|^2 =1$.  We solve the Jacobi equation with initial data $J(0) = 0$, $\covt J (0) \in TM_1$ so that \eqref{productconnex} and uniqueness of solutions imply that $J(t) = J_1(t)$, $J_2(t) \equiv 0$. In particular, we take $\covt J_1(0) \perp \dot\gamma_1$ and $|\covt J_1(0)|=1$.  The preceding bounds give $|J(t)| \geq t$.  Since the initial data $\covt J_1(0)$ can be chosen from a subspace of dimension at least $d_1-1$, the variational characterization \eqref{singvalvariation} implies that the first $d_1-1$ singular values arranged as in \eqref{singvaldecreasing} satisfy $\sigma_j(t)\gtrsim t$ and hence $\vartheta(t) \gtrsim t^{d_1-1}$.
\end{proof}

\begin{remark} Improvements on the universal $L^p$ bounds \eqref{soggeefcn} for product manifolds were considered thoroughly by Huang, Sogge, and Taylor \cite{HuangSoggeTaylor}, following the work of Iosevich and Wyman \cite{IosevichWymanProducts} in the case of spheres. We do not attempt to give a rigorous comparison of their results to ours, instead viewing product manifolds as a interesting case where our hypotheses are satisfied. 
	
\end{remark}
\subsection{Integrable geodesic flows}\label{SS:integrable}
Let $H_1: TM \to \RR$, $H_1(p,V) = \frac 12 |V|_{g(p)}^2$ be the Hamiltonian function generating the geodesic flow.  The geodesic flow on $(M,g)$ is said to be  \emph{integrable} if there exists real-valued functions $H_2,\cdots, H_\d$ on $TM$ such that $\{H_j,H_k\} \equiv 0$ for all $1 \leq j,k \leq \d$ and $dH_1,\dots,dH_\d$ are linearly independent on a dense open set.  Here $\{\cdot,\cdot\}$ denotes the Poisson bracket $\{H_j,H_k\} =\nu(\Xi_{H_j},\Xi_{H_k})$, where $\nu$ is the symplectic form on $TM$ defined in \S\ref{S:Sasaki}.  

The level sets of the \emph{moment map} $(p,V) \mapsto (H_1(p,V),\dots,H_\d(p,V))$ are known to foliate the dense open set where $dH_1,\dots,dH_\d$ are linearly independent.  Each such \emph{regular} level set defines a Lagrangian submanifold with each connected component diffeomorphic to the torus $\TT^\d$, see e.g. \cite[\S5.2]{AbrahamMarsden}, \cite[\S49]{ArnoldClassicalMechanics}.  Given such a component, the Liouville-Arnold theorem furnishes \emph{action-angle coordinates} in a flow-invariant neighborhood $\Omega$ of it.  This is a symplectomorphism $\tilde\kappa (\phi,\I ) :  \TT^\d \times B  \to \Omega$ on an open ball $B \subset \RR^{\d}$ which expresses the pullback $\tilde H:=\tilde\kappa^* H_1$ as a function of $\I $ alone where  $ (\phi,\I ) \in \TT^\d \times B$: $H (\tilde\kappa(\phi,\I ))= \tilde H(\I ) $ for some $\tilde H$ independent of $\phi$.  

In what follows we assume there is a neighborhood $\Omega \subset TM$ as above and action-angle coordinates $ (\phi,\I ) \in \TT^\d \times B$ for which $\tilde H(\I )$ satisfies the following
\begin{equation}\label{levelcurv}
	\text{the level set $\{\I  \in B : 2\tilde H(\I ) =1\}$ has nonvanishing Gaussian curvature in }B.
\end{equation}
This property is not universally satisfied.  However, it is a common hypotheses in the analysis of integrable Hamiltonian systems, in particular KAM theory.  One sufficient condition for \eqref{levelcurv} is the \emph{Kolmogorov nondegeneracy condition} which states that
\begin{equation}\label{kndc}
	\text{the Hessian }\frac{\prtl^2 \tilde H}{\prtl \I^2}\text{ is nonsingular}. 
\end{equation}

The conditions \eqref{levelcurv} and \eqref{kndc} are satisfied in many cases of interest. The work of Kn\"orrer \cite{Knorrer} shows that the geodesic flow on the triaxial ellipsoid satisfies \eqref{kndc}, see also Gomes and Zelditch \cite{GomesZelditch}.  The work of Zung \cite{ZungKolmogorov} showed that \eqref{kndc} is satisfied nearby (but not at!) a hyperbolic singularity of the moment map.  On simple surfaces of revolution, Bleher \cite{BleherSurfaceOfRevolution} introduced a \emph{twist hypothesis} on the geodesic flow, which is in some sense generic and includes ellipsoids of revolution except for the sphere.  As shown by Chabert \cite{chabert2025bounds}, this twist hypothesis implies \eqref{levelcurv}.  We note that the canonical sphere is one case where \eqref{kndc} is not satisfied, see \cite[p.384]{RoyIntegrable}.

The Hamiltonian vector field of $\tilde H$ in action angle coordinates takes the simple form $\Xi_{\tilde H} = \sum_{j=1}^d \frac{\prtl\tilde H}{\prtl \I _j} \frac{\prtl}{\prtl \phi_j}$ so that the corresponding flow takes the form $t \mapsto (\phi + t d\tilde H(\I ), \I )$.  Consequently, the matrix of the differential of the flow in the basis $\frac{\prtl}{\prtl \phi_1}, \dots, \frac{\prtl}{\prtl \phi_\d}, \frac{\prtl}{\prtl \I_1}, \dots, \frac{\prtl}{\prtl \I_\d}$ takes the form 
\begin{equation*}
	\begin{bmatrix}
		I_\d & t\frac{\prtl^2 \tilde H}{\prtl \I^2}\\
		0_\d & I_\d
	\end{bmatrix}.
\end{equation*}
Even though the basis used here is symplectic, it is not necessarily orthogonal.  However, given a vector $\Upsilon = \Upsilon^j \frac{\prtl}{\prtl \phi_j} + \Upsilon^{j+\d}\frac{\prtl}{\prtl \I_j} \in T_{(\phi,\I)}TM$, there are implicit constants so that uniformly in $\Omega$, $ |\Upsilon|_{\bg}^2 \approx \sum_{j=1}^{2\d} (\Upsilon^{j})^2 $.  In other words, the length of $\Upsilon$ in the Sasaki metric is comparable to its length in the Euclidean metric.  This allows us to use \eqref{singvalvariation} to get that the first $d-1$ singular values of the geodesic flow restricted to $\W_0$ satisfy $\sigma_j(t) \approx 1+ |t|$ and hence one can take $\mu(t) \approx 1+ |t|$ and $\vartheta(t) \approx (1+|t|)^{d-1}$.  Indeed, the variational bound can be checked in the Euclidean metric, at which point they follow since the Sasaki metric is comparable.  Recall the significance of these bounds is discussed in Remark \ref{R:polyimprovement}.

\begin{theorem}\label{T:integrable}
	Suppose the geodesic flow on $(M,g)$ integrable.  Suppose $\Omega \subset TM$ is a neighborhood of a regular level set on which action-angle coordinates can be taken and that \eqref{levelcurv} is satisfied.  Then for some $\veps_0>0$ sufficiently small, if $(p,V) \in \Omega \cap \{|V| \in [1-\veps_0,1+\veps_0]\}$, then the upper bound $\mu(t)$ \eqref{muupperhyp} and the lower bound in \eqref{varrholowerhyp} for the differential at $(p,V)$ can be taken to satisfy $\mu(t) \approx |t|$ and $\vartheta(t) \approx |t|^{d-1}$ respectively for $|t| \geq 1$.
\end{theorem}

\section{The Complex Riccati equation}\label{S:riccati}
\subsection{The equation and its solutions} Given a geodesic $\gamma: \RR \to M$ at any nontrivial speed, let $\Pi_{\dot\gamma(t)}^\perp:T_{\gamma(t)}M \to T_{\gamma(t)}M$ denote the projection onto the normal space of $\gamma$.  As in \eqref{tidalforce}, let $\R_{\dot\gamma(t)}(V) = R(V, \dot\gamma(t))\dot\gamma(t)$ denote the endomorphism on $T_{\gamma(t)}M$ formed by contracting the Riemann curvature tensor along $\dot\gamma$.
In this section, we examine complex solutions to the following Riccati equation, treating $\W_t$ is an endomorphism on the complexified tangent spaces $T_{\gamma(t)}^\CC:=(T_{\gamma(t)}M) \otimes \CC$ (equivalently, $\W_t$ is a (1,1) tensor on $T_{\gamma(t)}^\CC$)
\begin{equation}\label{riccatiraised}
	\covt\W_t + \W_t\circ\Pi_{\dot\gamma(t)}^\perp\circ\W_t  + \Rgt =0.
\end{equation}
The metric allows us to lower an index in $\W_t$ to obtain a $(0,2)$ tensor $\omega = \W^\flat$ satisfying
\begin{equation}\label{riccatiflat}
	\covt\omega_t + \big(\omega_t^\#\circ\Pi_{\dot\gamma(t)}^\perp\circ\omega_t^\#\big)^\flat + Rm(\cdot,\gammadot(t),\gammadot(t),\cdot) =0,
\end{equation}
which is the complex Riccati equation for $\omega_t$ in \eqref{phaseintro}.

The main result of this section is that there exists solutions $\W_t$ to \eqref{riccatiraised} satisfying 
\begin{equation}\label{keyriccati}
	\begin{split}
	\big\langle\W_t(V), W \big\rangle_g = \big\langle W, \W_t(V)\big\rangle_g  \qquad &\text{(symmetric)}\\
	\Im \big\langle \W_t(V),\bar V\big\rangle_g >0  \qquad &\text{(positive definite imaginary part)}
	\end{split}
\end{equation}
for all $V,W \in T_{\gamma(t)}^\CC$.  For brevity, let $\gamma^\perp(t)$ denote the normal space over $\gamma(t)$, that is the collection of vector fields along $\gamma(t)$, $X \in T_{\gamma(t)}^\CC$ such that $\left\langle X, \dot{\gamma}(t)\right\rangle_g=0$.  We will show there exists a solution to $\W_t$ \eqref{riccatiraised} with $\gamma^\perp(t)$, span$\{\dot\gamma(t)\}$ both invariant subspaces
\begin{equation}\label{wtdef}
  \W_t(X) = \begin{cases}
                   \V_t(X), &  X \in \gamma^\perp(t),\\
                   iX, & X \in\text{span} \{ \dot{\gamma}(t) \},
                 \end{cases}
\end{equation}
where\footnote{It is unfortunate that the choice of notation $\V_t$ here is similar to the $\Vsc_t$ used for the parametrix in Theorem \ref{T:Vtthm}.  However, the parametrix does not appear in this section and the tensor notation $\V_t$ is limited to \S\ref{S:riccati}.} $\V_t: \gamma^\perp(t) \to  \gamma^\perp(t)$ satisfies the properties \eqref{keyriccati} (for $V,W \in \gamma^\perp(t)$) and satisfies
\begin{equation}\label{riccati}
  \covt \V_t + \V_t^2 + \Rgt =0.
\end{equation}

Observe that defining $\W_t$ as in \eqref{wtdef} means that its restriction to $\text{span} \{ \dot{\gamma}(t) \}$ satisfies \eqref{riccatiraised}.
Indeed, $(\W_t\circ\Pi_{\dot\gamma(t)}^\perp\circ\W_t )(\dot\gamma(t)) + \Rgt (\dot\gamma(t))=0$ by symmetries of the Riemann curvature tensor.  Moreover, since $\nabla_t\dot\gamma(t)=0$, it follows that
\begin{equation*}
	\big(\covt\W_t\big)(\dot\gamma(t))  = \covt(\W_t(\dot\gamma(t)) )- \W_t(\covt\dot\gamma(t))=  i\covt(\dot\gamma(t)) =0.
\end{equation*}

In the case where $\V_t$ is a real endomorphism on $T_{\gammadot(t)} M$ (not its complexification), Riccati equations of the form \eqref{riccati} have been studied extensively in Riemannian geometry.  For example, the Hessian of a distance function satisfies this equation, see e.g. \cite[Ch. 11]{LeeRiemannian}.  The complex equation in \eqref{riccati} can be solved in nearly the same way by making use of Jacobi fields.

As noted in \S\ref{S:Sasaki}, if a Jacobi field and its covariant derivative both lie in the normal space $\gamma^\perp(t_0)$ for some $t_0$, then this persists for all $t$.  With this in mind, we say that a \emph{Jacobi tensor} on $\gamma^\perp$ is an endomorphism solving the equation
\begin{equation}\label{jacobitensor}
 \covt^2 \J_t + \R_{\gammadot} \circ \J_t=0.
\end{equation}
We define Jacobi tensors by starting with any real orthonormal frame for $\gamma^\perp$ parallel along $\gamma(t)$, $E_1(t),\dots, E_{d-1}(t) \in T_{\gammadot(t)} M$ (cf. \eqref{parallelONframe}), then defining $\J_t(E_i(t))$ by
\begin{equation}\label{Jacobitensorframe}
	\J_t(E_i(t)) \text{ is the Jacobi field with prescribed values } \J_0(E_i(0)), \covt \J_0(E_i(0)).
\end{equation}
If $\J_t$ is invertible for all $t\in \RR$, it is then verified that $\V_t=\covt \J_t \circ \J_t^{-1}$ is a solution to \eqref{riccati}.  Note that complex solutions to \eqref{jacobitensor} exist in $T_{\gamma(t)}^\CC$ since the Jacobi equation \eqref{JacobiEqn} is linear and the associated system \eqref{jacobifirstorder} has real coefficients.

Our next two results were essentially observed by Dahl \cite{DahlLeadingTerm}, who also made use of Jacobi tensors.  However, we do opt for a complete treatment in the interest consistency with our formalism.
\begin{theorem}\label{T:mainriccati}
Suppose $\V_0$ is an endomorphism on $\gamma^\perp(0)$ that is symmetric and has positive definite imaginary part as in \eqref{keyriccati}.  There exists an invertible Jacobi tensor $\J_t$ such that $\V_t=\covt \J_t \circ \J_t^{-1}$ solving \eqref{riccati} with initial data $\V_0$ and satisfies \eqref{keyriccati} for all $t \in \RR$. Consequently, there is a global solution $\W_t$ to the Riccati equation \eqref{riccatiraised} on $\RR$ satisfying \eqref{keyriccati}, \eqref{wtdef} for all $t \in \RR$. 
\end{theorem}

\begin{proof} 
Given complex, normal Jacobi fields $X(t)$, $Y(t)$ it is verified using symmetries of the Riemann curvature tensor that the following symplectic forms\footnote{The choice of notation $\nu$ here is appropriate since the conservation of these quantities can also be seen as a consequence of the results in \S\ref{S:Sasaki}, using that the symplectic form $\nu$ in \eqref{symptgtbundle} is preserved under the Hamiltonian flow.  
} over $\gamma^\perp$ are constant in $t$
\begin{equation*}
	\begin{split}
		\nu(X,Y)&= \left\langle X(t),\covt Y(t) \right\rangle_g - \left\langle \covt X(t), Y(t) \right\rangle_g , \\ 
		\nu_\CC(X,Y)&= \big\langle X(t),\covt \overline{Y(t)} \big\rangle_g - \big\langle \covt X(t), \overline{Y(t)} \big\rangle_g  . 
	\end{split}
\end{equation*}
Hence the following are constant in $t$ for a Jacobi tensor $\J$
\begin{align}
&\left\langle \J E_k,(\covt \J)E_\ell \right\rangle_g - \left\langle (\covt \J )E_k, \J E_\ell \right\rangle_g  \label{symplectic1},\\
\frac{1}{2i}&\big(\big\langle (\covt \J )E_k, \overline{\J}E_\ell \big\rangle_g-\big\langle \J E_k,\overline{(\covt \J)}E_\ell \big\rangle_g \big)  \label{symplectic2}.
\end{align}

The desired properties for $\V_t$ will follow by choosing initial data $\J_0, \covt \J_0$ such that \eqref{symplectic1} vanishes and that the matrix with $(k,l)$-th entry given by \eqref{symplectic2} is positive definite.  In other words, if
\begin{equation*}
\covt \J_0^T \circ \J_0 = \J_0^T \circ \covt \J_0
\end{equation*}
then $\covt \J_t^T \circ \J_t$ is symmetric for any $t \in \RR$ and if the initial values for $\J$ are such that
\begin{equation}\label{initialriccatiskewsymm}
  \frac{1}{2i}\Big( \overline{\J_0^T} \circ \covt \J_0 - \overline{\covt \J_0^T} \circ \J_0 \Big) \text{ is symmetric positive definite}
\end{equation}
then the same will hold for every $t \in \RR$.  This in turn implies that $\J$ is invertible, for if $\J_tX =0$, for some $X= c^kE_k(t)$, then
\begin{equation*}
0= \langle (\covt\J) X, \overline{\J X} \rangle_g  - \langle \J X , \overline{(\covt \J)X } \rangle_g = c^k \bar{c}^\ell \big(\big\langle \J E_k,\overline{(\covt \J)}E_\ell \big\rangle_g - \big\langle (\covt \J )E_k, \overline{\J}E_\ell \big\rangle_g  \big)
\end{equation*}
and hence $c^1 = \cdots = c^{d-1} =0$ since \eqref{symplectic2} defines a positive definite matrix.

To solve \eqref{riccati} with the desired initial data $\V_0$, take initial data for $\J$ so that $\V_0=\covt \J_0 \circ \J_0^{-1}$.  There is no unique choice of data to achieve this, among the degrees of freedom is that both $(\J_0,\covt\J_0)$ and $(c\J_0,c\covt\J_0)$ will generate the same $\V_0$ for any constant $c\neq 0$.  Since the hypotheses imply that $\V_0$ is invertible, one choice\footnote{This is consistent with what we will take below to achieve $\V_0=iI$, namely $\J_0 = I$, $\covt \J_0 = iI$.} is $\J_0 = i\V_0^{-1}$, $\covt \J_0 = iI$.  

We now have that $\V_t = \covt \J_t \circ \J_t^{-1}$ solves \eqref{riccati} in a neighborhood of $t=0$ and \eqref{symplectic1} implies that $\V_t$ is symmetric in its maximal domain.  The proof is concluded by observing that 
\begin{equation*}
	 \frac{1}{2i}\big( \overline{\J_t^T} \circ \covt \J_t - \overline{\covt \J_t^T} \circ \J_t \big) = 
	 \frac{1}{2i}\big( \overline{\J_t^T} \circ \big(\V_t -\overline{\V_t^T}\big)\circ \J_t \big) =  
	 \overline{\J_t^T} \circ \big( \Im \V_t\big) \circ \J_t.
\end{equation*}
Thus if $\Im\V_0$ is positive definite, then \eqref{initialriccatiskewsymm} is satisfied, meaning that $\J_t^{-1}$ is always invertible and the maximal domain is all of $\RR$.  Hence the left hand side is positive definite for all $t\in \RR$, which implies the same for $\Im \V_t$.
\end{proof}

\begin{corollary}
Let $\V_t$ be the solution to \eqref{riccati} furnished by Theorem \ref{T:mainriccati}.  The ODE 
\begin{equation}\label{atODE}
a'(t) + \frac 12\tr (\V_t) a(t) =0
\end{equation}
is solved by the function
\begin{equation}\label{atsoln}
a(t) = a(0) (\det \J_t)^{-\frac 12},
\end{equation}
where the square root is always given by a branch cut that ensures continuity of $a(t)$.
\end{corollary}

\begin{proof}
Applying Jacobi's formula to the matrix function for $\J_t$ determined by the orthonormal frame $E_1(t),\dots, E_{d-1}(t)$ gives that 
\begin{equation*}
\frac{d}{dt} \log(\det \J_t) = \tr(\covt \J_t \circ \J_t^{-1}) = \tr(\V_t).
\end{equation*}
A routine calculation then shows that \eqref{atsoln} is the unique solution to \eqref{atODE}.
\end{proof}

Instead of \eqref{riccatiflat}, the ansatz in \S\ref{SS:ansatz} below makes use of the equations in the following corollary.
\begin{corollary}\label{C:cor4ansatz}
Let $\gamma:t \to \RR$ be a unit speed geodesic and let $c >0$ be a constant.  Suppose $\widetilde\W_{t}$, $\tilde{a}(t)$ are the solutions to \eqref{riccatiraised}, \eqref{atODE} determined by the speed $c$ geodesic $t\mapsto \gamma(c t)$. The equations
\begin{align}
 &c \covt \W_t + (\W_t\circ\Pi_{\dot\gamma(t)}^\perp\circ\W_t )   + c^2 \R_{\dot\gamma(t)}=0, \label{corRiccati}\\
 &c a'(t) + \frac 12\tr (\W_t \circ \Pi_{\dot\gamma(t)}^\perp) a(t) =0, \label{corODE}
\end{align}
are solved by taking $\W_t = \widetilde\W_{t/c}$, $a(t) = \tilde{a}(t/c)$.
\end{corollary}
\begin{proof}
Recall from \eqref{geodrescale} that $t\mapsto\gamma(ct)$ is the unique geodesic with initial position $\gamma(0)$ and initial tangent vector $c\dot\gamma(0)$.   Moreover the operator in \eqref{tidalforce} determined by $\gamma(c t)$ can be expressed in terms of $\gamma(t)$ as $c^2\R_{\gammadot(c t)}$.  Solutions to the equation \eqref{corRiccati} are obtained by rescaling the ODE's satisfied by the matrix elements $\langle \widetilde{W}_t( E_j(t)),E_k(t)\rangle_g$.  Since \eqref{wtdef} implies that $\tr (\widetilde{\W}_t \circ \Pi_{\dot\gamma(t)}^\perp) = \tr(\widetilde{\V}_t)$, rescaling \eqref{atsoln} gives a solution to \eqref{corODE} as well.
\end{proof}

\subsubsection{Matrices associated to $\J_t$ and $\V_t$}\label{SSS:matrices}
Although Theorem \ref{T:mainriccati} is stated for very general choices of initial data $\V_0$, we are interested in the case where $\V_0 = iI$.  We thus consider solutions given by the Jacobi tensor $\J_t$ with initial data $\J_0 = I$, $\covt \J_0 = iI$, as this will yield solutions consistent with the conventions in \cite{CombescureRobertBook}, \cite{FollandPhaseSpace} and other literature.  

We return to the orthonormal frame $E_1(t), \dots, E_{\d-1}(t)$ in \eqref{parallelONframe}, which spans $\gamma^\perp$.  Recall from the proof of Proposition \ref{P:SVD} that we expressed solutions to the Jacobi equation as $J(t) = J^k(t)E_k(t)$ and by setting $L^k(t) = \dot J^k(t)$, we obtain a system of ODEs with a fundamental matrix $\Fsc(t)$.  Moreover, $\Fsc(t)$ is a symplectic matrix satisfying \eqref{fundsymp}.  The following discussion pertains to any fixed time $t$, so we now omit the $t$-dependence in the notation. We express $\Fsc$ in block form
\begin{equation}\label{blocksymplectic}
	\Fsc =  \begin{bmatrix}
		A & B \\
		C & D
	\end{bmatrix},
\end{equation}
with each block a $(\d-1) \times (\d-1)$ matrix.  The blocks $A,C$ (resp. $B,D$) result from solving the linear system with $J(0)=I$, $L(0)= 0$ (resp. $J(0)=0$, $L(0)= I$). 
The crucial properties \eqref{fundsymp} are known to be equivalent to either one of the following lines (see e.g. \cite[Proposition 4.1]{FollandPhaseSpace})
 \begin{equation}\label{sympblockprop}
 \begin{split}
 	A^T D - C^T B = I, \quad A^TC = C^T A, \quad \text{ and } \quad B^T D = D^T B;\\
 	A D^T -  B C^T= I, \quad AB^T = B A^T, \quad \text{ and } \quad CD^T = DC^T.
 \end{split}
 \end{equation} 
 
 We can now use \eqref{blocksymplectic} and the definition of a Jacobi tensor in \eqref{Jacobitensorframe} to express $\J$ and $\V$ at any time $t$ so that the matrices of $\J, \covt\J, \V$ in the frame $E_1,\dots, E_{d-1}$ are 
\begin{equation}\label{matrixrealize}
	[\J] = A+iB, \qquad [\covt \J] = C+iD, \qquad 	[\V] = (C+iD)(A+iB)^{-1}.
\end{equation}
Moreover, restoring $t$ dependence \eqref{atsoln} takes the form
\begin{equation}\label{atmatrix}
	a(t) = a(0)\detexp^{-\frac 12}(A_t+iB_t)
\end{equation}

The analysis above shows that $[\V]$ is a matrix in the \emph{Siegel Upper Half-Plane}, denoted by $\Gf$, the space of $(\d-1)\times (\d-1)$, complex symmetric matrices with positive definite imaginary part.  Indeed, the use of the orthonormal frame means that the expressions in \eqref{keyriccati} are aligned with those for the usual Euclidean dot product.

\subsection{The Siegel Upper Half-Plane and the Siegel Disk}\label{SS:SiegelPlaneDisk}
In this subsection, we consider the inverse matrix $([\V] + i)^{-1}$, where $[\V]$ is defined in \eqref{matrixrealize}.  We are led to this matrix in \S\ref{S:matrixelements} when we treat the phase space kernel of the parametrix in Theorem \ref{T:Vtthm}.  A crucial development will be to characterize this matrix in terms of the singular values of the symplectic transformation $\Fsc$ which birthed $\V$.   To this end, we use elements of the approach in \cite[\S4.5]{FollandPhaseSpace} and also \cite{CombescureRobertBook}, which entails complexifying $\Fsc$.  

Unless we remark otherwise, we do not need to assume that $\Fsc$ is determined by the Jacobi equation and instead just assume this is an arbitrary symplectic matrix $\Fsc$ satisfying \eqref{fundsymp}, \eqref{sympblockprop}.

We use the following from\footnote{The discrepancy between Folland's definition \cite[p.201]{FollandPhaseSpace} and \eqref{alphaactiondef} is rooted in his convention for the Schr\"odinger representation, which is opposite of what is typically desirable for applications on PDE.  This is addressed in his remark on p. 203.  However, our \eqref{alphaactiondef} is consistent with the convention in \cite{CombescureRobertBook}.} \cite[\S4.5]{FollandPhaseSpace} and \cite[Lemma 21]{CombescureRobertBook}.  Given a block matrix as in \eqref{blocksymplectic},  define the following action on a $(\d-1) \times (\d-1)$ matrix $W$ as 
\begin{equation}\label{alphaactiondef}
	\alpha[\Fsc](W) = (C+DW)(A+BW)^{-1}.
\end{equation}
At this stage, the domain of $\alpha(\Fsc)$ is any matrix $W$ for which $A+BW$ is invertible; we do not assume $\Fsc$ is symplectic here as we will apply this to the complexifications of $\Fsc \in \Spdr$ below.  In this case, an algebraic calculation shows the homomorphism property
\begin{equation}\label{homomorphsiegel}
	\alpha[\Fsc_1\Fsc_2](W) = \alpha[\Fsc_1](\alpha[\Fsc_2](W)),
\end{equation}
provided all expressions here are well-defined.  Next we define the determinant 
\begin{equation*}
	\delta(\Fsc,W) = \det(A+BW) 
\end{equation*}
In this case, we have the following cocycle identity, when defined:
\begin{equation}\label{deltacocycle}
	\delta(\Fsc_1\Fsc_2,W) = \delta(\Fsc_1,\alpha[\Fsc_2](W))\delta(\Fsc_2,W).
\end{equation}

In the special case where $\Fsc$ is symplectic, we have that $\alpha[\Fsc]:\Gf \to \Gf$, and in particular $A+BW$ is well defined whenever $W \in \Gf$.  In some sense this is implicit from the proof of Theorem \ref{T:mainriccati}, however, an explicit algebraic argument is given in \cite[Lemma 21]{CombescureRobertBook}.  While we do not reiterate it at this stage, we use a variation on it in proving Proposition \ref{P:alphadisk} below.

We now define the complexification map from $\RR^{2(\d-1)}$ to $\CC^{2(\d-1)}$ as\footnote{The analysis below concerns $x-iX$ rather than its conjugate, hence its appearance as the first entry here.} $(x,X)\mapsto \frac{1}{\sqrt{2}}(x-iX,x+iX)$, that is, the linear map determined by the unitary matrix
\begin{equation*}
	 \mathscr{W}=\frac{1}{\sqrt{2}} 
	\begin{bmatrix}
		I & -i\\
		I & \phantom{-}i
	\end{bmatrix}.
\end{equation*}
Note that 
\begin{equation}\label{widentities}
	\overline{\mathscr{W}}\mathscr{W}^{-1} = 
	\begin{bmatrix}
		0 & I\\
		I & 0
	\end{bmatrix}
	\quad \text{ and } \quad \overline{\mathscr{W}}\Jb \mathscr{W}^{-1} =- i\Jb.
\end{equation}

Given any symplectic matrix $\Fsc$ in block form as in \eqref{blocksymplectic}, conjugating $\Fsc$ by $\mathscr{W}$ yields\footnote{The factor of $1/2$ in the definition of $Y,Z$ is a slight deviation from \cite[p.67]{CombescureRobertBook}.}
\begin{equation}\label{cplxdympldef}
\begin{split}
	&\Fsc_c := \mathscr{W} \Fsc \mathscr{W}^{-1} = 
	\begin{bmatrix}
		Y & \bar Z\\
		Z & \bar Y
	\end{bmatrix}, 	\\ 
		\text{ where }
			&2Y:= A+D + i(B-C), \qquad 2Z: = A-D +i(B+C).
\end{split}
\end{equation}
 Since $\mathscr{W}$ is unitary, we have that $(\Fsc^T)_c = (\Fsc_c)^*$ and hence $(\Fsc_c)^T = \overline{(\Fsc^T)_c}$. 
 We further conjugate the symplectic flip $\Jb$ in \eqref{fundsymp} by $\W$ to define
\begin{equation*}
	\Ksc := -i\mathscr{W} \Jb \mathscr{W}^{-1} = 
	\begin{bmatrix}
		I & 0 \\
		0 & -I
	\end{bmatrix} .
\end{equation*}
Since $\Fsc$ satisfies the fundamental identities \eqref{fundsymp},  conjugation with $\W$ gives
\begin{equation}\label{fundsympcplx}
	\Ksc = \Fsc_c^* \Ksc \Fsc_c \quad \text{ and } \quad \Fsc_c^{-1} = \Ksc \Fsc_c^* \Ksc.
\end{equation}
The identity $\Fsc_c^{-1} \Fsc_c = I$ now gives
\begin{equation}\label{cplxsympID1}
	Y^* Y - Z^*Z = I , \quad Y^T Z = Z^T Y, \quad Y^* \bar Z = Z^* \bar Y.
\end{equation}
In particular, this implies that the singular values of $Y$ are all nonzero, hence $Y$ is invertible.  

We now introduce the \emph{Siegel disk}, denoted by $\Cf$, as the space of $(\d-1) \times (\d-1)$ complex, symmetric matrices $W$ which satisfy $I-W^*W>0$, or equivalently
\begin{equation*}
	\|W\| = \sup \{|Wz| : z \in \CC^{\d-1}, |z|=1\} <1,
\end{equation*}
where both the adjoint and the norm is taken with respect to the standard Hermitian form on $\CC^{d-1}$. 
Just as $\alpha(\Fsc)$ maps $\Gf$ to itself whenever $\Fsc \in \Spdr$, the next proposition shows that $\alpha(\Fsc_c)$ maps $ \Cf$  to itself.
\begin{proposition}\label{P:alphadisk}
	Let $\Fsc_c$ be the complexification of any matrix $\Fsc \in \Spdr$ as in \eqref{cplxdympldef}.  The action $\alpha[\Fsc_c]$ 
	\begin{equation}\label{alphadisk}
		\alpha[\Fsc_c] (W) = (Z+\bar Y W)(Y + \bar Z W)^{-1}, \qquad W \in \Cf,
	\end{equation}
is well-defined map from $\Cf$ to itself, with $Y + \bar Z W$ always invertible.
\end{proposition}
\begin{proof}
	As noted above, the proof is inspired by \cite[Lemma 21]{CombescureRobertBook} concerning the analogous result for $\Gf$. Given $W \in \Cf$, define the $(\d-1)\times (\d-1) $ matrices $R,S$ by block matrix multiplication
	\begin{equation*}
		\begin{bmatrix}
			R \\
			S 
		\end{bmatrix}
		:=
		\Fsc_c \begin{bmatrix}
			I \\
			W 
		\end{bmatrix}
		=
		\begin{bmatrix}
			Y & \bar Z\\
			Z & \bar Y
		\end{bmatrix}
		\begin{bmatrix}
			I \\
			W 
		\end{bmatrix}
		= 
		\begin{bmatrix}
			Y+\bar Z W \\
			Z + \bar Y W 
		\end{bmatrix}.		
	\end{equation*}
	In this notation, $\alpha[\Fsc_c] (W) = SR^{-1}$, provided $R$ is invertible.  Observe that by \eqref{fundsympcplx}, we have that
	\begin{equation*}
		I-W^*W = 
		\begin{bmatrix}
			I & W^* 
		\end{bmatrix}
		\Ksc
		\begin{bmatrix}
			I \\ W
		\end{bmatrix} 
		=		
		\begin{bmatrix}
			I & W^* 
		\end{bmatrix}
		\Fsc_c^* \Ksc \Fsc_c
		\begin{bmatrix}
			I \\ W
		\end{bmatrix} 
		= 
		\begin{bmatrix}
			R^* & S^*
		\end{bmatrix}
		\Ksc
		\begin{bmatrix}
			R\\ S
		\end{bmatrix}
		= R^* R - S^* S		
	\end{equation*}
	Since the left hand side is a positive definite matrix, so is the right hand side.  Thus if there was $z \neq 0$ with $Rz = 0$, we would have the contradictory
	$
		0 < |Rz|^2 - |Sz|^2 = -|Sz|^2.
	$
	Hence $R = Y+\bar Z W$ is invertible.  Moreover, using that $I-W^*W>0$ once more
	\begin{equation*}
		I-(SR^{-1})^*SR^{-1} = R^{-*}(R^*R - S^*S) R^{-1} = R^{-*}(I-W^*W) R^{-1} >0.
	\end{equation*}
	
	It remains to show $SR^{-1}$ is symmetric. By the first identity in \eqref{widentities} and the definition of $\Ksc$
	\begin{equation*}
		0 
		=
				\begin{bmatrix}
			I & W
		\end{bmatrix} 
		\overline{\mathscr{W}}\mathscr{W}^{-1}\Ksc
				\begin{bmatrix}
		I \\ W
	\end{bmatrix} 
	= -i 
	\begin{bmatrix}
		I & W
	\end{bmatrix} 
	\overline{\mathscr{W}}\Jb\mathscr{W}^{-1}
	\begin{bmatrix}
		I \\ W
	\end{bmatrix} 
		= -i 
	\begin{bmatrix}
		I & W
	\end{bmatrix} 
	\overline{\mathscr{W}}\Fsc^T\Jb\Fsc\mathscr{W}^{-1}
	\begin{bmatrix}
		I \\ W
	\end{bmatrix} . 
	\end{equation*}
	We now use that $(\Fsc_c)^T = \overline{\mathscr{W}} \Fsc^T \overline{\mathscr{W}^{-1}}$ and the second identity in \eqref{widentities} to write.
	\begin{equation*}
		\overline{\mathscr{W}}\Fsc^T\Jb\Fsc\mathscr{W}^{-1} =- i(\Fsc_c)^T \Jb \Fsc_c 
	\end{equation*}
	Substituting this in the preceding line implies that 
	\begin{equation*}
		0 = 
		\begin{bmatrix}
			I & W
		\end{bmatrix} 
		 (\Fsc_c)^T \Jb \Fsc_c 
		\begin{bmatrix}
			I \\ W
		\end{bmatrix} 
		=
		\begin{bmatrix}
		S^T & R^T 
		\end{bmatrix}
		\Jb
		\begin{bmatrix}
		R\\ S
		\end{bmatrix} 
		=
		S^T R - R^TS		.
	\end{equation*}
	We now have $R^{-T}S^T = R^{-T}(S^T R)R^{-1} = R^{-T}(R^TS) R^{-1} = SR^{-1}$ as desired.
\end{proof}

\begin{remark}\label{R:cayley}
Momentarily returning to $[\V]$ as in \S\ref{SSS:matrices}, an algebraic calculation reveals that
	\begin{equation}\label{halftodisk}
		([\V] + i)^{-1} = \frac{1}{2i} \big( I + W\big), \qquad W := ZY^{-1},
	\end{equation}
where $Z,Y$ result from the complexification of the fundamental matrix $\Fsc$ of the ODE system.  By applying Proposition \ref{P:alphadisk} in the special case $\alpha[\Fsc_c] (0)$, we see that $W \in \Cf$.  It also follows as in \cite[Lemma 23]{CombescureRobertBook} since \eqref{cplxsympID1} implies that $Y^*Y >0$ so $Y$ is invertible and
	\begin{equation*}
		I-W^*W = (Y^{-1})^*\big( Y^* Y - Z^* Z\big) Y^{-1} = (Y^{-1})^*Y^{-1} >0.
	\end{equation*}
\end{remark}

We now observe special cases of complexification $\Fsc_c$ and its implications for the action $\alpha[\Fsc_c]$.  Our aim is to calculate how this action behaves when $\Fsc_c$ is expressed in terms of a singular value decomposition, see Corollary \ref{C:alphaactionsvd} below. Any omitted proofs follow by routine calculation.
\begin{proposition}\label{P:diagonalSiegel} 
	Suppose $\Gsc$ is a diagonal matrix in the form \eqref{blocksymplectic}, with  
	\begin{equation}\label{diagonalhyp}
		B=C=0, \quad A = \diag(\sigma_1, \dots, \sigma_{\d-1}) = D^{-1}, \quad \text{ where  } \quad  \sigma_1 \geq \cdots \geq \sigma_{\d-1} \geq 1.
	\end{equation}
	Then 
	\begin{equation*}
		\begin{gathered}
					\Gsc_c=
			\begin{bmatrix}
				Y & \bar Z\\
				Z & \bar Y
			\end{bmatrix}, \quad \text{ where } \\
			2Y = \diag\Big(\sigma_1 + \frac{1}{\sigma_1}, \dots, \sigma_{\d-1} + \frac{1}{\sigma_{\d-1}}\Big), \quad 2Z = \diag\Big(\sigma_1 - \frac{1}{\sigma_1}, \dots, \sigma_{\d-1} - \frac{1}{\sigma_{\d-1}}\Big).			
		\end{gathered}
	\end{equation*}	
	Moreover, taking $W =0$ in $\Cf$, then 
	\begin{equation*}
		\alpha[\Gsc_c]0 = \diag\Big(\frac{\sigma_1 - \sigma_1^{-1}}{\sigma_1 + \sigma_1^{-1}}, \dots, \frac{\sigma_{\d-1} - \sigma_{\d-1}^{-1}}{\sigma_{\d-1} + \sigma_{\d-1}^{-1}}\Big) \quad \text{ and } \quad \delta(\Gsc_c,0) = \prod_{j=1}^{\d-1} \Big(\sigma_j + \frac{1}{\sigma_j} \Big). 
	\end{equation*}
\end{proposition}

Next we recall properties of $\Spdr \cap \Odr$, where $\Odr$ is the orthogonal group on $\RR^{2(\d-1)}$.  If $\Hsc$ lies in this intersection, then combining this with the fundamental relations \eqref{fundsymp} with $\Hsc^{-1} = \Hsc^T$ gives $\Hsc \Jb = \Jb \Hsc$, and hence it is seen that $\Hsc$ must take the block form 
\begin{equation}\label{symportho}
	\Hsc = 
	\begin{bmatrix}
		 A & B \\
		- B &  A 
	\end{bmatrix},
\end{equation}
 satisfying any one of the three sets of equivalent conditions
\begin{equation}\label{symporthoprop}
	\begin{aligned}
	 &A^T A +  B^T  B = I \text{ and }  A^T B  =  B^T A; \\
	 &A A^T +  B B^T = I \text{ and }  A B^T  =  B A^T;\\
	 &A + iB \in \Udr \text{ (the unitary group on }\CC^{\d-1}).
	\end{aligned}
\end{equation}

\begin{proposition}\label{P:unitarySiegel}
Suppose $\Hsc \in \Spdr \cap \Odr$ is as in \eqref{symportho}, \eqref{symporthoprop}.  Then 
\begin{equation*}
	\Hsc_c=
	\begin{bmatrix}
		 V & 0\\
		0 & \bar{V}
	\end{bmatrix}, \quad \text{ where } \quad  V =  A + i B \in \Udr.
\end{equation*}	
Consequently, if $ W \in \Cf$, then $\alpha[\Hsc_c] (W) = \bar{V}  W  V^{-1}$ and $\delta[\Hsc_c](W) = \det(V)$ so that 
\begin{equation}\label{deltaunitarySiegel}
	|\delta(\Hsc_c, W)| =1 .
\end{equation}  
In particular, taking $ W=0$ in $\Cf$, $\alpha[\Hsc_c](0)=0$.
\end{proposition}

\begin{corollary}\label{C:alphaactionsvd}
	Suppose $\widetilde{\Hsc}, \Hsc \in \Spdr\cap\Odr$ and that $\Gsc$ is a diagonal matrix satisfying \eqref{diagonalhyp}. Then 
	\begin{equation}\label{siegelsvd}
		\alpha[\widetilde{\Hsc}_c\Gsc_c\Hsc_c](0) = \bar V \diag\Big(\frac{\sigma_1 - \sigma_1^{-1}}{\sigma_1 + \sigma_1^{-1}}, \dots, \frac{\sigma_{\d-1} - \sigma_{\d-1}^{-1}}{\sigma_{\d-1} + \sigma_{\d-1}^{-1}}\Big) V^{-1}.
	\end{equation}
	where $V$ is the matrix which results from applying Proposition \ref{P:unitarySiegel} to the first matrix $\widetilde{\Hsc}_c$. Moreover, 
	\begin{equation}\label{deltasvd}
		|\delta(\widetilde{\Hsc}_c\Gsc_c\Hsc_c,0)| = |\delta(\Gsc_c,0)| =  \prod_{j=1}^{\d-1} \Big(\sigma_j + \frac{1}{\sigma_j} \Big). 
	\end{equation}
\end{corollary}
\begin{proof}
	By Proposition \ref{P:unitarySiegel}, we have $\alpha[\Hsc_c] (0) =0$.  Consequently, \eqref{siegelsvd} is a consequence of repeated use of the homomorphism property \eqref{homomorphsiegel} for $\alpha$ and Propositions \ref{P:diagonalSiegel}, \ref{P:unitarySiegel}:
	\begin{equation*}
	\alpha[\widetilde{\Hsc}_c\Gsc_c\Hsc_c] (0)= \alpha[\widetilde{\Hsc}_c\Gsc_c](\alpha[\Hsc_c] (0) ) = \alpha[\widetilde{\Hsc}_c\Gsc_c](0 ) =   \alpha[\widetilde{\Hsc}_c]\big(\alpha[\Gsc_c](0)\big) .
	\end{equation*}
	The second identity \eqref{deltasvd} then follows similarly from \eqref{deltacocycle} and the preceding propositions.
\end{proof}
\begin{remark}\label{R:alphaactionsvd}
	This last corollary now allows to calculate $([\V]+i)^{-1}$ as in \eqref{matrixrealize} in terms of the singular value decomposition of the fixed time fundamental matrix $\Fsc$ which determined it.  Indeed, \eqref{symplecticsvd} implies that we can write $\Fsc = \widetilde{\Hsc}\Gsc \Hsc$ where $\widetilde{\Hsc}, \Hsc \in \Spdr\cap\Odr$ while arranging $\Gsc$ to satisfy \eqref{diagonalhyp}. Moreover, \eqref{halftodisk}, \eqref{alphadisk} show that $\alpha[\Fsc_c](0) = ZY^{-1} = 2i(([\V]+i)^{-1}-I)$ and we also have $\delta(\Fsc_c,0) = Y$.
\end{remark}
\subsection{Regularity of solutions}\label{SS:RiccatiReg}
In this section, we examine bounds on solutions to the Riccati equation.  We treat this by proving a general result concerning the matrix $(C+iD)(A+iB)^{-1}$ formed by any $\Fsc \in \Spdr$ in block form \eqref{blocksymplectic}.  

We use the usual matrix norm
\begin{equation}\label{matrixnorm}
	\|E\| = \max_{|z| = 1} |Ez| = \max\{\sqrt{\rho}: \rho \in \spec(E^*E)\} =  \max\{\sqrt{\rho}: \rho \in\spec(EE^*) \},
\end{equation}
where $z$ is real or complex valued depending on whether or not the matrix is. Here $\spec$ denotes the spectrum and the last two identities observe that the matrix norm is the largest singular value. 
 In the case where $E$ is symmetric positive definite, it is diagonalizable by a unitary transformation and $\|E\|$ is just the largest eigenvalue.  Expressing $[\V_t]$ in terms of an orthonormal basis with respect to the metric $g$ as above, the results here imply bounds on $\|\V_t\|$ calculated in terms of the metric. 

Begin by observing that if $\Fsc \in \Spdr$ then $AB^T, CD^T$ are symmetric (cf. \eqref{sympblockprop}), hence
\begin{equation}\label{conjugateidentities}
	\begin{split}
	&(A+iB)(A-iB)^T = AA^T + BB^T = (A-iB)(A+iB)^T,\\
	&(C+iD)(C-iD)^T = CC^T + DD^T = (C-iD)(C+iD)^T.
	\end{split}
\end{equation}
Consequently, algebraic calculation reveals the real and imaginary parts of $(C+iD)(A+iB)^{-1}$
\begin{equation}\label{realimagsiegel}
	\begin{split}
	\Re \big[(C+iD)(A+iB)^{-1} \big]& = (CA^T +DB^T)(AA^T + BB^T)^{-1},\\
	\Im \big[ (C+iD)(A+iB)^{-1} \big] & = (AA^T + BB^T)^{-1}.		
	\end{split}
\end{equation}

\begin{proposition}\label{P:riccatireg}
	Suppose that $\Fsc \in \Spdr$ is in block form \eqref{blocksymplectic} and that 
	$
		\|\Fsc\| \leq \mu.
	$
	The matrix $\tilde\omega = (C+iD)(A+iB)^{-1}=\alpha[\Fsc](iI) $ satisfies the following bounds 
	\begin{gather}
		\|\tilde\omega\| \leq \mu^2\label{siegelbd1}\\
		\big\| \Im  \tilde\omega  \big\| = \|(A+iB)^{-1} \|^2 \leq \mu^2\label{siegelbd2},\\
	\big\| \big(\Im  \tilde\omega  \big)^{-1}\big\| \lesssim \mu^2 \quad\text{ and } \quad(\Im\tilde\omega) y \cdot y \geq \mu^{-2}|y|^2  \text{ for all } y \in \RR^{d-1} \label{siegelbd5}\\
		\left\|\left[ \Im ( \tilde\omega +iI)\right]^{-1} \right\| \leq 1 \label{siegelbd3},\\
		\left\|( \tilde\omega +iI)^{-1} \right\| \leq 1 \label{siegelbd4}.
	\end{gather}	
\end{proposition}

\begin{proof}
	It is well-known that if $\rho \in \spec(\Fsc)$, then $\rho^{-1}, \bar \rho, \bar \rho^{-1}\in \spec(\Fsc)$, so $\spec(\Fsc) = \spec(\Fsc^{-1})$. Moreover, symplectic matrices are closed under transposition, inversion, and composition operations.  Hence $\spec(\Fsc\Fsc^T) = \spec(\Fsc^{-T}\Fsc^{-1}) \subset (0,\infty)$ as these matrices are positive definite.  This implies that we must have $\mu \geq 1$ and $\|\Fsc^{-1}\| = \|\Fsc^T\| = \|\Fsc\| \leq \mu$.  Moreover,
	\begin{equation*}
		\min \spec(\Fsc\Fsc^T)= \frac{1}{\max \spec(\Fsc\Fsc^T)} \geq \mu^{-2}
	\end{equation*}
	Now observe that
	\begin{equation*}
		\Fsc \Fsc^T =\begin{bmatrix}
			A & B \\
			C & D
		\end{bmatrix}
		\begin{bmatrix}
			A^T & C^T \\
			B^T & D^T
		\end{bmatrix} =
		\begin{bmatrix}
			AA^T + BB^T & AC^T+BD^T \\
			CA^T+DB^T & CC^T + DD^T
		\end{bmatrix}
	\end{equation*}
	Hence 
	\begin{equation*}
		\mu^{-2} \leq \min\spec(\Fsc\Fsc^T) \leq \min\spec (AA^T + BB^T) \leq \max\spec (AA^T + BB^T) \leq \max\spec(\Fsc\Fsc^T) \leq \mu^2
	\end{equation*}
	where the second (respectively fourth) inequality follows by comparing the minimum (resp. maximum) of  $\Fsc\Fsc^T z \cdot z$ over all $|z| = 1$ to the minimum (resp. maximum) obtained by constraining this quantity to vectors whose last $\d-1$ entries vanish. Since $(\Im \tilde\omega)^{-1} = AA^T+BB^T$ by \eqref{realimagsiegel}, the bound  \eqref{siegelbd5} now follows from \eqref{matrixnorm}, with the second half following since this gives $|y|^2 \lesssim \mu^2|(\Im\tilde\omega)^{\frac 12}y|^2$.
	We can also deduce \eqref{siegelbd2} by using that \eqref{conjugateidentities} implies
	\begin{equation*}
		\|(A+iB)^{-1} \|^2 = \left\|(AA^T + BB^T)^{-1}\right\| \leq 
		\frac{1}{\min\spec (AA^T + BB^T)} \leq \mu^2.
	\end{equation*}
	
	To see \eqref{siegelbd1}, observe that \eqref{conjugateidentities} implies we can argue similarly to the above to get that $\max\spec (CC^T + DD^T) \leq \max\spec(\Fsc\Fsc^T)$ and hence $ \|C+iD\| \leq \mu$ again by \eqref{matrixnorm}, at which point \eqref{siegelbd1} follows in combination with the above.
	
To see \eqref{siegelbd3}, we use that 
\begin{equation*}
	\left(\Im( (C+iD)(A+iB)^{-1} +i)\right)^{-1}= \left((AA^T+ BB^T)^{-1}+I\right)^{-1} = (AA^T+ BB^T)(I+AA^T+ BB^T)^{-1}
\end{equation*}
The spectral mapping theorem and $AA^T+ BB^T>0$ now imply
\begin{align*}
	\left\| \left(\Im( (C+iD)(A+iB)^{-1} +i)\right)^{-1} \right\|= \max\left\{ \frac{\rho}{1+\rho}: \rho \in \spec(AA^T+ BB^T)\right\} \leq 1.
\end{align*}

Finally, \eqref{siegelbd4} follows from \eqref{halftodisk}, Corollary \ref{C:alphaactionsvd}, and the ensuing Remark \eqref{R:alphaactionsvd}.  Indeed, taking the singular value decomposition of $\Fsc$ as in \eqref{symplecticsvd}, we see that 
\begin{equation*}
	\big( (C+iD)(A+iB)^{-1} +i\big)^{-1} = \frac{1}{2i} (I+W),
\end{equation*}
where $W$ has singular values which are all bounded above by 1.
\end{proof}

\section{Stationary phase for complex phases}\label{S:statphase}
In this section, we consider the method of stationary phase applied to complex-valued phases. This dates back to the work of Melin and Sj\"ostrand \cite{MelinSjostrand}.  The main difference with the real-valued case is that even under a typical set of hypotheses such as \eqref{phasehyp} below, nearby critical points can ``turn complex'' in a manner so that they affect asymptotics in nontrivial way.   To handle this, \cite{MelinSjostrand} introduced the notion of an \emph{almost analytic} function, which we take to mean a $C^\infty$ function $\tilde{a}: \CC^\d \to \CC$ such that $|\bar\prtl_{z_j} \tilde a(z)| \lesssim_{a,N} |\Im z|^N$ for all $ j \in \{1,\dots,\d\}$ and $N \in \mathbb{N}$.  In this section, but not others, we reserve $z = x+iy$ as a variable in $\CC^\d$ and interpret $\prtl_{z_j} = \frac 12(\prtl_{x_j}-i\prtl_{y_j})$, $\bar\prtl_{z_j} = \frac 12(\prtl_{x_j}+i\prtl_{y_j})$ in the usual way.

Given a Schwartz class function $a$ on $\RR^\d$, it is known that there exists an almost analytic extension $\tilde a$ to all of $\CC^\d$.  While this is proved in \cite{MelinSjostrand}, these can also be defined by taking
\begin{equation*}
	\tilde a(z)  = (2\pi)^{-\d} \int_{\RR^{\d}}e^{i\xi\cdot z} \psi(z,\xi)\hat{a}(\xi)\,d\xi, \qquad \psi(z, \xi) := \prod_{j=1}^\d   \tilde{\psi}(\Im z_j)\tilde{\psi}(\xi_j\,\Im z_j)= \prod_{j=1}^\d   \tilde{\psi}(y_j)\tilde{\psi}(\xi_j\,y_j)
\end{equation*}
where $\hat{a}(\xi)$ is the classical Fourier transform of $a$ and $\tilde{\psi}$ is a bump function on $\RR$ identically one in a neighborhood of the origin (so that $\tilde a|_{\RR^\d} = a$).  This is the higher dimensional version of the formula in \cite[Theorem 3.6]{ZworskiSemiclassicalAnalysis} and the proof there shows the more precise estimate
\begin{equation}\label{aachar}
	|\bar\prtl_{z_k} a(z)| \lesssim_N |\Im z|^N \sum_{|\alpha| \leq N+1 +\lceil \frac{\d+1}{2} \rceil} \|\prtl^\alpha a\|_{L^2}  \quad  \text{ for all } k \in \{1,\dots,\d\} \text{ and } N \in \mathbb{N}.
\end{equation}

We claim that by adapting the proof of \eqref{aachar}, we also have that
\begin{equation}\label{aacharL1}
	\int_{\RR^d} |\bar\prtl_{z_k} \tilde a(x+iy)| \,dx \lesssim |y|^N \big( \|a\|_{L^1} + \|\prtl_{\tilde x_k}^{N+1} a\|_{L^1} \big) \quad  \text{ for all } k \in \{1,\dots,\d\} \text{ and } N \in \mathbb{N}.
\end{equation}
To see this, we set
\begin{equation*}
K_{k,N}(x,y,\tilde x) = (-i)^{N+1}y_k^N \int e^{i\xi\cdot(x-\tilde x+iy)}\tilde{\psi}(y_k)\Big(\frac{\tilde{\psi}'(\xi_k\,y_k)}{(\xi_k\,y_k)^N}\Big)\prod_{j\neq k}\tilde{\psi}(y_j)\tilde{\psi}(\xi_j\,y_j) \,d\xi,
\end{equation*}
and since $\psi$ is identically 1 near the origin, the function in parentheses is $C_c^\infty$. We can now write
\begin{equation}
	\begin{split}
		\bar\prtl_{z_k} \tilde a(x+iy) &=  \int e^{i\xi\cdot(x+iy)}\hat a(\xi) \big(\tilde{\psi}(y_k)\tilde{\psi}'(\xi_k\,y_k)\xi_k + \tilde{\psi}'(y_k)\tilde{\psi}(\xi_k\,y_k)\big)
		\prod_{j\neq k}\tilde{\psi}(y_j)\tilde{\psi}(\xi_j\,y_j) \,d\xi
		\\
		&=\int_{\RR^d} K_{k,N}(x,y,\tilde x) \prtl_{\tilde x_k}^{N+1} a (\tilde x) \,d\tilde x
		+ \int e^{i\xi\cdot(x+iy)}\hat a(\xi)  \tilde{\psi}'(y_k)\tilde{\psi}(\xi_k\,y_k)
		\prod_{j\neq k}\tilde{\psi}(y_j)\tilde{\psi}(\xi_j\,y_j) \,d\xi.
	\end{split}
\end{equation}
By adjusting the arguments we use for the first term, it is routine to see that the $L^1$ norm of the second term is dominated by the right hand side of \eqref{aacharL1}, as $\tilde\psi'(y_k)$ is supported where $|y_k| \approx 1$.  For the first term, integration by parts with respect to $\xi$ shows that
\begin{equation*}
	|K_{k,N}(x,y,\tilde x)| \lesssim |y|^N(1+|x-\tilde x|)^{-d-1},
\end{equation*}
at which point \eqref{aacharL1} follows by Young's inequality.

In addition to \eqref{aacharL1}, we also have
\begin{equation}\label{aaprop}
	\begin{split}
		&\supp(\tilde{a}) \subset \{z: \Im z_j \in \supp(\tilde{\psi}), j=1,\dots \d\}\\
		&\int_{\RR^\d} |\prtl_x^\alpha \tilde{a}(x+iy)|^2\,dx \lesssim_\alpha \|a\|_{L^2}^2 \text{ for any } \alpha \text{ and } y \in \RR^\d.
	\end{split}
\end{equation}

The stationary phase asymptotics in \cite[\S2]{MelinSjostrand} assume that $\phi(w,x): U \to \CC$ is a $C^\infty(U)$ phase function  over some neighborhood $U \subset \RR^n \times \RR^d$ of the origin satisfying 
\begin{equation}\label{phasehyp}
\Im \phi \geq 0, \quad \phi(0,0) = 0, \quad
		d_x \phi(0,0) = 0, \quad \text{ and } \quad d_x^2 \phi(0,0) \text{ is nonsingular}.
\end{equation}
Now let $\tilde\phi$ be an almost analytic extension of $\phi$ (over the support of the amplitude $a$ below).  The last of these properties and the implicit function theorem implies that there is a $C^\infty$ function $w \mapsto Z(w)$ defined in a neighborhood of the origin such that
\begin{equation*}
	\prtl_{z_j} \phi(w,Z(w)) = 0\quad \text{ for all } j \in \{1,\dots,\d\}.
\end{equation*}
The stationary phase asymptotics in \cite{MelinSjostrand} imply that if $a(w,x)$ is a $C^\infty$ function of sufficiently small enough compact support in a neighborhood of $(0,0)$, then
\begin{multline}\label{pureasympexpn}
	\left|\int e^{\frac ih \phi(w,z) } a(w,z) \, dz - (2\pi h)^{\frac d2}\detexp^{-\frac 12} \Big(\frac 1i d_z^2\phi\big(w,Z(w)\big)\Big) e^{\frac ih \phi(w,Z(w))}\sum_{j=0}^{N } h^j 
	A_j(w,D_z) \tilde a\big|_{z=Z(w)}\right|
	\\
	\leq C_{\phi,N} h^{N+1} \| a \|_{C^{J(N)}},
\end{multline}
where the $A_j$ are differential operators of order $2j$ depending on the phase function with $A_0 = 1$. Here $J(N)$ is a sufficiently large order of differentiability depending on $N$.  

By itself, \eqref{pureasympexpn} is useful to us in Appendix \ref{A:wptpdo}, but in the proof of Theorem \ref{T:matrixeltasymp}, we will make use of a more careful version of stationary phase asymptotics.  Namely, since solutions to the complex Riccati equation as in \S\ref{S:riccati} do not yield quadratic phases which satisfy uniform bounds, we use an approach that ensures we have satisfactory asymptotics nonetheless.  Curiously, we do not use the full strength of the hypotheses on $C$ in \eqref{spphasehyp} below.

\begin{theorem}\label{T:spquad}
	Let $C$ be a $\d \times \d$ complex, symmetric matrix with $\Im C>0$ and let $\xi \in \CC^\d$, $b \in\CC$ be fixed.  For $z \in \CC^\d$, let $\phi(z)$ be the complex polynomial
	\begin{equation*}
		\phi(z) =  b + C\xi \cdot z + \frac 12 Cz\cdot z = b-\frac 12 C\xi\cdot\xi + \frac 12 C(z+\xi)\cdot (z+\xi)
	\end{equation*}
	Suppose that for some $0<\veps <\frac 12$,
	\begin{equation}\label{spphasehyp}
			\Im \phi(x) \geq 0 \text{ for all } x \in \RR^d \quad \text{ and } \quad
			\|C^{-1}\| ,
			\|(\Im C)^{-1}\| \lesssim h^{-1+\veps}. 
	\end{equation}
	Given $a \in C_c^\infty(\RR^\d)$, let $\tilde{a}$ be an almost analytic extension to $\CC^\d$ satisfying \eqref{aachar}, \eqref{aaprop}.  Then 
	\begin{multline}\label{asympexpn}
		\left| \int_{\RR^\d}e^{\frac ih \phi(x)} a(x)\,dx - \frac{e^{\frac ih(b-\frac 12 C\xi\cdot\xi)}(2\pi h)^{\frac d2}}{\det^{1/2}(-iC)} \sum_{j=0}^{k-1} \frac{1}{j!} \Big(\frac{h}{2i}\Big)^j \big(C^{-1} D_x \cdot D_x \big)^j \tilde{a}(x-\xi)\Big|_{x=0} \right|\\
		\lesssim_{k,N,\veps}
		h^{\frac{\d}{2} +k} \|C^{-1}\|^{\frac{\d}{2} + k} \sum_{|\alpha| \leq 2k + \lceil \frac{d+1}{2} \rceil} \|\prtl^\alpha a\|_{L^2} + \Big(h^{\frac 12}\big\|C_2^{-\frac 12}\big\|\Big)^N \big( \|a\|_{L^1} + \max_k \|\prtl_{\tilde x_k}^{N+1} a\|_{L^1} \big) .
	\end{multline}	
\end{theorem}

\begin{proof}
	For convenience, we decompose into real and imaginary parts in the following way: $C = C_1 + iC_2, b= b_1 + ib_2, \xi = \xi_1+ i\xi_2$.  
	Consider the real vector $x_0 \in \RR^\d$ defined by  
	\begin{equation*}
	x_0 := -C_2^{-1}(\Im (C\xi)) = -C_2^{-1}(C_1\xi_2+C_2\xi_1)
	\end{equation*}
	Since $0 \leq \Im \phi(x_0) = b_2-\frac 12 C_2x_0 \cdot x_0$, a tedious calculation gives
	\begin{equation*}
		2b_2 \geq  C_2 x_0\cdot x_0 
		= C_2^{-1}(C_1\xi_2 + C_2 \xi_1)\cdot (C_1\xi_2 + C_2 \xi_1) = C_2^{-1}(C_1\xi_2 )\cdot (C_1\xi_2 ) + 2C_1\xi_1\cdot \xi_2 + C_2\xi_1 \cdot \xi_1.
	\end{equation*}
	Using that $\Im(C\xi\cdot\xi) = C_2 \xi_1\cdot \xi_1 - C_2\xi_2\cdot \xi_2 + 2C_1 \xi_1\cdot \xi_2$, we can rewrite this as
	\begin{equation}\label{posimag}
		b_2 - \frac 12 \Im(C\xi\cdot\xi) \geq \frac 12 C_2^{-1}(C_1\xi_2 )\cdot(C_1\xi_2 )+ \frac 12 C_2\xi_2\cdot\xi_2.
	\end{equation}
	
	Let $U$ denote the $\d+1$-chain defined as the image of $\RR^\d \times [0,1]$ under $(x,t) \mapsto x-t\xi$.  The boundary is the $\d$-chain given as the  image of $\RR^\d$ under $x \mapsto x-t\xi$ for $t=0,1$. Along the chain $U$
	\begin{equation*}
		\phi(x-t\xi) = b -\frac 12 C\xi\cdot\xi + \frac 12 C(x+(1-t)\xi) \cdot (x+(1-t)\xi) .
	\end{equation*}
	Using \eqref{posimag} we get that for $t \in [0,1]$ and $x \in \RR^d$,
	\begin{equation}\label{deformcontrol}
		\begin{aligned}
			\Im \phi(x-t\xi) &\geq 
			 \frac t2 (2-t)\big(C_2^{-1}(C_1\xi_2 )\cdot (C_1\xi_2 ) +  C_2\xi_2\cdot \xi_2\big) \\
			 &\phantom{\geq}+ \frac 12 C_2 \big(x+(1-t)\big(\xi_1 +C_2^{-1}(C_1\xi_2)\big)\big)\cdot \big(x+(1-t)\big(\xi_1 +C_2^{-1}(C_1\xi_2)\big)\big)\\
			&  \geq \frac t2  C_2\xi_2\cdot \xi_2 = \frac t2 \big|C_2^\frac 12\xi_2\big|^2.
		\end{aligned}
	\end{equation}
	
	Since $\bar\prtl_{z_j}\phi(z)\equiv 0$ for all $j$, the generalized Stokes' Theorem implies
	\begin{equation}\label{stokesy}
		\int_{\RR^\d}e^{\frac ih \phi(x -\xi )} a(x -\xi)\,dx - \int_{\RR^\d}e^{\frac ih \phi(x)} a(x)\,dx =
		\sum_{k=1}^\d  \int _U e^{\frac ih \phi(z)} \bar{\prtl}_{z_k} \tilde a(z) d\bar{z}_k \wedge dz_1 \wedge dz_2 \wedge \cdots dz_\d.
	\end{equation}
	Note that $|\bar{\prtl}_{z_k} \tilde a(x-t\xi)| \lesssim_N (t|\xi_2|)^N$ and hence for $t \in [0,1]$, \eqref{deformcontrol} shows that 
	\begin{equation*}
		\begin{split}
		\int_{\RR^d} | e^{\frac ih \phi(x-t\xi)} \bar{\prtl}_{z_k} \tilde a(x-t\xi)| \,dx &\leq \Big(\sup_{x\in \RR^d} (t^{\frac 12}|\xi_2|)^Ne^{- \frac 1h \Im\phi(x-t\xi)} \Big) \times (t|\xi_2|)^{-N} \int |\tilde a(x-t\xi)|\,dx  \\
		& \lesssim \Big(h^{\frac 12}\big\|C_2^{-\frac 12}\big\|\Big)^N \big( \|a\|_{L^1} + \|\prtl_{\tilde x_k}^{N+1} a\|_{L^1} \big)
		\end{split}
	\end{equation*}
	where we used \eqref{aacharL1} and that $|\xi_2| \lesssim \|C_2^{-\frac 12}\||C_2^\frac 12\xi_2|$ and $t \leq t^{\frac 12}$. Hence by taking $N$ large enough, the right hand side in \eqref{stokesy} is dominated by the last term on the right hand side in \eqref{asympexpn}.
	
	We conclude by observing that 
	\begin{equation*}
		\int_{\RR^\d}e^{\frac ih \phi(x -\xi )} \tilde a(x -\xi)\,dx = e^{\frac ih(b-\frac 12 C\xi\cdot \xi)} \int_{\RR^{\d}} e^{\frac ih Cx\cdot x} \tilde a(x-\xi)\,dx
	\end{equation*}
	Hence the rest of the proof follows from standard asymptotics for purely quadratic phases as in \cite[Lemma 7.7.3]{HormanderI}, namely
	\begin{multline*}
		\left| \int_{\RR^{\d}} e^{\frac ih Cx\cdot x} \tilde a(x-\xi)\,dx - \frac{ (2\pi h)^{\frac d2}}{\det^{1/2}(-iC)} \sum_{j=0}^{k-1} \frac{1}{j!} \Big(\frac{h}{2i}\Big)^j \big( C^{-1} D_x,D_x \big)^j \tilde{a}(x-\xi)\Big|_{x=0}\right| \\\lesssim_k h^{\frac{\d}{2} +k} \|C^{-1}\|^{\frac{\d}{2} + k} \sum_{|\alpha| \leq 2k + \lceil \frac{d+1}{2} \rceil} \|\prtl^\alpha \tilde a\|_{L^2}
	\end{multline*}
where the implicit constant only depends on $k,d$. The rest of \eqref{asympexpn} now follows from \eqref{aaprop}.
\end{proof}

\section{The Wave Packet Transform}\label{S:wpt}
We begin by recalling our definition from \S\ref{S:Outline} of a wave packet transform adapted a smooth bump function $\tilde\beta \in C_c^\infty(0,\infty)$, along with the expression for its adjoint: 
\begin{equation}\label{wptdef}
	\begin{split}
			\Ssc f (x,\xi)&= 2^{-\frac d2} (\pi h)^{-\frac{3d}{4}}\frac{\tilde\beta\big(|\xi|_{g(x)}\big) }{\detexp^{\frac 14}(g_{jk}(x))}\int_{M} e^{\frac ih(- \xi(\expyno) + \frac i2 d^2(x,y))} \psi(x,y)  f(y)\,dv_g,
		\\
			\Ssc^*G(y) &= 2^{-\frac d2} (\pi h)^{-\frac{3d}{4}}\int_{T^* M} e^{\frac ih(\xi(\expyno) + \frac i2 d^2(x,y))} \frac{ \psi(x,y) \tilde\beta\big(|\xi|_{g(x)}\big) }{\detexp^{\frac 14}(g_{jk}(x))}G(x,\xi) \,dx d\xi.
	\end{split}
\end{equation}
As noted in \S\ref{S:Outline},  \cite[Proposition 3.1]{WunschZworskiFBI} proves the following, with another proof in Appendix \ref{A:wptpdo}.
\begin{proposition}\label{P:wptpdo}
	Suppose $\Ssc$ is a wave packet transform adapted to $\tilde\beta$.  Then $\Ssc \circ \Ssc^*$ is a semiclassical pseudodifferential operator whose symbol admits an asymptotic expansion $\sum_{j\geq 0} h^js_j(x,\xi)$ such that the principal symbol is $s_0(x,\xi) = \tilde\beta^2(|\xi|_{g(x)})$.
\end{proposition}

Recall from \S\ref{S:Outline} that the bump function $\tilde\beta$ in \eqref{wptdef} is taken so that $\tilde \beta|_{\supp(\beta)} \equiv 1$ (and hence $\tilde \beta \beta = \beta$) and with $\supp(\tilde\beta) \subset (-1+\veps_0,1+\veps_0)$.  We now prove Proposition \ref{P:invwpt1} claimed there.
\begin{proposition}\label{P:invwpt}
	There exists a semiclassical PDO $P$ such that $\beta^2(h\sqrtl) = (\Ssc^* \circ \Ssc) \circ P + \mathcal{O}(h^\infty)$, that is, for any $N$
	\begin{equation*}
		\big\|\beta^2(h\sqrtl)-(\Ssc^* \circ \Ssc) \circ P\big\|_{L^2(M) \to L^2(M)} \lesssim_N h^N
	\end{equation*}
	The symbol of $p$ can be taken to satisfy $\supp(p) \subset \{(x,\xi) \in T^*M: |\xi|_{g(x)} \in \supp(\beta)\}$ and lies in a uniformly bounded subset of $C^\infty$ for all $h>0$ sufficiently small.
\end{proposition}
\begin{proof}
In what follows, we use the notation in Proposition \ref{P:wptpdo} for the symbol of $\Ssc \circ \Ssc^*$. By a result of Strichartz \cite{StrichartzFunctionalCalculus} ,  $\beta^2(h\sqrtl)$ is a PDO.  The approach of Taylor \cite[p.296-7]{TaylorPDO} uses stationary phase to show that its symbol has an asymptotic expansion  $\sum_{j \geq 0} h^j a_j(x,\xi)$ where $a_0(x,\xi) = \beta^2(|\xi|_{g(x)})$.
Moreover, each $a_j$ can be taken so that $\supp(a_j) \subset \supp(a_0)$, in other words, the support of each $a_j $ is contained in the support of the principal symbol. 

We thus define the symbol of $P$ by an asymptotic series $ \sum_{j \geq 0} h^j p_j(x,\xi)$.  By the composition calculus \cite[Theorem 4.12]{ZworskiSemiclassicalAnalysis}, the symbol of $ (\Ssc^* \circ \Ssc) \circ P $ has the formal asymptotic expansion
\begin{equation*}
	\sum_{\ell \geq 0} \frac{i^\ell h^\ell}{\ell!} \big( D_\xi \cdot D_y \big)^\ell \bigg( \sum_{j \geq 0} h^j s_j(x,\xi) \bigg)\bigg(\sum_{k \geq 0} h^k p_k(y,\eta)\bigg)\Bigg|_{(x,\xi) = (y,\eta)}.
\end{equation*}
We rewrite this series as
\begin{equation*}
	\sum_{m\geq 0} h^m \tilde\beta^2 p_m + \sum_{m\geq 1}h^m b_m, \quad \text{ where } \quad b_m = \sum_{k = 0}^{m-1} \sum_{j+\ell = m-k} \big( D_\xi \cdot D_y \big)^\ell (s_j p_k )\Big|_{(x,\xi) = (y,\eta)}.
\end{equation*}
Since $s_0(x,\xi) = \tilde\beta^2(|\xi|_{g(x)})$ and $\tilde \beta \beta = \beta $ we set $p_0 = \beta^2(|\xi|_{g(x)})$ so that the principal symbols of $	\beta^2(h\sqrtl)$ and $(\Ssc^* \circ \Ssc) \circ P $ are identical.  The above shows that the remaining symbols must satisfy $a_m = \tilde\beta^2 p_m + b_m$, where $b_m$ depends only on $p_0,\dots, p_{m-1}$.  We define $p_m$ for $m \geq 1$ recursively by taking $p_m = a_m - b_m$ as this gives $\supp(p_m) \subset \supp(p_0)$ so that $\tilde\beta^2 p_m = p_m$.
\end{proof}

Recall from \eqref{intromainreduction}, that we are led to consider the compositions $P\circ Q_{h,\tilde\gamma}^*$ and $Q_{h,\tilde\gamma}\circ P^*$.  It is convenient to realize these compositions as a single PDO.  For the remainder of this section, we omit the dependence on $h,\tilde\gamma$ and consider their regularity in Fermi coordinates.  We will generally consider symbols $q$ satisfying 
 \begin{equation}\label{mlcutbds}
 	\begin{split}
 		|\prtl_{y'}^{\alpha} \prtl_{y_d}^j \prtl_{\eta'}^\beta\prtl_{\eta_d}^k q(y,\eta)| &\lesssim_{\alpha,j,\beta,k,N}  
 		h^{-\frac 12(|\alpha|+|\beta|)} \big(1+ h^{-\frac 12}|y'| + h^{-\frac 12}|\eta'|  \big)^{-N}, \\
 		\supp(q) &\subset \{(y,\eta): |y| \ll 1, |\eta -e_d| \ll 1 \}.
 	\end{split}
 \end{equation}
\begin{proposition}\label{P:composePDO}
Let $P$ be the semiclassical PDO furnished by Proposition \ref{P:invwpt}.  Suppose $q(y,\eta)$ is a symbol satisfying \eqref{mlcutbds}. If $Q$ is the PDO with symbol $q(y,\eta)$, then the composition $ Q \circ P^*$ is a PDO with symbol $\tilde{q}_1(y,\eta)$ satisfying the same regularity as in \eqref{mlcutbds}.  Moreover, the further composition $ ( P \circ Q^*) \circ (Q \circ P^*)$ is a PDO with symbol $\tilde{q}_2(y,\eta)$ again satisfying \eqref{mlcutbds}.
\end{proposition}
\begin{proof} We only consider the case of $ Q \circ P^*$ the other case $ ( P \circ Q^*) \circ (Q \circ P^*)$ follows by a very similar argument.  Routine calculation shows that the Schwartz kernel of $ Q \circ P^*$ takes the form 
\begin{equation*}
\frac{1}{(2\pi h)^\d}\int e^{\frac ih (x-z)\cdot \eta}\tilde{q}_1(x,\eta)\,d\eta \quad \text{ where } \quad \tilde{q}_1(x,\eta) = \frac{1}{(2\pi h)^\d} \iint e^{-\frac ih y\cdot \xi} q(y+x,\xi+\eta) p(y+x,\eta)\,dy d\xi.
\end{equation*}
The criticality of the amplitude \eqref{mlcutbds} means we cannot apply stationary phase in all variables.  But we can apply it in the $y_d,\xi_d$ variables jointly to obtain an amplitude $b(x,\xi',y',\eta)$ such that 
\begin{equation}\label{statphaseind}
	\begin{split}
	&\tilde{q}_1(x,\eta) = \frac{1}{(2\pi h)^{\d-1}} \iint e^{-\frac ih y' \cdot \xi'}b(x,\xi',y',\eta) \,dy' d\xi' + \mathcal{O}(h^\infty),\\
	|\prtl_{x',y'}^{\alpha} \prtl_{x_d}^j \prtl_{\xi',\eta'}^\beta\prtl_{\eta_d}^k &b(x,\xi',y',\eta) | 
	\lesssim_{\alpha,j,\beta,k,N}  h^{-\frac 12(|\alpha|+|\beta|)} \big(1+ h^{-\frac 12}|y'+x'| + h^{-\frac 12}|\xi'+\eta'|  \big)^{-N} .
	\end{split}
\end{equation}
Define the differential operator $\Lsc$ as
\begin{equation*}
\Lsc: = \frac{1+i\xi' \cdot d_{y'} + iy'\cdot d_{\xi'}  }{ 1+h^{-1}|\xi'|^2+h^{-1}|y'|^2 } \quad \text{ so that } \quad \Lsc e^{-\frac ih y' \cdot \xi'} = e^{-\frac ih y' \cdot \xi'}.
\end{equation*}
Hence
\begin{equation*}
	\Lsc^T: = \frac{4ih^{-1}\xi'\cdot y' }{ (1+h^{-1}|\xi'|^2+h^{-1}|y'|^2)^2 } + \frac{1-i\xi' \cdot d_{y'} - iy'\cdot d_{\xi'}  }{ 1+h^{-1}|\xi'|^2+h^{-1}|y'|^2 },
\end{equation*}
where the first term is dominated by $(1+h^{-1}|\xi'|^2+h^{-1}|y'|^2)^{-1} $. Integration by parts with respect to $\Lsc$ sufficiently many times then yields
\begin{equation*}
	\begin{split}
	h^{\frac 12(|\alpha|+|\beta|)}&\Big|\prtl_{x'}^{\alpha} \prtl_{x_d}^j \prtl_{\eta'}^\beta\prtl_{\eta_d}^k \tilde{q}_1(x,\eta)\Big|\\
	&\lesssim_{N} h^{1-\d} \iint \big(1+h^{-\frac 12}|y'|  +h^{-\frac 12}|\xi'| \big)^{-N-2\d} \big(1+ h^{-\frac 12}|y'+x'| + h^{-\frac 12}|\xi'+\eta'|  \big)^{-N} \,dy d\xi\\
	&\lesssim_N \big(1+ h^{-\frac 12}|x'| + h^{-\frac 12}|\eta'|  \big)^{-N} .
	\end{split}
\end{equation*}
Indeed, differentiating with respect to $\xi'\cdot d_{y'}$ or $y'\cdot d_{\xi'}$ in the integration by parts generally leads to losses of $h^{-\frac 12}|\xi'|$ or $h^{-\frac 12}|y'|$, but this is counterbalanced by the bound 
\begin{equation*}
	(h^{-\frac 12}|\xi'| + h^{-\frac 12}|y'|)( 1+h^{-1}|\xi'|^2+h^{-1}|y'|^2 )^{-1} \lesssim  (1+h^{-1}|\xi'|^2+h^{-1}|y'|^2 )^{-\frac 12},
\end{equation*} 
meaning each integration by parts gains a power of the right hand side.

Note that since $|\eta-e_d|^2 = |\eta'|^2 + (\eta_d-1)^2$, the condition that $|\eta-e_d| \ll 1$ can generally be replaced by $|\eta_d-1| \ll 1$.  Indeed, if $|\eta'| \leq h^{\frac 34}$, then making this replacement is just a matter of adjusting the small implicit constants.  Otherwise the symbol is $\mathcal{O}(h^\infty)$ whenever $|\eta'| \geq h^{\frac 34}$.  Consequently, the preservation of the support  $|\eta-e_d| \ll 1$ is just a consequence of how the application of stationary phase in \eqref{statphaseind} yields a symbol supported where $|\eta_d-1| \ll 1$.
\end{proof}

\begin{theorem}\label{T:multipliers} 
Suppose $q$ is a symbol satisfying \eqref{mlcutbds} in a coordinate chart.  Given $N_0 \in \mathbb{Z}$, let $\Mno$ denote the real-valued multiplier 
	\begin{equation}\label{mnodef}
		\Mno (x,\xi) = \big(1+h^{-1}|x'| ^2+ h^{-1}|\xi'|^2 + h^{-1}d^2\big((x,\xi);\supp(q)\big) \big)^{N_0/2}
	\end{equation}
If $Q$ is the semiclassical PDO with symbol $q$, then\footnote{Here we use a slight abuse of notation, treating the function $\Mno$ as the same as the operator it defines.}
\begin{equation}\label{multiplierTTstar}
	\big\|\Mno \circ \Ssc \circ Q \circ Q^* \circ \Ssc^* \circ \Mno \big\|_{ L^2(T^*M ) \to L^2(T^*M ) } \lesssim_{N_0} 1.
\end{equation}
Consequently,
\begin{equation}\label{multbds}
	\big\|\Mno \circ \Ssc \circ Q \|_{ L^2(M ) \to L^2(T^*M ) } = \|Q^* \circ \Ssc^* \circ \Mno \big\|_{ L^2(T^*M ) \to L^2(M ) } \lesssim_{N_0} 1.
\end{equation}
\end{theorem}
\begin{proof} By the same arguments in Proposition \ref{P:composePDO}, $Q\circ Q^*$ defines a semiclassical PDO with symbol \eqref{mlcutbds}.  
Define $\Phi$ as the phase in the integral defining $\Ssc$
\begin{equation*}
	\Phi(z,\zeta, y) = -\zeta(\exp_z(y)) + \frac i2 d^2(z,y) = -\zeta\cdot(y-z) +\frac 12\big(i G(z) -\widetilde{\Gamma}(z,\zeta)\big)(y-z) \cdot (y-z) +\mathcal{O}(|y-z|^3),
\end{equation*}
where the right hand side expresses the quadratic Taylor expansion and $G(z), \widetilde{\Gamma}(z,\zeta)$ denote the matrices with $j,k$-th entry $g_{jk}(z)$ and $\zeta_\ell \Gamma_{jk}^\ell(z)$ respectively.  	Let $\K((z,\zeta),(x,\xi))$ denote the kernel of $\Ssc \circ Q\circ Q^* \circ \Ssc^*$
\begin{equation}\label{microintegral}
	\begin{gathered}
		\K \big((z,\zeta),(x,\xi)\big)	=	(2\pi h)^{-\d}(\pi h)^{-\frac{3\d}{2}}2^{-\d} \int_{\RR^{3\d}}e^{\frac ih\Psi} q(y,\eta)a(z,\zeta,y)\bar{a}(x,\xi,w)\,d\eta dy dw, \\
		\Psi(z,\zeta,y,w,\eta,x,\xi) := \Phi(z,\zeta,y) +(y-w)\cdot \eta -\bar{\Phi}(x,\xi,w).
	\end{gathered}
\end{equation}

Our main claim is that for any $N \in\mathbb{N}$,
\begin{multline}\label{mlcutphkernelbd}
	\big|\K \big((z,\zeta),(x,\xi)\big)	\big|\lesssim_{N} 
	h^{-d} \Big(1+h^{-\frac 12}|(x',\xi')| + h^{-\frac 12}|(z',\zeta')| + h^{-\frac 12}|(x_\d-z_\d,\xi_\d-\zeta_\d)| \\+h^{-1}d^2\big((x,\xi);\supp(q)\big) +h^{-1}d^2\big((z,\zeta);\supp(q)\big)\Big)^{-N}.
\end{multline}
If this is true, then \eqref{multiplierTTstar} follows by taking $N$ sufficiently large relative to $N_0$, then applying Young's inequality.  It is routine that \eqref{multbds} then follows from \eqref{multiplierTTstar}. 

To see \eqref{mlcutphkernelbd}, we calculate the differentials of the phase $\Psi$
\begin{equation*}
	\begin{aligned}
		d_\eta \Psi &= y-w,\\
		d_y \Psi &= \eta - \zeta + \big(iG(z) -\widetilde{\Gamma}(z,\zeta)\big)(y-z) + \mathcal{O}(|y-z|^2),\\
		d_w \Psi &= \xi- \eta + \big(iG(x) +\widetilde{\Gamma}(x,\xi)\big)(w-x) + \mathcal{O}(|w-x|^2).
	\end{aligned}
\end{equation*}
Considering their real and imaginary parts, it then follows that for $|y-z|, |w-x|$ sufficiently small
\begin{equation}\label{midiffapprox}
	\begin{aligned}
		|d_y \Psi| &\approx|\eta - \zeta |+|y-z|   \text{ and }|\prtl_y^\alpha \Psi| \lesssim_{\alpha} |y-z| \text{ if }|\alpha| \geq 2;\\
		|d_w \Psi| &\approx |\eta-\xi|+|w-x|\text{ and } |\prtl_w^\alpha \Psi| \lesssim_{\alpha} |w-x| \text{ if }|\alpha| \geq 2. 
	\end{aligned}
\end{equation}

We now define
\begin{equation*}
	\Lsc = \frac{1-i d_\eta \Psi  \cdot d_\eta -i d_y\Psi \cdot d_y-i d_w\Psi  \cdot d_w }{1+|h^{-\frac 12}d_\eta \Psi|^2 +  |h^{-\frac 12 } d_y\Psi| ^2+  |h^{-\frac 12} d_w\Psi|^2}, \quad \text{ so that } \quad \Lsc e^{\frac ih \Psi} = e^{\frac ih \Psi}.
\end{equation*}
Using \eqref{mlcutbds} and \eqref{midiffapprox}, integration by parts with respect to $\Lsc$ in \eqref{microintegral} in the same way as in the proof of Proposition \ref{P:composePDO} gives 
\begin{multline*}
	|\K \big((z,\zeta),(x,\xi)\big)| \lesssim_N \\
	h^{-\frac{5d}{2}}\int_{\supp(q)\times \RR} \big( 1+h^{-1}|d_\eta \Psi|^2 + h^{-1} |d_y\Psi| ^2+ h^{-1} |d_w\Psi|^2\big)^{-N}\big(1+ h^{-1}|(y',\eta')|^2 \big)^{-N}\,d\eta dy dw.
\end{multline*}
We can then replace $|d_y \Psi|, |d_w \Psi|$ by their approximations in \eqref{midiffapprox} and it is then routine to see that \eqref{mlcutphkernelbd} follows.  Indeed, by taking $N$ large enough, it can be seen that integration in $\RR^{3\d}$ taken here yields a gain of $h^\frac{3\d}{2}$, leading to the leading factor of $h^{-d}$.
\end{proof}

\section{The Parametrix}\label{S:WPansatz}
In this section, we construct the operator $\Vsc_t$ in Theorem \ref{T:Vtthm}.  We begin with some technical matters in \S\ref{SS:weylquant} concerning the symbol of $h\sqrtl$ as a PDO and consider its Taylor expansion.  In \S\ref{SS:ansatz}, we then consider our ansatz for a single wave packet which approximately solves the half-wave equation.  Finally, in \S\ref{SS:paramcon} we construct $\Vsc_t$ by taking superpositions of these wave packets, then state our main result on error estimates for the parametrix, whose proof is completed in \S\ref{S:matrixelements}.

\subsection{Weyl quantization of $h\sqrtl$}\label{SS:weylquant}
To calculate $h\sqrtl$, we use the semiclassical Weyl quantization.  Thus if $a(y,\eta)$ is a symbol $\Op(a)$ denotes the operator determined by $a$ in the Weyl calculus, that is, the operator with Schwartz kernel given by the oscillatory integral
\begin{equation*}
	\frac{1}{(2\pi h)^d}\int e^{\frac ih (y-z)\cdot\eta}  a\Big(\frac{y+z}{2}, \eta\Big)\,d\eta .
\end{equation*}
Even though previous sections used standard quantization, we prefer Weyl quantization here.
\begin{proposition}\label{P:weyllap}
	In any coordinate chart, the Weyl quantization of $h \sqrtl$ is given by a symbol $b(y,\eta)$ which admits an asymptotic expansion for $|\eta| \approx 1$, $b \sim \sum_{j \geq 0} h^{j}b_j$ with
	\begin{equation*}
		\bigg| \prtl_{y}^\beta\prtl_{\eta}^\alpha \Big( b(y,\eta) -\sum_{j=0}^{N-1} h^{j} b_j(y,\eta) \Big)\bigg|  \lesssim h^N \quad \text{and }\quad |\prtl_y^\beta\prtl_\eta^\alpha b_j | \lesssim_{j,\alpha,\beta} 1 \quad \text{ for }|\eta| \approx 1.
	\end{equation*}
Moreover, we have the explicit formulas for $|\eta| \geq h$
	\begin{equation}\label{leadingweyl}
		b_0(y,\eta) = |\eta|_{g(y)}=\sqrt{g^{jk}(y)\eta_j\eta_k} , \qquad b_1(y,\eta) = 
		\frac 12 g^{jk}(y) \frac{D_{y_j}\big( \sqrt{g(y)} \big)}{\sqrt{g(y)}} \frac{\eta_k}{|\eta|_g} .
	\end{equation}
\end{proposition}
\begin{proof}
	In terms of \emph{classical} symbols in the H\"ormander classes $S_{1,0}^m$, it is well-known there is an asymptotic expansion of the symbol $\sigma_{\sqrtl}(y,\zeta)$ of $\sqrtl$:
	\begin{equation*}
				\sigma_{\sqrtl}(y,\zeta) \sim \sum_{j\geq 0} \tilde b_j(y,\zeta), \qquad \tilde b_j \in S_{1,0}^{1-j} \text{ is homogeneous of degree $1-j$ in $\zeta$ for $|\zeta| \gg 1$ },
	\end{equation*}
	where the asymptotic sum means $\sigma_{\sqrtl}(y,\zeta) - \sum_{j= 0}^{N-1} \tilde b_j(y,\zeta) \in S_{1,0}^{1-N}$. Moreover, $\tilde b_0$, $\tilde b_1$ satisfy the right hand sides of \eqref{leadingweyl} for $|\zeta| \geq 1$, a consequence of the Weyl symbolic calculus (see e.g.  \cite[(14.23)]{TaylorII}). In particular this uses that $\Op(\tilde b) \circ \Op(\tilde b) - \Op(\tilde b^2)  $ is a pseudodifferential operator given by a symbol in $S_{1,0}^{0}$ whenever $\tilde b \in S_{1,0}^1$ since the Poisson bracket $\{\tilde b,\tilde b\}$ vanishes\footnote{This property is the reason why we prefer Weyl quantization in this section.}.

	Now set $b_j(y,\eta) = h^{1-j}\tilde b_j(y,\eta/h)$, when $|\eta/h| \geq 1 $ this is just $b_{j}(y,\eta)$.  We change variables $\zeta = \eta/h$ in each Schwartz kernel defined by $b_j$ to get
	\begin{equation*}
		\frac{1}{(2\pi)^d} \int e^{i(x-y)\cdot \zeta} \tilde b_j\Big( \frac{x+y}{2},\zeta \Big) \,d\zeta = 
		 \frac{h^{j-1}}{(2\pi h)^d} \int e^{\frac ih (x-y)\cdot \eta}  b_j\Big( \frac{x+y}{2},\eta \Big) \,d\eta, 
	\end{equation*}
	where the right hand side uses the homogeneity of $b_j$.  As a result, for any $\eta$ we have
	\begin{equation}\label{semiasymp}
		\begin{split}
			\sigma_{h\sqrtl}(y,\eta) \sim \sum_{j \geq 0} h^{j}b_j(y,\eta), \text{ where }  |\prtl_{y}^\beta\prtl_{\eta}^\alpha b_j(y,\eta)| \lesssim_{j,\alpha} h^{1-j-|\alpha|}(1+|\eta/h|)^{1-j-|\alpha|}\\
			\Big| \prtl_{y}^\beta\prtl_{\eta}^\alpha \Big( \sigma_{h\sqrtl}(y,\eta) -\sum_{0 \leq j \leq  N-1} h^{j} b_j(y,\eta) \Big)\Big| 
			\lesssim_{N,\alpha, \beta} h^{1-|\alpha|}(1+|\eta/h|)^{1-N-|\alpha|} .
		\end{split}
	\end{equation}
	The desired properties for $|\eta| \approx 1$ now follow as a special case.
	\end{proof}

We next recall the following formula for symbols $a(y,\eta) = \sum_{\alpha} a_\alpha(y)\eta^\alpha$ which are polynomial in the fiber variables and hence determine differential operators (see \cite[p.79-80]{TaylorII})
\begin{equation}\label{weyldiffop}
	(\Op(a)u)(y) = \sum_{\alpha} \sum_{\beta + \gamma = \alpha}  \binom{\alpha}{\beta} 2^{-|\gamma|} (hD_y)^\gamma a_{\alpha}(y) (hD_y)^\beta u(y).
\end{equation}
We note the following special cases for linear and quadratic polynomials (\cite[(14.34), (14.35)]{TaylorII})
\begin{equation}\label{weylspecialcases}
	\begin{split}
		a(y,\eta) = a^j(y)\eta_j &\implies (\Op(a)u)(y) = a^j(y)(h D_{y_j})u(y) + \frac 12 \big((hD_{y_j})a^j(y)\big)u(y),\\
		a(y,\eta) = a^{jk}(y)\eta_j\eta_k &\implies (\Op(a)u)(y) =(hD_{y_j})\big(a^{jk}(y)(hD_{y_k})u(y)\big) + \frac{h^2}4 \big((D_{y_j}D_{y_k} a^{jk}(y)\big)u(y).
	\end{split}
\end{equation} 

\begin{lemma}\label{L:quadexpn}
Consider normal coordinates centered at any point in $M$ and $b_0(y,\eta) = \sqrt{g^{\ell m}(y)\eta_\ell\eta_m}$ in these coordinates (as defined in \eqref{leadingweyl}).  Define the quadratic Taylor polynomial of the function $(y,\eta) \mapsto b_0(y,\xi+\eta)$ centered at $(y,\eta) =(0,0)$:
\begin{equation*}
	p_2(y,\eta) := \sum_{|\alpha|+|\beta|  \leq 2} \frac{1}{\alpha!\beta!} (\prtl_y^\beta \prtl_\eta^\alpha b_0)(0,\xi) y^\beta \eta^\alpha
\end{equation*}
Then $\Op(p_2)$ is the differential operator
\begin{equation}\label{quadraticquantize}
	|\xi| + \frac{\xi_j}{|\xi|} hD_{y_j}  +  
			\frac{1}{2|\xi|} \Big( \delta_{jk} -  \frac{\xi_j\xi_k}{|\xi|^2}\Big)(hD_{y_j})(hD_{y_k}) 
			+ \frac{1}{12|\xi|} \big( R_{\ell kj m} + R_{ \ell j k m} \big)y_j y_k\xi_\ell \xi_m.
\end{equation}
\end{lemma}
\begin{proof}
In a normal coordinate system, $g_{\ell m}(0) = \delta_{\ell m} = g^{\ell m}(0)$ and the first partials of the metric coefficients and their inverse vanish at the origin: $\prtl_{y_j}g_{\ell m}(0) =\prtl_{y_j}g^{\ell m}(0)  =0$.  In fact, we have the well-known\footnote{Perhaps the Taylor expansion of $g_{jk}(y)$ is more common to find, but this follows as a consequence.}  Taylor expansion
\begin{equation}\label{cometricnormal}
	g^{\ell m}(y) = \delta_{\ell m} +\frac 16\big( R_{\ell kj m} + R_{ \ell j k m}\big) y_j y_k  +  \mathcal{O}(|y|^3).
\end{equation}
Hence  $\prtl_{y_j}b_0(0,\eta)=0$ for all $\eta$, $j$ and $b_0(0,\xi) = |\xi|_g = |\xi|$. The Taylor expansion of $p_2$ simplifies as
\begin{equation*}
	p_2(y,\eta) = 
	b_0(0,\xi) +  \prtl_{\eta_j} b_0(0,\xi)\eta_j  + \frac 12 \prtl^2_{\eta_j\eta_k}b_0(0,\xi)\eta_j\eta_k  + \frac 12 \prtl^2_{y_jy_k}b_0(0,\xi)y_jy_k .
\end{equation*}
Using the observations above about the metric coefficients, we have 
\begin{equation*}
	\prtl_{\eta_j} b_0(0,\xi) = \frac{\xi_j}{|\xi|}, \quad  \prtl^2_{\eta_j\eta_k}b_0(0,\xi) = \frac{1}{|\xi|} \Big( \delta_{jk} -  \frac{\xi_j\xi_k}{|\xi|^2}\Big), \quad  \prtl^2_{y_jy_k}b_0(0,\xi) = \frac{\prtl_{y_jy_k}^2g^{\ell m}(0)\xi_\ell\xi_m}{2|\xi|}.
\end{equation*}
The expression \eqref{quadraticquantize} follows by using \eqref{weylspecialcases} to quantize the first 2 terms and \eqref{cometricnormal} on the last.
\end{proof}

\begin{remark}
	Below, we will use \eqref{quadraticquantize} in the special case where $\xi$ points in the direction of the $d$-th standard basis vector, that is $\xi_j = |\xi|_g\delta_{jd}$.  In this case, \eqref{quadraticquantize} takes the form
	\begin{equation}\label{xid}
		|\xi| + h D_{y^d} + \frac{1}{2|\xi|} \sum_{j,k = 1}^{d-1} (hD_{y_j})(hD_{y_k}) + \frac{|\xi|}{12} \sum_{j,k = 1}^{d-1} \big( R_{d kj d} + R_{ d j k d}\big)y_jy_k 
	\end{equation}
\end{remark}

\subsection{Ansatz for single wave packets}\label{SS:ansatz}
In contrast to previous sections, we start considering the \emph{bicharacteristic flow} associated to the half wave equation \eqref{halfwaveeqn}, namely, the Hamiltonian flow on $T^*M$ (minus the zero section) defined by $\tilde H: T^*M \to \RR$, $\tilde H (x,\xi) = |\xi|_{g(x)}$.  Hence the flow $\tilde\kappa_t : T^*M \to T^*M$ is generated by the vector field given in local coordinates by
\begin{equation*}
	\frac{g^{jk}(x)\xi_k}{|\xi|_{g(x)}} \frac{\prtl}{\prtl x_j} - \frac{\prtl_{x_\ell}g^{jk}(x)\xi_j\xi_k}{2|\xi|_{g(x)}} \frac{\prtl}{\prtl \xi_\ell}.
\end{equation*}
In what follows, we write $\tilde\kappa_t (x,\xi) = (x_t(x,\xi), \xi_t(x,\xi))$ in local trivializations of $T^*M$ taken for $(x,\xi)$ and $\tilde\kappa_t(x,\xi)$.   However, this subsection treats the case of a single bicharacteristic curve, so $(x_t,\xi_t)$ will denote the image of a single point $(x,\xi)$ under $\tilde\kappa_t$.  In \S\ref{SS:paramcon} below we begin to vary $(x,\xi)$.  

If we had considered the Hamiltonian $\frac 12 |\xi|_{g(x)}^2$, the corresponding flow on $T^*M$ would be identified with the geodesic flow on $TM$ from \S\ref{SS:diffgeodflow} via the musical isomorphism.  Instead, these two flows agree on the cosphere bundle $S^*M = \{(x,\xi) \in T^*M: |\xi|_{g(x)}=1\}$.  The main difference is that since $\tilde H$ is homogeneous of degree 1 in $\xi$, then $x_t(x,\xi)$ and $\xi_t(x,\xi)$ are homogeneous of degree 0 and 1 in $\xi$ respectively.  Unlike the geodesic flow, $x_t$ thus parameterizes a unit speed geodesic in $M$ for any $\xi$.  Moreover, the tangent vector $\dot x_t \in T_{x_t}M$ results from raising the covector\footnote{The Hamiltonian flow preserves $|\xi|_{g(x)}$, so in what follows we write this concisely as $|\xi|_g$.} $\xi/|\xi|_{g}$ under the musical isomorphism.  In this sense, $\tilde\kappa_t$ extends the geodesic flow on the cosphere bundle $S^*M = \{(x,\xi) \in T^*M: |\xi|_{g(x)}=1\}$ to all of $T^*M$ so that it is homogeneous of degree 1 in $\xi$.  As a consequence, we have that for $|\xi|_g \approx 1$, $\|d\tilde\kappa_t\| \lesssim \|d\kappa_t\| \leq \mu(t)$.

We now consider the action of the half-wave operator $hD_t + h\sqrtl$ to a function $u$ of the form
\begin{align}
	u(t,y) &= (\pi h)^{-\frac d4} a(t)e^{\frac ih \phi(t,y)}	\label{udef}\\
\phi(t,y) &= \xi_t(\expy) + \frac 12 \omega_t\big(\expy, \expy \big),\label{phidef}
\end{align}
where $\xi_t(\expy) \in \RR$ is the pairing of $\xi_t $ with the tangent vector $\expy \in T_{x_t}M$.  Moreover,  $\omega_t$ is a $(0,2)$ tensor defined at each tangent space $T_{x_t}M$. Upon raising an index to form a $(1,1)$-tensor $ \omega_t^\#$, it satisfies the Riccati equation considered in Corollary \ref{C:cor4ansatz} with $c = |\xi|_g$,
\begin{equation}\label{riccatiansatz}
 |\xi|_g \covt \omega_t^\# + \omega_t^\# \circ \Pi^\perp \circ \omega_t^\# + |\xi|_g^2 R_{\dot x(t)}=0, \qquad  \omega_0^\# = iI.
\end{equation}
Note that the initial condition $\omega_0^\# = iI$ as a $(1,1)$ tensor on $T_{x_0}M$ means that  as a $(0,2)$ tensor we have in coordinates $(\omega_0)_{jk} = ig_{jk}$ and
\begin{equation}\label{initialdist}
	\omega_t\big(\expy, \expy \big)\big|_{t=0} = id_g^2(x,y)
\end{equation} 
Finally, $a(t)$ is a solution to the ODE also considered in Corollary \ref{C:cor4ansatz}
\begin{equation}\label{amplitudeansatz}
	|\xi|_g a'(t) + \frac{1}{2} \tr(\omega_t^\# \circ \Pi_{\dot x_t})a(t)=0, \qquad a(0)=1.
\end{equation}

Even though $\phi$ is coordinate invariant, it is convenient to perform the calculations where $y$ values are expressed in a normal coordinate chart centered at $x_t$, which we assume is taken in the remainder of the subsection.  In particular, we consider a fixed orthonormal frame $E_1(t),\dots,E_d(t)$ for $T_{x_t}M $ defined in \eqref{parallelONframe}, then take normal coordinates defined by the local diffeomorphism 
\begin{equation*}
	(y_1,\dots,y_d) \mapsto \exp_{x_t}(y_1E_1(t)+ \cdots +y_d E_d(t))
\end{equation*} 
in a neighborhood of the origin. In these coordinates $\expy$ is simply expressed by the vector $y \in \RR^d$, that is, $\expy = \sum_{j=1}^d y_j \frac{\prtl}{\prtl y_j}|_0$.  Moreover, $\dot x_t=E_d(t) = \frac{\prtl}{\prtl y_d}|_0$ and hence the covector $\xi_t$ over $x_t$ is of the form $\xi_t= |\xi|_g\,dy_d|_0$ so that $\xi_t(\expy) = |\xi|_g\, y_d$. Overall, 
\begin{equation}\label{fixtphi}
	\phi(t,y) = |\xi|_gy_d + \frac 12\sum_{j,k=1}^{d-1} (\tilde\omega_t)_{jk}y_jy_k + \frac i2 y_d^2, \quad \text{ where }\quad (\tilde\omega_t)_{jk} = \omega_t (E_j(t),E_k(t)).
\end{equation}

The expression \eqref{fixtphi} introduces the notation we use in what follows: $\omega_t$ denotes the full tensor on $T_{x_t}^\CC M$, while $\tilde\omega_t$ denotes its restriction to the normal space of $\dot x_t$.  As matrices, if the fundamental matrix for the Jacobi equation expressed in terms of $E_1,\dots,E_{d-1}$ takes the form as in \eqref{blocksymplectic}
\begin{equation}\label{blocktildeomega}
	\begin{bmatrix}
		A_s & B_s\\
		C_s & D_s
	\end{bmatrix} \quad \text{ where } \quad s=\frac{t}{|\xi|_g},
\end{equation}
then $\tilde\omega_t = (C_s+iD_s)(A_s+iB_s)^{-1}$ as a $(d-1)\times(d-1) $ matrix.  Recall from Corollary \ref{C:cor4ansatz} that as a function of $s$, the Jacobi equation here is solved along the geodesic $(x_s,|\xi|_g\dot x_s) \in TM$. In block form, the full matrix for $\omega_t$ is
\begin{equation}\label{blockomega}
\omega_t = 	\begin{bmatrix}
		\tilde{\omega}_t & 0\\
		0 & i
	\end{bmatrix}
	=
	\begin{bmatrix}
		(C_s+iD_s)(A_s+iB_s)^{-1} & 0\\
		0 & i
	\end{bmatrix}
\end{equation}

\begin{remark} 
	Since we take $c = |\xi|_g$ in Corollary \ref{C:cor4ansatz}, we have the following  interpretation of $\omega$: it is determined by the geodesic flow on $S^{(r)}M$ with $r=|\xi|_g$, but since $x_t$ is of unit speed, the rescaling $s=t/|\xi|_g$ ensures $\omega_t$ operates on $T_{x_t}^{\CC}M$ and not some other tangent space along the geodesic $x_t$.
\end{remark}

\begin{remark}\label{R:hvsmu}
	Throughout this section, we assume that the norm of the matrix \eqref{blocktildeomega} is bounded above by $\mu(t)$ as per the hypotheses of Theorems \ref{T:mainthmexp} and \ref{T:mainthmpoly}.  We emphasize the observation \eqref{mubddh} that $\mu = \mu(t) \leq h^{-\frac{1}{26}}$ for all $0 \leq t \leq T$.  This property will get used implicitly at many stages of the proof in what follows.  It means that many quantities appearing our error analysis are $\mathcal{O}(h^{\veps})$ for some $\veps >0$ (e.g. the right hand side in \eqref{wtuL2norm} below). In fact, we would have a satisfactory error analysis under the weaker assumption that $\mu(t) \leq h^{\frac 16-\veps}$ as mentioned in Remark \ref{R:baderror} above: our stronger hypotheses are not crucially used until \S\ref{SS:errorestimates}.
\end{remark}

\begin{remark} \label{R:cutoff}
Strictly speaking, in order to ensure that the $\expy$ is well-defined, we should instead consider the product $\psi(x_t,y)u(t,y)$ where $u$ is defined above in \eqref{udef} and $\psi$ is defined as in \eqref{wptdef1}.  But since $\psi$ is identically one in a neighborhood of the diagonal in $M\times M$, it follows from \eqref{siegelbd5} and Remark \ref{R:hvsmu} that
\begin{equation*}
	\|(\prtl_t\psi(x_t,\cdot))u(t,\cdot)\|_{L^2(M)} = \mathcal{O}(h^\infty)\quad \text{ and } \quad \|(1-\psi) (x_t,\cdot) u(t,\cdot)\|_{L^2(M)} = \mathcal{O}(h^\infty).
\end{equation*}
Hence the difference between this approximate solution and the one given by $u$ is inconsequential, and this subsection presents a local error analysis of $u$ for $y$ sufficiently close to $0$ in our coordinates.
\end{remark}

\begin{proposition}\label{P:uL2norm}
	For each $|t| \leq T$, the approximate solution $u(t,y)$ is $L^2$-normalized with respect to Lebesgue measure in $\RR^d$ in that
	\begin{equation}\label{uL2norm}
		\|u(t,\cdot)\|_{L^2(\RR^d)} = 1 
	\end{equation}
	Moreover, for any $k \geq 0$,
	\begin{equation}\label{wtuL2norm}
		\big\||y|^ku(t,\cdot)\big\|_{L^2(\RR^d)} \lesssim (h^{\frac 12}\mu)^{k}
	\end{equation}
\end{proposition}
\begin{proof} 
Recall from \eqref{atmatrix} that $a(t) = \detexp^{-\frac 12}(A_s+iB_s)$ where $s = t/|\xi|$.  Hence
	\begin{equation*}
		|a(t)|^{-4} = |\det(A_s+iB_s)|^2 = \big|\det(A_s+iB_s)\det(A_s^T-iB_s^T)\big| = \det(A_sA_s^T+B_sB_s^T) = \det(\Im \omega_t),
	\end{equation*}
where the last identity follows from \eqref{realimagsiegel}. We now have that with $y'=(y_1,\dots,y_{d-1})$
	\begin{equation*}
			\int_{\RR^d} |u(t,y)|^2\,dy = 
			(\pi h)^{-\frac{d}{2}} |a(t)|^2 \int_{\RR^d} e^{-\frac{1}{2h}(\Im \tilde\omega_t)y'\cdot y' -\frac{1}{2h}|y_d|^2}\,dy = |a(t)|^2\detexp^{-\frac 12}(\Im \tilde\omega_t) =1.
	\end{equation*}
Moreover, by \eqref{siegelbd5}, $\|(\Im \tilde\omega_t)^{-\frac{1}{2}}\| \lesssim \mu$ and hence \eqref{wtuL2norm} follows from
\begin{equation*}
\begin{split}
	\int_{\RR^d} |y|^{2k}|u(t,y)|^2\,dy &\lesssim 
	\big(1+\|(\Im \tilde\omega_t)^{-\frac{1}{2}}\|\big)^{2k} \int_{\RR^d}\big(|(\Im \tilde\omega_t)^{\frac{1}{2}} y'|^2+|y_d|^2\big)^{k} |u(t,y)|^2\,dy\\ 
	&\lesssim h^k\big(1+\|(\Im \tilde\omega_t)^{-\frac{1}{2}}\|\big)^{2k} \lesssim h^k\mu^{2k}.
	\end{split}
\end{equation*}
\end{proof}

\begin{remark}
	Proposition \ref{P:uL2norm} concerns the $L^2$ norm of $u$ on $\RR^d$. If we instead consider the $L^2$ norm with respect to Riemannian volume on $M$,  \eqref{cometricnormal} implies $\sqrt{\det{g_{jk}(y)}} = 1+\mathcal{O}(|y|^2)$ and hence
	\begin{equation*}
		\|\psi(x_t,\cdot)u(t,\cdot)\|_{L^2(M)}^2 = \int |\psi(x_t,y) u(t,y)|^2\sqrt{\det{g_{jk}(y)}}\,dy = 1+\mathcal{O}(h^{\frac 12}\mu).
	\end{equation*}	
\end{remark}

The rest of this subsection is now dedicated to proving the following.
\begin{theorem}\label{T:halfwaveu}
	For $N_0$ sufficiently large, there exists $N_1$ depending only on $N_0$ such that
	\begin{equation}\label{halfwaveu}
		(hD_t+h\sqrtl) u = \Big(\sum_{|\gamma| \leq N_1}f_\gamma(\omega,h) y^\gamma \Big)u + E(t,y) 
	\end{equation}
	where $\|E(t,\cdot)\|_{L^2} \lesssim h^{N_0}$.  The coefficients $f_\gamma(\omega,h)$ satisfy
	\begin{equation}\label{fcoeffs}
		| f_\gamma(\omega,h)| \lesssim 
		\begin{cases}
			\mu^{2|\gamma|}, 
			& \text{ if } |\gamma| \geq 3, \\
			h \mu^6,  
			& \text{ if } |\gamma| = 2,\\
			h \mu^4,  
			& \text{ if } |\gamma| = 1 ,\\
			h^2\mu^4, & \text{ if } |\gamma| =0.
		\end{cases}
	\end{equation}
\end{theorem}

Our first step is to calculate $\prtl_t \phi$.  Here we cannot use \eqref{fixtphi} as this assumes $x_t$ is fixed.  However, resolved by appealing to the second part of Proposition \ref{P:covariantlogexpn}. 
\begin{lemma}\label{L:firsttderiv}
Let $\phi$ be defined as in \eqref{phidef}.  Then in normal coordinates centered at $x_t$, we have that for any integer $N_1\geq 3$
\begin{equation}\label{firsttderiv}
\prtl_t \phi = - |\xi|_g  - iy_d + \frac{|\xi|_g}{6}(R_{dk\ell d} + R_{d\ell k d})\, y_k y_\ell + \frac 12 (\covt \omega_t)(\expy, \expy) +\widetilde{E}
\end{equation}
where $\widetilde{E} = \sum_{3 \leq |\gamma|\leq N_1} \widetilde{E}_\gamma y^\gamma+ \mathcal{O}\big((1+\|\omega\|)|y|^{N_1+1}\big) $
near $y=0$ where $\widetilde{E}_\gamma $ depends on $\omega,x,\xi$ and satisfies $| \widetilde{E}_\gamma | \lesssim_{\gamma} 1+\|\omega\|$. 
\end{lemma}

\begin{proof}
Since $x_t$ parameterizes a unit speed geodesic $\nabla_t \dot x_t \equiv 0$, thus since  $\xi_t =|\xi|_{g} (\dot{ x}_t )^\flat$, it follows that $\nabla_t \xi_t \equiv 0$ since covariant differentiation commutes with the musical isomorphism.     Hence
\begin{equation*}
	\prtl_t \phi = \xi_t\big( \covt \expy\big) + \omega_t (\covt\expy, \expy) +	\frac 12 (\covt \omega_t)(\expy, \expy) 
\end{equation*}
Since $\dot{x}_t = \frac{\prtl}{\prtl y^d} |_{y=0}$, \eqref{covariantlogexpndir} in Proposition \ref{P:covariantlogexpn} gives that near $y=0$
\begin{equation*}
 \covt \expy =  \nabla_{\dot x_t} \expy =-\frac{\prtl }{\prtl y^d}\bigg|_0 + \frac{1}{6}  \big( R_{d k\ell }^j + R_{d \ell k }^j \big) y_k y_\ell \big)\frac{\prtl }{\prtl y_j}\bigg|_0+ \mathcal{O}(|y|^3).
\end{equation*}
Consequently, using that $\xi_t = |\xi|_g dy_d|_0$ again, it follows that
\begin{equation*}
	\xi_t\big( \covt \expy\big) + \omega_t (\covt\expy, \expy) = - |\xi|_g  + \frac{|\xi|_g}{6}(R_{dk\ell d} + R_{d\ell k d})\, y_k y_\ell - iy_d + \mathcal{O}(|y|^3).
\end{equation*}
Indeed, since the metric at $y=0$ satisfies $g_{jk}(0) = \delta_{jk}$, we have that $R_{d k\ell }^j = R_{d k\ell j} $.
\end{proof}

We now turn to the symbol $b$ defined in Proposition \ref{P:weyllap}.
We write the image of $u$ under $\Op(b)$, and similarly each $\Op(b_j)$, as
\begin{multline}\label{weylshift}
	\frac{a(t)}{2^d(\pi h)^{\frac{3d}{2}}} \iint e^{\frac ih (x-y)\cdot \eta} b\Big(\frac{x+y}{2},\eta \Big) e^{\frac ih(y\cdot \xi + \frac 12 \omega y\cdot y)}\,dy\,d\eta =\\
	\frac{e^{\frac ih x\cdot \xi}a(t)}{2^d(\pi h)^{\frac{3d}{2}}} \iint e^{\frac ih (x-y)\cdot \eta} b\Big(\frac{x+y}{2},\eta +\xi\Big) e^{\frac{i}{2h} \omega y\cdot y}\,dy\,d\eta.
\end{multline}
In the first integral, we implicitly assume that $q$ is supported where $|\eta| \approx 1$ and that $\xi$ is taken so that $|\xi| \approx 1$ so that the bounds in Proposition \ref{P:weyllap} are satisfied:  this is possible since we ultimately consider $\xi$ with $\beta(|\xi|_g) \neq 0$ and we encounter a phase with no critical points when $|\eta| \not\approx 1$.

Using Lemma \ref{L:quadexpn}, namely \eqref{xid}, we take the quadratic part of the Taylor polynomial of $b_0$ centered at $(0,\xi)$. For now we use this to calculate the leading order contribution of $\Op(b_0)u$, treating the remaining contributions as error terms below
\begin{multline}\label{quadonexp}
	e^{-\frac{i}{2h}\omega y \cdot y}\Op(p_2) (e^{\frac{i}{2h}\omega y \cdot y}) = \frac{h}{2i|\xi|} \big( \delta_{jk} -  \delta_{jd}\delta_{kd} \big) \omega_{jk} \\
	 +|\xi| + iy_d + \frac{1}{2|\xi|} \big( \delta_{jk} -  \delta_{jd}\delta_{kd} \big)  \omega_{j\ell}\omega_{km} y_\ell y_m   + \frac{|\xi|}{12 } \sum_{j,k=1}^{d-1}\big( R_{d k jd} + R_{d j k d}\big)y_k y_j  .
\end{multline}
We now combine \eqref{firsttderiv} and \eqref{quadonexp} to get that 
\begin{multline*}
	\big(hD_t + \Op(p_2)\big) u = \frac{h}{i|\xi|} \Big(|\xi| a'(t) + \frac{1}{2} \tr(\omega_t^\# \circ \Pi_{\dot x_t})a(t)\Big)e^{\frac ih \phi}\\
	+\frac{1}{2|\xi|}\Big(|\xi| (\nabla_t \omega_t)_{j k } y_j y_k  + 
	\frac{|\xi|^2}{2} \big(R_{d j k d} + R_{d k j d}  \big)y_j y_k  +  \big( \delta_{jk} -  \delta_{jd}\delta_{kd} \big) \omega_{j\ell}\omega_{km} y_\ell y_m  \Big)u + \widetilde{E}u
\end{multline*}
We now examine the two expressions in parentheses on the right, the first of these is exactly \eqref{amplitudeansatz}.  For the second one,  we use symmetries of the $(0,4)$ Riemann curvature tensor 
\begin{equation*}
	 R_{d j k d} + R_{d k j d}  = 2R_{j d d k}.
\end{equation*}
Using again that $g_{jk}(0) = \delta_{jk}$, it follows that the second term in parentheses is 
\begin{multline*}
	|\xi|_g \covt \omega_t\big(\expy,\expy\big) + (\omega_t^\# \circ \Pi^\perp \circ \omega_t^\#)^\flat \big(\expy,\expy\big) \\+ |\xi|_g^2 Rm\big(\expy,\dot x_t,\dot x_t,\expy\big) =0 ,
\end{multline*}
which vanishes as this is just \eqref{riccatiansatz} evaluated along $(\expy,\expy)$.

By Proposition \ref{P:uL2norm} and the bound $\|\omega\| \lesssim \mu^2$ from \eqref{siegelbd1}, we have that
\begin{equation*}
	\Big( \int \Big|\widetilde E(t,y) - \sum_{3 \leq |\gamma| \leq N_1} \widetilde E_\gamma(t) y^\gamma \Big|^2 |u(t,y)|^2\,dy\Big)^{\frac 12}\lesssim_{N_1} h^{\frac{N_1+1}{2}}\mu^{N_1+1} (1+\|\omega\|) \lesssim h^{\frac{N_1+1}{2}} \mu^{N_1+3}.
\end{equation*}
Thus by taking $N_1$ sufficiently large, the contribution of this difference can be absorbed into the term $E$ in \eqref{halfwaveu}.  Moreover, by Lemma \ref{L:firsttderiv}, each coefficient satisfies the bound $|\widetilde E_\gamma(t)| \lesssim (1+\|\omega\|)$, which is at least as strong as the right hand side of \eqref{fcoeffs}.  

It remains to show that $\big(h\sqrtl-\Op(p_2) \big)u $ has the form of the right hand side of \eqref{halfwaveu}. 
We use Proposition \ref{P:weyllap}, setting $\tilde{b}_{N_0} := h^{-N_0}(b-\sum_{j=0}^{N_0}h^jb_j)$ so that\footnote{As above, we can implicitly assume kernel of these PDOs is an integral supported near $|\eta|\approx 1$.} $|\prtl_{y,\eta}^\alpha \tilde b_{N_0}| \lesssim_{\alpha,N_0} 1$ and
\begin{equation}\label{weyllapconsequence}
	\big(h\sqrtl-\Op(p_2) \big)u = \big(\Op(b_0)-\Op(p_2) \big)u + \sum_{j=1}^{N_0-1}h^j\Op(b_j)u + h^{N_0}\Op(\tilde{b}_{N_0})u.
\end{equation}
By Proposition \ref{P:uL2norm} and standard $L^2$ bounds for PDO, the last term on the right is $\mathcal{O}(h^{N_0})$ in $L^2$ and hence can be embedded into the error term $E(t,y)$ in \eqref{halfwaveu}.
We then want to take higher order Taylor approximations to the symbols $b_0,\dots,b_{N_0-1}$ centered at $(0,\xi)$, in particular going beyond the quadratic approximation to $b_0$ in Lemma \ref{L:quadexpn}. The next two lemmas allows us to achieve this: the first calculates the effect of applying a polynomial symbol to $u$ while the second bounds the error in the Taylor approximation.  For convenience instead of treating $u$, we consider 
\begin{equation}\label{vdef}
	v(y) =e^{\frac{i}{2h} \omega y\cdot y }= \frac{e^{-i|\xi|_gy_d}}{a(t)}u(t,y)= e^{\frac{i}{2h} \tilde\omega y'\cdot y'-\frac{1}{2h}y_d^2}
\end{equation}

\begin{lemma}\label{L:polysymbolgenGaussian}
Let $a(y,\eta) = y^\beta \eta^\alpha$ be a monomial symbol.  Then with $v$ as in \eqref{vdef},
\begin{equation}\label{polysymbolgenGaussian}
	\big(\Op(a) v\big)(y) = 
	v(y)\sum_{0\leq |\gamma| \leq |\alpha|+|\beta|}
	h^{\frac 12(|\alpha|+|\beta|-|\gamma|)}c_{\alpha,\beta,\gamma}(\omega) y^\gamma 
\end{equation}
where $c_{\alpha,\beta,\gamma }(\omega) \equiv 0$ if $|\gamma|$ and $|\alpha|\pm|\beta|$ have opposite parity. Otherwise, $c_{\alpha,\beta,\gamma}(\omega)$ is homogeneous polynomial of degree $\frac 12(|\gamma|+|\alpha|-|\beta|)$ in the entries of $\omega$ otherwise (possibly the zero polynomial). 
\end{lemma}
\begin{proof}
By induction on $|\nu|$, we have that
\begin{equation}\label{diffgenGaussian}
	 \big( h D_y \big)^\nu v(y) 
	= v(y)\sum_{0 \leq |\gamma| \leq |\nu|} h^{\frac 12(|\nu|-|\gamma|)}\tilde{c}_{\nu\gamma}(\omega)y^\gamma
\end{equation}
where $\tilde{c}_{\nu\gamma}$ is a homogeneous polynomial of degree $\frac 12(|\nu|-|\gamma|)$ in the entries of $\omega$ which vanishes whenever $|\nu|\pm|\gamma|$ are odd (hence the polynomial on the right has the same parity as $|\nu|$).  

Next, as a special case of \eqref{weyldiffop}, we observe that for some constants $ b_{\alpha\beta\nu}$ which vanish if $\nu \not\leq \beta + \alpha$ 
\begin{equation*}
 \begin{split}
 	(\Op(a)v)(y) &= 
	\sum_{\nu \leq \alpha} b_{\alpha\beta\nu} h^{|\alpha|-|\nu|}y^{\beta -(\alpha-\nu)} (hD_y)^{\nu } v(y) \\ 
	&=v(y)\sum_{\nu \leq \alpha}  \sum_{0 \leq |\tilde\gamma| \leq |\nu|}  b_{\alpha\beta\nu} h^{|\alpha|-|\nu|+\frac 12(|\nu|-|\tilde\gamma|)}\tilde{c}_{\nu\tilde\gamma}(\omega)y^{\beta-(\alpha-\nu)+\tilde\gamma}.
 \end{split}
\end{equation*}
We then reindex the double sum in terms of the variable $\gamma = \beta-(\alpha-\nu) + \tilde\gamma$.  The exponent of $h$ can then be expressed as $\frac 12(|\alpha|+|\beta|-|\gamma|)$ and the degree of the polynomial $\tilde{c}_{\nu\tilde\gamma}(\omega)$ is $\frac 12(|\alpha|-|\beta|+|\gamma|)$. The expression \eqref{polysymbolgenGaussian} then follows. 
\end{proof}

\begin{lemma}\label{L:highordervanbd}
	Suppose $\tilde a$ is a symbol in $S(1)$. Then with $v$ as in \eqref{vdef},
\begin{equation}\label{highordervanbd}
	\big\|(\Op(z^\beta \eta^\alpha \tilde a)v)(y)\big\|_{L^2} \lesssim h^{\frac 12 (|\beta|+|\alpha|) }  \mu^{|\beta|+3|\alpha|} \|v\|_{L^2}.
\end{equation}
\end{lemma}
\begin{proof}
Observe that 
\begin{equation}\label{weylhighordervan}
	(\Op(z^\beta \eta^\alpha \tilde a)v)(y)=\frac{1}{(2\pi h)^d} \iint e^{\frac ih (y-z)\cdot\eta } \Big( \frac{y+z}{2} \Big)^\beta \eta^\alpha \tilde a\Big(\frac{y+z}{2}, \eta\Big) e^{-\frac{1}{2h} \omega z \cdot z}\,dzd\eta.
\end{equation}
We then use the following identities to integrate by parts in the integral
\begin{equation*}
	e^{-\frac ih (y-z)\cdot \eta}(hD_\eta)^\beta e^{\frac ih (y-z)\cdot \eta} = (y-z)^\beta,
	 \qquad e^{-\frac ih (y-z)\cdot \eta}(-hD_z)^\alpha e^{\frac ih (y-z)\cdot \eta} = \eta^\alpha.
\end{equation*}
Integration by parts with respect to $z$, results in a sum of integrals of the following type with $\alpha_1 + \alpha_2 + \alpha_3 = \alpha$ and $\alpha_1 \leq \beta$
\begin{equation}\label{weylhighordervansub}
	\frac{h^{|\alpha_1|+|\alpha_2|}}{(2\pi h)^d} \iint e^{\frac ih (y-z)\cdot\eta } \Big( \frac{y+z}{2} \Big)^{\beta-\alpha_1}  \big(\prtl_z^{\alpha_2}\tilde a\big)\Big(\frac{y+z}{2}, \eta\Big) (hD_z)^{\alpha_3} e^{-\frac{1}{2h} \omega z \cdot z}\,dzd\eta ,
\end{equation}
where we have omitted inconsequential constants in the expressions.  We now use the binomial theorem to write
\begin{equation*}
	 \Big( \frac{y+z}{2} \Big)^{\beta-\alpha_1} = \sum_{\beta_1 + \beta_2 = \beta - \alpha_1} \binom{\beta-\alpha_1}{\beta_1} 2^{-|\beta_1|} (y-z)^{\beta_1}z^{\beta_2}.
\end{equation*}
We now use this to integrate by parts in each integral in \eqref{weylhighordervansub} to get a sum of integrals of the form
\begin{equation*}
		\frac{h^{|\alpha_1|+|\alpha_2| +|\beta_1|}}{(2\pi h)^d} \iint e^{\frac ih (y-z)\cdot\eta }  \big(\prtl_\eta^{\beta_1}\prtl_z^{\alpha_2}\tilde a\big)\Big(\frac{y+z}{2}, \eta\Big) z^{\beta_2}(hD_z)^{\alpha_3} e^{-\frac{1}{2h} \omega z \cdot z}\,dzd\eta ,
\end{equation*}
omitting inconsequential constants again.  We now use \eqref{diffgenGaussian} and standard $L^2$ bounds on PDOs to get that the $L^2$ norm of the function defined by this integral is dominated by 
\begin{equation*}
	h^{|\alpha_1|+|\alpha_2| +|\beta_1|+\frac 12 (|\beta_2|+|\alpha_3|)}  \mu^{|\beta_2|+3|\alpha_3|} \|v\|_{L^2} . 
\end{equation*}
We now see that the largest contributions to function defined by \eqref{weylhighordervansub} comes from the term with $\beta_2 = \beta-\alpha_1$, $\beta_1=0$, meaning its $L^2$ norm is bounded by 
\begin{equation*}
		h^{\frac 12 (|\beta|+|\alpha_1|+|\alpha_3|)+|\alpha_2| }  \mu^{|\beta|-|\alpha_1|+3|\alpha_3|} \|v\|_{L^2} .
\end{equation*}
We then see the largest contribution to \eqref{weylhighordervan} comes from the integral \eqref{weylhighordervansub} with $\alpha = \alpha_3$, $\alpha_1 = \alpha_2 =0$, at which point \eqref{highordervanbd} follows.
\end{proof}

Given these lemmas, we conclude the proof of Theorem \ref{T:halfwaveu}.  Returning to \eqref{weyllapconsequence}, we express each $\Op(b_j)u$ in the same way as \eqref{weylshift}. We then take a Taylor approximation of each $b_j$ about $(0,\xi)$ in the integral to sufficiently high order so that the remainder vanishes to order $N_1+1$ at $(0,\xi)$ at which point Lemma \ref{L:highordervanbd} implies that this contribution of the error is $\mathcal{O}(h^{\frac{N_1+1}{2}} \mu^{3(N_1+1)})$ which is $\mathcal{O}(h^{N_0})$ provided $N_1$ is sufficiently large.  More precisely, it suffices to consider the Weyl quantization of the following Taylor polynomial acting on $u$
\begin{multline}\label{bigtaylor}
	\sum_{3 \leq |\alpha|+|\beta| \leq N_1} \frac{1}{\alpha!\beta!}\big(\prtl_{x}^\beta \prtl_{\xi}^\alpha b_0\big)(0,\xi) y^\beta \eta^\alpha + 
	\sum_{1 \leq |\alpha|+|\beta| \leq N_1}  \frac{h}{\alpha!\beta!} \big(\prtl_{x}^\beta \prtl_{\xi}^\alpha b_1\big)(0,\xi) y^\beta \eta^\alpha \\ 
	+ \sum_{j=2}^{N_0}  \sum_{0\leq |\alpha| + |\beta|\leq N_1} \frac{h^j}{\alpha!\beta!}\big(\prtl_{x}^\beta \prtl_{\xi}^\alpha b_j\big)(0,\xi) y^\beta \eta^\alpha.
\end{multline}
Indeed, the difference $h\sqrtl-\Op(p_2)$ cancels out all monomials of degree 2 or smaller in the Taylor expansion for $b_0$. Moreover, $b_1(0,\xi)=0$ by  \eqref{leadingweyl} since the first partials of the metric vanish in our normal coordinate system.  

We now apply Lemma \ref{L:polysymbolgenGaussian}, to each term in \eqref{bigtaylor}, which produces a polynomial of degree $N_1$ in $y$ multiplied by $u$.  The bounds on coefficient $f_\gamma(\omega,h)$ of $y^\gamma$ in this polynomial are determined by a sum over $\alpha,\beta$ satisfying $|\gamma| \leq |\alpha|+|\beta| \leq N_1$
\begin{equation}\label{coeffbds}
	|h^{j+\frac 12(|\alpha|+|\beta|-|\gamma|)}c_{\alpha,\beta,\gamma}(\omega)| 
	\lesssim h^{j+\frac 12(|\alpha|+|\beta|-|\gamma|)}
	\mu^{|\gamma|+|\alpha| - |\beta|} = h^j(h^{\frac 12}\mu)^{|\alpha|+|\beta|-|\gamma|}\mu^{2(|\gamma|-|\beta|)},
\end{equation}
and $c_{\alpha,\beta,\gamma}(\omega)$ vanishes unless $|\alpha|+|\beta|$ has the same parity as $|\gamma|$.  Using \eqref{mubddh}, $h^j(h^{\frac 12}\mu)^{|\alpha|+|\beta|-|\gamma|}\leq 1$.  The first case of \eqref{fcoeffs} with $|\gamma| \geq 3$ then follows, as the quantity is maximized when $|\beta| = 0$. 

The remaining cases of \eqref{fcoeffs} with $0\leq |\gamma|\leq 2$ are then determined by the case $j=0$.   In this case, we first maximize the right hand side of \eqref{coeffbds} over all $\alpha,\beta$ satisfying  $N=|\alpha|+|\beta|$ where $N\geq 3$ and $N$ has the same parity as $|\gamma|$.  This is maximized when $|\beta|=0$, which gives
\begin{equation*}
	(h^{\frac{1}{2}}\mu)^{N-|\gamma|}\mu^{2|\gamma|}.
\end{equation*}
We then maximize this over all $N \geq 3$ of the same parity as $|\gamma|$.  The maximum occurs when $N=3$ in the $|\gamma| =1$ case and when $N=4$ in the $|\gamma|=2,0$ cases, which is exactly what appears in \eqref{fcoeffs}.

The contribution of the terms where $j \geq1$ are estimated similarly and satisfy bounds which are stronger than \eqref{fcoeffs}.  Although we have to consider smaller values of $N=|\alpha|+|\beta|$ (limited to $N \geq 1$ when $j=1$), this is more than compensated by the larger exponent of $h$ (cf. \eqref{mubddh}).

\begin{corollary}\label{C:halfwaveu}
	Let $P$ be a PDO with symbol in $S(1)$.
	For $N_0$ sufficiently large, there exists $N_1$ depending only on $N_0$ such that
	\begin{equation}\label{halfwaveuP}
		P \circ (hD_t+h\sqrtl) u = \Big(\sum_{|\gamma| \leq N_1} \tilde f_\gamma(\omega,h) y^\gamma \Big)u + E(t,y) .
	\end{equation}
	As in Theorem \ref{T:halfwaveu}, $\|E(t,\cdot)\|_{L^2} \lesssim h^{N_0}$ and the coefficients $\tilde f_\gamma(\omega,h)$ satisfy the bounds in \eqref{fcoeffs}.
\end{corollary}

Since the present section does not use Proposition \ref{P:invwpt}, we allow $P$ to denote different PDOs than the one there as their use is isolated to \S\ref{S:WPansatz}.

The corollary is proved by taking the expansion \eqref{halfwaveu} furnished by Theorem \ref{T:halfwaveu}.  For each term, we write its image under $P$ similarly to  \eqref{weylshift}, then take a Taylor expansion of the symbol of $p$ about $(0,\xi)$ as before.  Arguing similarly to the above, using analogous versions of Lemmas \ref{L:polysymbolgenGaussian} and \ref{L:highordervanbd}, it then follows the image of each term under $P$ is a polynomial in $y$ times $u$ where the coefficients satisfy the bounds in \eqref{fcoeffs}.

\subsubsection{Relation with the classical Gaussian beam ansatz}\label{SSS:classicalbeams}
One classical approach to Gaussian beams (e.g. \cite{RalstonGaussianBeamsSurvey}, \cite{KatchalovKurylevLassasBook}) (in an arbitrary coordinate system) begins with an ansatz on a phase function $\tilde\phi(t,y)$ such that the eikonal quantity $	\prtl_t\tilde\phi +|d_y\tilde\phi|_{g(y)}$ vanishes to second order on $x_t$ (as opposed to vanishing identically).  It is then observed that if $\tilde\phi$ has the form
\begin{equation}\label{classicalansatz}
	\tilde\phi(t,y) = \xi_t\cdot(y-x_t) + W_t(y-x_t)\cdot (y-x_t) 
\end{equation}
the eikonal equation will then vanish at least to first order along $x_t$.  Another approach, taken by Combescure and Robert \cite{CombescureRobertBook}, is to take a quadratic expansion of the symbol of $h \sqrtl$, analogously to Lemma \ref{L:quadexpn}, then use the metaplectic representation to arrive at the same phase function.  Either way, it is necessary and sufficient that the entries of the matrix $W_t$ satisfy
\begin{equation}\label{Weqn}
	\dot{W}_{jk} + W_{j\ell} (\prtl_{\xi_\ell \xi_m} \tilde H) W_{m k} + W_{j\ell } (\prtl_{x_k \xi_\ell} \tilde H) + (\prtl_{x_j\xi_\ell } \tilde H) W_{\ell k} + \prtl_{x_j x_k} \tilde H =0,
\end{equation}
where $\tilde H(x,\xi) = |\xi|_{g(x)}$ as above.
In \cite{KatchalovKurylevLassasBook}, \cite{KatchalovLassas}, it is observed $W$ is not invariant under changes of coordinates, but that the following correction yields an invariant (0,2)-tensor on $T_{x_t}M$
\begin{equation}\label{invarianttensor}
	W_{jk}(t) - \Gamma_{jk}^\ell (x(t)) \xi_\ell (t)
\end{equation}
Moreover, it is seen there that if one raises an index of this tensor, then the resulting $(1,1)$ tensor satisfies the equation \eqref{riccatiansatz}.  In other words, $\omega$ is exactly the tensor in \eqref{invarianttensor}.

While not crucial for our proof, it is interesting to observe that $\phi$ in \eqref{phidef} agrees with \eqref{classicalansatz} up to third order in $y-x_t$.
Indeed, given Proposition \ref{P:covariantlogexpn}, we have in an arbitrary coordinate system
\begin{equation*}
	\phi= (\xi_t)_k (y-x_t)^k + \frac 12  \big( (\xi_t)_k \Gamma_{ij}^k(x_t) +(\omega_t)_{ij} \big) (y-x_t)^i (y-x_t)^j  + \mathcal{O}(|y-x_t|^3),
\end{equation*}
and the matrix expressing the purely quadratic term is exactly $W$.  Hence the main difference in the two approaches is that using $\expy$ is completely coordinate invariant (and vanishes at  $y=x_t$).

\subsection{Construction of the parametrix}\label{SS:paramcon}
In this section, we consider the approximate solutions  
\begin{equation}\label{udefparametrix}
	u(t,y,x,\xi) = (\pi h)^{-\frac d4} \psi(x_t(x,\xi),y)a(t,x,\xi)e^{\frac ih \phi(t,y,x,\xi)} 
\end{equation} 
where $a,\phi$ are as in \eqref{phidef}, \eqref{riccatiansatz}, \eqref{amplitudeansatz}.  In other words, we redefine $u$ as in \eqref{udef}, to include the cutoff discussed in Remark \ref{R:cutoff}, use notation that reflects the dependence on the point $(x,\xi) \in T^*M$.  We use this to define the operator $\Vsc_t$ we need in Theorem \ref{T:Vtthm} satisfying \eqref{Vproperties}, namely
\begin{equation}\label{Vtexpression}
 (\Vsc_t G)(y) = 	(2\pi h)^{-\frac d2 }\int_{T^* M} u(t,y,x,\xi) \frac{  \tilde\beta\big(|\xi|_{g(x)}\big) }{\detexp^{\frac 14}(g_{jk}(x))}G(x,\xi) \,dx d\xi.
\end{equation}
To see that $\Vsc_0 = \Ssc^*$, recall from \eqref{wptdef} that 
\begin{equation*}
	\begin{split}
				(\Ssc^*G)(y) &= 2^{-\frac d2} (\pi h)^{-\frac{3d}{4}}\int_{T^* M} e^{\frac ih(\xi(\expyno) + \frac i2 d^2(x,y))} \frac{ \psi(x,y) \tilde\beta\big(|\xi|_{g(x)}\big) }{\detexp^{\frac 14}(g_{jk}(x))}G(x,\xi) \,dx d\xi\\
				&= (2\pi h)^{-\frac d2 }\int_{T^* M} u(0,y,x,\xi) \frac{  \tilde\beta\big(|\xi|_{g(x)}\big) }{\detexp^{\frac 14}(g_{jk}(x))} G(x,\xi) \,dx d\xi,		
	\end{split}
\end{equation*}
where we used \eqref{initialdist} in the second identity and that $(x_t,\xi_t)|_{t=0} = (x,\xi)$.

We now use Theorem \ref{T:halfwaveu} to write
\begin{equation*}
	((hD_t+h\sqrtl) u)(t,y,x,\xi) = \pbf(t,y,x,\xi)u+ E(t,y,x,\xi) ,
\end{equation*}
where $\pbf(t,y,x,\xi) = \sum_{|\gamma| \leq N_1}f_\gamma(\omega_t(x,\xi),h) y^\gamma$  when $y$ is expressed in normal coordinates at $x_t$.
Alternatively, we can replace any $y_j$ in $y^\gamma$ by $dy_j|_{x_t} (\expy)$ in $\pbf$.  Since $\|E(t,\cdot,x,\xi)\|_{L^2} \lesssim h^{N_0}$ 
\begin{equation}\label{Econtrib}
	\bigg\| (2\pi h)^{-\frac d2 }\int_{T^* M} E(t,y,x,\xi)  \frac{  \tilde\beta\big(|\xi|_{g(x)}\big) }{\detexp^{\frac 14}(g_{jk}(x))}G(x,\xi) \,dx d\xi\bigg\|_{L^2(M)} \lesssim h^{N_0-\frac{d}{2}}\|G\|_{L^2(T^*M)} 
\end{equation}
which follows from Cauchy-Schwarz and Fubini's theorem, using that $\tilde \beta$ is compactly supported.  By taking $N_0$ large enough, the right hand side here yields a stronger contribution than the right hand side of \eqref{Vproperties}.

Now define 
\begin{equation}\label{etdef}
	(\E_tG)(y)	= (2\pi h)^{-\frac d2 }\int_{T^* M}  \pbf(t,y,x,\xi)u(t,y,x,\xi) \frac{  \tilde\beta\big(|\xi|_{g(x)}\big) }{\detexp^{\frac 14}(g_{jk}(x))}G(x,\xi) \,dx d\xi .
\end{equation}
As shown below, Theorem \ref{T:Vtthm} reduces to the following, which we prove in \S\ref{SS:errorestimates}.
\begin{theorem}\label{T:paramerror} 
	Let $\tilde{\tilde\beta} $ be a bump function supported in $(0,\infty)$ such that $\tilde{\tilde\beta} \tilde\beta = \tilde\beta$, that is, $\tilde{\tilde\beta}$ is identically 1 on the support of $\tilde\beta$.  Let $\Sscwt$ be the wave packet transform adapted to $\tilde{\tilde\beta}$.  Its composition with $\E_t$ satisfies
	\begin{equation}\label{paramerror}
		\| \Sscwt \circ \E_t \|_{L^2(T^*M) \to L^2(T^*M)} \lesssim h^{\frac 32} \mu^{2d+9}.
	\end{equation}
Moreover, if $P$ is a PDO with symbol in $S(1)$, then
	\begin{equation}\label{paramerror2}
	\| \Sscwt \circ P \circ \E_t \|_{L^2(T^*M) \to L^2(T^*M)} \lesssim h^{\frac 32} \mu^{2d+9}.
\end{equation}
\end{theorem}

The rest of the section will show that \eqref{Vproperties} follows from Theorem \ref{T:paramerror}.  Define the PDO $P := h^{-1} (\Sscwt^* \circ \Sscwt - \tilde{\tilde\beta}^2(h\sqrtl))$, which has a symbol in $S(1)$ by Proposition \ref{P:wptpdo}.  We write 
\begin{equation*}
	\E_t = (I-\Sscwt^* \circ \Sscwt) \circ \E_t + \Sscwt^* \circ \Sscwt \circ \E_t = \big(I-\tilde{\tilde\beta}^2(h\sqrtl)\big) \circ \E_t +hP \circ \E_t + \Sscwt^* \circ \Sscwt \circ \E_t,
\end{equation*} 
hence \eqref{paramerror} and the usual $L^2$ bounds on $\Ssc^*$ \eqref{wptL2bds} imply that the $L^2(T^*M) \to L^2(M)$ norm of the last term is $\mathcal{O}(h^{\frac 32} \mu^{2d+9})$.  Moreover,
\begin{equation}\label{strongerror}
 	\Big\|\big(I-\tilde{\tilde\beta}^2(h\sqrtl)\big) \circ \E_t\Big\|_{L^2(T^*M) \to L^2(M)}  = \mathcal{O}(h^\infty).
\end{equation}
To see this, recall that by the same idea in the proof of Proposition \ref{P:invwpt}, we have that up to $\mathcal{O}(h^\infty)$ error, the support of the symbol of $I-\tilde{\tilde\beta}^2(h\sqrtl)$ is contained in $\overline{\{(x,\xi): 1-\tilde{\tilde\beta}^2(|\xi|_{g(x)}) \neq 0\}}$.
Thus if $\xi $ is such that $\tilde\beta(|\xi|_{g}) \neq 0 $ (cf. the integral \eqref{etdef}), then its distance to the support of $(1-\tilde{\tilde\beta})(|\xi|_{g})$ is uniformly bounded below, so it is routine to see that $\|(I-\tilde{\tilde\beta}^2(h\sqrtl))u(t,\cdot,x,\xi)\|_{L^2(M)} = \mathcal{O}(h^\infty)$ as this is realized by an oscillatory integral without critical points.

It remains to show that
$	\|P \circ \E_t\|_{L^2(T^*M) \to L^2(M)}  \lesssim  h^{\frac 12} \mu^{2d+9}.
$ One might expect that this contribution would be much stronger than needed, exploiting that the symbol of $P$ can be taken to be supported away from where $\tilde\beta(|\xi|_{g(x)}) \neq 0$ as before, but the proof of Proposition \ref{P:wptpdo} in Appendix \ref{A:wptpdo} does not reveal this.  
To see this, we again write $P \circ \E_t = (I-\Sscwt^* \circ \Sscwt) \circ P \circ \E_t + \Sscwt^* \circ \Sscwt \circ P \circ \E_t$, so that \eqref{paramerror2} implies that the second term satisfies a stronger bound than needed.  Moreover, we may replace $\E_t$ by $P\circ\E_t$ in \eqref{strongerror} by reasoning similarly as above: for each $x,\xi$, we can take a Taylor expansion of $P$ about $(x,\xi)$ to sufficiently high order, then argue as in Lemmas \ref{L:polysymbolgenGaussian} and \ref{L:highordervanbd}.  The distorted Gaussians resulting from the former satisfy $\mathcal{O}(h^\infty)$ bounds as in \eqref{strongerror}, while the latter bounds the error resulting from the Taylor expansion.  Hence for $N_0$ large enough
\begin{equation*}
	\|(I-\Sscwt^* \circ \Sscwt) \circ P \circ \E_t \|_{L^2(T^*M) \to L^2(M)}  \lesssim  h\| P^2 \circ \E_t \|_{L^2(T^*M) \to L^2(M)} + \mathcal{O}(h^{N_0}),
\end{equation*}
which reduces us to showing that $\| P^2 \circ \E_t \|_{L^2(T^*M) \to L^2(M)} \lesssim  h^{-\frac 12} \mu^{2d+9}$.  

Iterating this argument $N-1$ more times reduces us to $\| P^N \circ \E_t \|_{L^2(T^*M) \to L^2(M)} \lesssim  h^{\frac 32-N} \mu^{2d+9}$.  However, $\|P\|_{L^2(M) \to L^2(M)} \lesssim 1$ and the $E \equiv 1$ case in \eqref{Econtrib} implies we always have the crude bound $\|\E_t\|_{L^2(T^*M)\to L^2(M)} \lesssim h^{-\frac d2}$.  Hence by taking $N$ sufficiently large, we conclude \eqref{Vproperties}.

\section{Phase Space Kernels}\label{S:matrixelements}
In this section we prove Theorems \ref{T:paramerror} and \ref{tdepreduction}, thus completing the proofs of our main theorems.  Both of them will result from establishing asymptotics on the following integrals, where the integrand is expressed in normal coordinates centered at $x_t$:
\begin{equation}\label{matrixeltgamma}
	\begin{aligned}
		\Gsc_\gamma(z,\zeta,t,x,\xi) &:= 2^{-\d}(\pi h)^{-\frac{3\d}{2}}\detexp^{-\frac 12}(A_t+iB_t)\int_{M} e^{\frac ih \Phi(t,x,y,z,\xi,\zeta)} y^\gamma \mathbf{A}(t,x,y,z,\xi,\zeta)dv_g(y), \\
	\Phi(t,x,y,z,\xi,\zeta) &:=	-\zeta(\expynobz) + \xi_t(\expy ) + \frac i2 d^2(z,y) + \frac 12\omega_t(\expy,\expy),\\
	\mathbf{A} (t,x,y,z,\xi,\zeta) &:= \frac{  \tilde\beta(|\xi|_{g(x)}) \tilde\beta(|\zeta|_{g(z)}) }{\detexp^{\frac 14}(g_{jk}(x))\detexp^{\frac 14}(g_{jk}(z))}\psi(z,y)\psi(x_t,y).
	\end{aligned}
\end{equation}
Here $\detexp^{-\frac 12}(A_t+iB_t) = a(t)$ as in \eqref{amplitudeansatz}.  In the interest of consistency with the construction in \S\ref{SS:ansatz}, we also assume that $(\xi_t)_j = |\xi_t|_g \delta_{jd}$.  In the proof of Theorem \ref{T:matrixeltasymp} below, we will assume that $\mathbf A$ is supported where $|y-z| \ll 1$ and $|y|\ll 1$.  Strictly speaking this is smaller than what is given by our original definition of $\psi$, but the integral over the complement of this region is $\mathcal{O}(h^\infty)$.

The integral $\Gsc_\gamma$ expresses coordinate invariant quantities despite the choice of normal coordinates we use to calculate the integral.  Indeed, the kernel of $\Ssc \circ \Vsc_t$ (mapping functions of $(x,\xi)$ to $(z,\zeta)$) is simply $\Gsc_0$, which is of concern for Theorem \ref{tdepreduction}.  Moreover harmlessly replacing $\tilde\beta(|\zeta|_{g(z)})$ by $\tilde{\tilde\beta}(|\zeta|_{g(z)}) $, the definition in \eqref{etdef} means the integral kernel of $\Sscwt \circ \E_t$ takes the form
\begin{equation*}
	\sum_{|\gamma| \leq N_1}\Gsc_\gamma(z,\zeta;t,x,\xi)f_\gamma(\omega_t(x,\xi),h),
\end{equation*}
where $f_\gamma$ satisfies \eqref{fcoeffs}; this is our primary concern for proving Theorem \ref{T:paramerror}.

\subsection{Asymptotics of the phase space kernels}  We begin by setting the stage for our main result on the asymptotic behavior of \eqref{matrixeltgamma}.  We use Proposition \ref{P:covariantlogexpn} to expand $\expynobz$, using that $d(z,y) = |\expynobz|_g$, but since normal coordinates are taken at $x$, the components of $\expyno$ are simply $y_j \frac{\prtl}{\prtl y_j}|_0$ as in \S\ref{SS:ansatz}.  However, $|\Gamma_{kl}^j(z)| + |g_{k\ell}(z)-\delta_{k\ell}|\lesssim |z|$, so 
\begin{equation*}
	\begin{split}
		\Phi &=-\zeta\cdot(y-z) + |\xi_t|_g y_d - \frac 12\big( \zeta_j\Gamma_{k\ell}^j(z) +i g_{k\ell} (z)\big) (y-z)^k(y-z)^\ell + \frac 12 \omega_t y\cdot y +\mathcal{O}(|y-z|^3)\\
	&=-\zeta\cdot(y-z) + |\xi_t|_g y_d + \frac i2 |y-z|^2 + \frac 12 \omega_t y\cdot y +\mathcal{O}(|y-z|^3 + |z||y-z|^2)\\
	&=\zeta\cdot z + \frac i2|z|^2 - y' \cdot (\zeta'+iz') + y_d(|\xi_t|_g-\zeta_d-iz_d) + \frac 12 (\omega_t+i)y\cdot y +r(x,y,z,\xi,\zeta)
	\end{split}
\end{equation*}
Here $r = r_1 + r_2$ expresses the error term in the preceding line, namely for some $r_{1,\nu}$, $ r_{2,j,\nu}$ uniformly bounded in $C^\infty$,
\begin{equation*}
	r_1 = \sum_{|\nu|=3}(y-z)^\nu r_{1,\nu}(x,y,z,\xi,\zeta), \qquad 	r_2=\sum_{j=1}^\d \sum_{|\nu|=2}z_j(y-z)^\nu r_{2,j,\nu}(x,y,z,\xi,\zeta) .
\end{equation*}
We abbreviate the quadratic part of $\Phi$ as 
\begin{equation*}
	\phi(x,y,z,\xi,\zeta) = \zeta\cdot z + \frac i2|z|^2 - y' \cdot (\zeta'+iz') + y_d(|\xi_t|_g-\zeta_d-iz_d)+\frac 12 (\omega_t+i)y\cdot y, \quad \text{ so that }\quad \Phi = \phi + r.
\end{equation*}
We also define $\eta$ as
\begin{equation*}
	\eta := (\omega_t+i)^{-1}(\zeta'+iz',\zeta_d+iz_d-|\xi_t|_g) = \big( (\tilde\omega_t +i)^{-1}(\zeta'+iz'), \frac i2 (|\xi_t|_g-\zeta_d -iz_d) \big)
\end{equation*}
so that $\eta$ plays the role of ``$\xi$'' in Theorem \ref{T:spquad}.  The theorem leads us to define
\begin{equation}\label{psidef}
 \psi(x,z,\xi,\zeta) =	
	\zeta\cdot z + \frac i2|z|^2 -\frac 12(\tilde\omega_t +i)^{-1}(\zeta'+iz')\cdot(\zeta'+iz')  + \frac i4\big(|\xi_t|_g-\zeta_\d -i z_\d\big)^2
\end{equation}
\begin{theorem}\label{T:matrixeltasymp} 
Under the hypotheses above, the integral \eqref{matrixeltgamma} admits an asymptotic expansion as $h \to 0$
\begin{multline}\label{matrixeltasymp}
	2^{-\frac{\d}{2}}(\pi h)^{-\d}\detexp^{-\frac 12}(A_t+iB_t)\detexp^{-\frac 12}(-i(\omega_t +i))e^{\frac ih \psi}\sum_{k=0}^N \sum_{|\nu| \leq 3k+|\gamma|} h^{\frac{|\nu|}{2}-k} \mathbf{P}_{\gamma,k,\nu,h}(z,\eta) \\+ \mathcal{O}\big(h^{\frac{|\gamma| + 3N-3\d}{2}} \mu^{|\gamma|+N}\big). 
\end{multline}
Here $\mathbf{P}_{\gamma,k,\nu,h}(z,\eta)$ are homogeneous polynomials in $(z,\eta)$ of degree $3k+|\gamma|-|\nu|$, whose coefficients may depend on $(z,\zeta,t,x,\xi)$, but are uniformly bounded in $C^\infty$.
\end{theorem}
\begin{proof}
Define
\begin{equation*}
\I(s) = (2\pi h)^{-\frac{\d}{2}} \int e^{\frac ih (\phi + s r)}y^\gamma \mathbf{A}\,dy,
\end{equation*}
so that the integral \eqref{matrixeltgamma} is just
$
	2^{-\frac{\d}{2}}(\pi h)^{-\d}\detexp^{-\frac 12}(A_t+iB_t)\I(1)
$. We get an asymptotic expression for $\I(1)$ via the Taylor expansion 
\begin{equation}\label{kthderivI}
 	\begin{aligned}
 		\I(1) &= \sum_{k=0}^{N-1} \frac{ \I^{(k)}(0)}{k!} + \frac{1}{N!} \I^{(N)}(s_0), \quad \text{ where }0<s_0<1,\\
 		\I^{(k)}(s) &= (2\pi h)^{-\frac{\d}{2}} h^{-k}\int e^{\frac ih (\phi + s r)}y^\gamma (ir)^k \mathbf{A}\,dy.		
 	\end{aligned}
\end{equation}

We begin by claiming that for each $0\leq s \leq 1$, we have the uniform bound
\begin{equation}\label{INerrorbd}
	\big|\I^{(N)}(s)\big| \lesssim h^{-\frac{\d}{2}-N} \int e^{-\frac 1h \Im (\phi+sr)} |y|^{|\gamma|} |r|^N |\mathbf{A}| \,dy 
	\lesssim  h^{\frac{|\gamma|+3N}{2} }\mu^{|\gamma| + N} 
\end{equation}
To see this note that for some uniform constant $C$,
\begin{equation}\label{rbd}
|r| \leq C|y-z|^2\big(|y-z|+|z|\big) \leq 2C|y-z|^3+ C|y-z|^2|y| 
\end{equation}
Using only the first inequality here, we see that
\begin{equation*}
\Im (\phi + s r)  = 
\frac 12 |z-y|^2 + \frac 12 (\Im\omega) y\cdot y + s\Im r 
\geq \frac 12 |z-y|^2 + \frac 12 \mu^{-2} |y|^2 - C|y-z|^2(|y-z|+|z|). 
\end{equation*}
Since we can take $\mathbf{A}$ to be supported in a region where $|y-z| \ll 1$ and $|y|\ll 1$, we have that $|z|\ll1$ and hence by taking these supports which are small enough, we can assume
\begin{equation*}
\Im (\phi + sr)  
\geq 
\frac 14 |z-y|^2 + \frac 14 \mu^{-2} |y|^2 .
\end{equation*}
The bound \eqref{rbd} on $|r|$ above now gives that 
\begin{equation*}
	\int e^{-\frac 1h \Im (\phi+sr)} |y|^{|\gamma|} |r|^N |\mathbf{A}| \,dy \lesssim \sum_{\ell = 0}^{N} \binom{N}{\ell} \int e^{-\frac 1{4h}(|y-z|^2 +\mu^{-2}|y|^2) } |y|^{|\gamma| + \ell} |y-z|^{3N-\ell}  \,dy,
\end{equation*}
where we have absorbed the constant $C$ and the contribution of $\sup |\mathbf{A}|$ into the implicit constant on the right. The bound \eqref{INerrorbd} now follows from an application of Cauchy-Schwarz which bounds each term by the product of integrals in $|y|$ and $|y-z|$ alone.

We now consider $\I^{(k)}(0)$, as in $k$-th term in the expansion \eqref{kthderivI}.  We write the amplitude as
\begin{equation*}
y^\gamma (ir)^k \mathbf{A} = \sum_{|\nu| \leq 3k +|\gamma|} \widetilde{\mathbf{P}}_{\gamma,k,\nu,h}(z,\eta) (y-\eta)^{\nu} \mathbf{A}_\nu(x,y,z,\xi,\eta)
\end{equation*}
where $\widetilde{\mathbf{P}}_{\gamma,k,\nu,h}(z,\eta) $ is a homogeneous polynomial of degree  $3k+|\gamma|-|\nu|$ in $(z,\eta)$.  This yields a sum of oscillatory integrals which can be treated by Theorem \ref{T:spquad}, using that \eqref{siegelbd3}, \eqref{siegelbd4} ensure that the crucial hypotheses in \eqref{spphasehyp} are satisfied. Since each term $(y-\eta)^\nu$ vanishes to order $|\nu|$, the approximating sum in \eqref{asympexpn} can be restricted to $j \geq \lceil \frac{|\nu|}{2} \rceil$.  The result follows by taking the asymptotic expansion in Theorem \ref{T:spquad} out to sufficiently high order.
\end{proof}

\subsection{Singular values and a return to the Siegel disk}  Here we examine how to express the phase function $\psi$ defined in \eqref{psidef} in terms of the singular values of \eqref{blocktildeomega}.  Recall that we are taking conventions consistent with \S\ref{SS:ansatz}, in particular $\omega_t$ has the structure \eqref{blockomega} where $\tilde\omega_t = (C_s+iD_s)(A_s+iB_s)^{-1}$ with $s = t/|\xi|_g$ as in \eqref{blocktildeomega}.  
We now return to the identity in \eqref{halftodisk} and the discussion in Remark \ref{R:cayley} that $(\tilde\omega_t+i)^{-1} = \frac{1}{2i}(I+W_t)$ where $W_t = \alpha[(\Fsc_t)_c](0)$ and $\Fsc_t$ is the block matrix in \eqref{blocktildeomega}. In the notation of \eqref{cplxdympldef}, we have
\begin{equation*}
	W_t = Z_t Y_t^{-1}, \text{ where } 2Z_t = A_t-D_t+i(B_t+C_t) \text{ and } 2Y_t = A_t+D_t+i(B_t-C_t)
\end{equation*}
Hence
\begin{equation*}
	\begin{split}
		\psi &= 	\zeta\cdot z + \frac i2|z|^2 -\frac 12(\tilde\omega_t +i)^{-1}(\xi_t'-\zeta'-iz')\cdot(\xi_t'-\zeta'-iz')  + \frac i4(|\xi_t|_g-\zeta_\d -i z_\d)^2\\
		&= \zeta\cdot z + \frac i2 |z|^2 + \frac i4 (I+W_t)(\xi_t'-\zeta'-iz')\cdot(\xi_t'-\zeta'-iz')  + \frac i4(|\xi_t|_g-\zeta_\d -i z_\d)^2\\
		&= \frac 12\zeta\cdot z +\frac 12 |\xi_t|_g z_d+\frac i4 |z|^2 + \frac i4 |\zeta'|^2 + \frac i4(|\xi_t|_g-\zeta_d)^2+\frac i4 W_t(\zeta'+iz')\cdot(\zeta'+iz')		
	\end{split}
\end{equation*}
By the results in \S\ref{SS:SiegelPlaneDisk}, namely Remark \ref{R:alphaactionsvd}, we have that for some $(d-1)\times (d-1)$ unitary $V_t$,
\begin{equation*}
	W_t = V_t^{-T} \Lambda_t V_t^{-1}, \qquad \Lambda_t = \diag\Big(\frac{\sigma_1 - \sigma_1^{-1}}{\sigma_1 + \sigma_1^{-1}}, \dots, \frac{\sigma_{\d-1} - \sigma_{\d-1}^{-1}}{\sigma_{\d-1} + \sigma_{\d-1}^{-1}}\Big) ,
\end{equation*}
where $\sigma_j = \sigma_j(t/|\xi|)$ denote the first $d-1$ singular values of \eqref{blocktildeomega}, all of which are at least 1. We are now led to set
\begin{equation*}
	\upsilon = (\upsilon', \upsilon_d)=(V^{-1}(\zeta'+iz'),\zeta_\d-|\xi|_g+iz_\d).
\end{equation*}
Note that since $V$ is unitary, $  |z|^2 + |\zeta'|^2 + (|\xi|_g-\zeta_d)^2= |\upsilon|^2$.  Also, since 
\begin{equation}\label{upsilonswitch}
	(\omega_t +i)\eta=(\zeta'+iz',\zeta_d-|\xi|_g+iz_d) =  (V\upsilon',\upsilon_d),
\end{equation}
it follows that the components of $z = \Im (V\upsilon',\upsilon_d)$ can be written as a polynomial function of $\Re\upsilon,\Im\upsilon$ and the matrix entries of $V$.  Moreover, $\eta$ can be written as a polynomial in $\Re\upsilon,\Im\upsilon$ and the matrix entries of $V$, $(\omega+i)^{-1}$, the latter of which are uniformly bounded (cf. \eqref{siegelbd4}).

We now express the imaginary part of $\psi$ in terms of $\upsilon$ 
\begin{equation}\label{impsiupsilon}
	\begin{split}
		\Im \psi &= \frac 14|\upsilon|^2 + \frac 14 \sum_{j=1}^{\d-1} \frac{\sigma_j - \sigma_j^{-1}}{\sigma_j + \sigma_j^{-1}} \big((\Re \upsilon)_j^2 -(\Im \upsilon)_j^2 \big)\\ &=
		\frac 12 \sum_{j=1}^{\d-1} \frac{\sigma_j}{\sigma_j + \sigma_j^{-1}} (\Re \upsilon)_j^2  +
		\frac 12 \sum_{j=1}^{\d-1} \frac{\sigma_j^{-1}}{\sigma_j + \sigma_j^{-1}} (\Im \upsilon)_j^2  + \frac 14|\upsilon_\d|^2.
	\end{split}
\end{equation}
Moreover, with $2Y_t = A_t+D_t+i(B_t-C_t)$ as in \eqref{cplxdympldef},
\begin{align*}
	\det(-i(\omega_t+i))\det(A_t+iB_t) &= 2\det(-i(\tilde\omega_t+i))\det(A_t+iB_t)\\
	&= 2\det\big( (I-i(C_t+iD_t)(A_t+iB_t)^{-1})(A_t+iB_t) \big) \\
	&= 2\det(A_t+iB_t-i(C_t+iD_t)) = 2\det(A_t+D_t +i(B_t-C_t))\\
	&= 2\det(2Y_t) = 2^\d \det Y_t.
\end{align*}
Thus by Corollary \ref{C:alphaactionsvd} and Remark \ref{R:alphaactionsvd}, using that $\det Y_t = \delta((\Fsc_t)_c,0)$
\begin{equation*}
	|\det(-i(\omega_t+i))\det(A_t+iB_t)|^{-\frac 12} = 2^{-\frac{\d}{2}} |\det Y_t|^{-\frac 12} = 2^{-\frac d2}
	\prod_{j=1}^{\d-1} \Big(\sigma_j + \frac{1}{\sigma_j}\Big)^{-\frac 12} \leq 2^{-\frac{\d}{2}} \vartheta(t)^{-\frac 12}
\end{equation*}

We also observe that since $\sigma_j \geq 1$ for $j=1,\dots, d-1$, we have $\frac{\sigma_j^{-1}}{\sigma_j + \sigma_j^{-1}} \geq \frac{1}{2}\sigma_j^{-2}$ so \eqref{impsiupsilon}  and the identity \eqref{upsilonswitch} implies that 
\begin{equation*}
	\Im \psi \geq \mu^{-2}\big( |z'|^2+| \zeta'|^2\big)+ \big( |z_\d|+| |\xi_t|_g-\zeta_\d| \big)^2
\end{equation*}
Incorporating \eqref{varrholowerhyp}, we observe the following crude, but useful bound
\begin{equation*}
	|\Gsc_\gamma(z,\zeta,t,x,\xi)| \lesssim_N h^{-\d}\vartheta^{-\frac 12}\left(1+h^{-\frac 12}\mu^{-1}\big( |z'|+| \zeta'|\big)+ h^{-\frac 12}\big( |z_\d|+| |\xi_t|_g-\zeta_\d|\big) \right)^{-N}.
\end{equation*}
 Since our normal coordinate system is centered at $x_t$ with $(\xi_t)_j = |\xi_t|_{g}\delta_{jd}$, we also have the following coordinate-free characterization of the same bound (where the metric/cometric are taken at $x_t$):
\begin{multline}\label{crudeme}
	|\Gsc_\gamma(z,\zeta,t,x,\xi)| \lesssim_N \\ 
	h^{-\d}\vartheta^{-\frac 12}\left(1+h^{-\frac 12}\mu^{-1}\big( |\expz|_g+| \xi_t-\zeta|_g\big)+ h^{-\frac 12}|\xi_t|_{g}^{-1}\big( |\xi_t(\expz)|+ | \langle\xi_t,\xi-\zeta \rangle_g |\big) \right)^{-N}.
\end{multline}
Moreover, in an arbitrary coordinate system, we can always approximate $|z-x_t| \approx |\expy|_g$ and $|\xi_t-\zeta|\approx |\xi_t-\zeta|_g$, where each left hand side uses standard Euclidean length.  Hence we may replace the right hand side in \eqref{crudeme} by
\begin{equation}\label{crudemearbcoord}
	h^{-\d}\vartheta^{-\frac 12}\left(1+h^{-\frac 12}\mu^{-1}\big( |z-x_t| + |\xi_t-\zeta| \big)+ h^{-\frac 12}|\xi_t|_{g}^{-1}\big( |\xi_t(\expz)|+ | \langle\xi_t,\xi-\zeta \rangle_g |\big) \right)^{-N}.
\end{equation}

We also have the following lemma which will be used in \S\ref{SS:errorestimates} below.
\begin{lemma}\label{L:gsclemma}
	Any $\Gsc_\gamma(z,\zeta,t,x,\xi)$ as in \eqref{matrixeltgamma} satisfies the $L^2$ bound
	\begin{equation}\label{matrixL2} 
		\sup_{x,\xi}\Big(\int |\Gsc_\gamma(z,\zeta,t,x,\xi)|^2 \,dz d\zeta \Big)^{\frac 12} \lesssim_\gamma (h^{\frac 12}\mu)^{|\gamma|} h^{-\frac d2}
	\end{equation}
\end{lemma}

\begin{proof}
Consider the $L^2$ norm of any term in \eqref{matrixeltasymp} and change variables $(z,\zeta)$ to $(\upsilon, \bar\upsilon)$
	\begin{multline}\label{upsilswitch}
		h^{\frac{|\nu|}{2} -k}(\sqrt2 \pi h)^{-d}\left( |\det(-i(\omega_t+i))\det(A_t+iB_t)|^{-1} \int | e^{\frac ih \psi} \mathbf{P}_{\gamma,k,\nu,h}(z,\eta)|^2\,dzd\zeta \right)^{\frac 12} \\ =h^{\frac{|\nu|}{2} -k}(2 \pi h)^{-d}\left(\prod_{j=1}^{\d-1} \Big(\sigma_j + \frac{1}{\sigma_j}\Big)^{-1} \times \int  e^{-\frac{2}{h} \Im \psi}\, |\mathbf{P}_{\gamma,k,\nu,h}(\Re\upsilon,\Im\upsilon)|^2\,d(\Re\upsilon) d(\Im\upsilon) \right)^{\frac 12},
	\end{multline}
where we allow $\mathbf{P}_{\gamma,k,\nu,h}(\Re\upsilon,\Im\upsilon)$ to denote the polynomial $\mathbf{P}_{\gamma,k,\nu,h}$ expressed in terms of $\upsilon$ (cf. \eqref{upsilonswitch} and the ensuing observation).  This is harmless as it does not affect that the coefficients of the polynomial remain uniformly bounded.  The expression of $\psi$ in \eqref{impsiupsilon} leads to further change variables $\upsilon_\d \mapsto h^{\frac 12} \upsilon_\d $ and
\begin{equation*}
	 \Re\upsilon_j \mapsto h^{\frac 12} \Big(\frac{\sigma_j}{\sigma_j+ \sigma_j^{-1}}\Big)^{-\frac 12} \Re\upsilon_j, \quad \Im\upsilon_j \mapsto h^{\frac 12} \Big(\frac{\sigma_j^{-1}}{\sigma_j+ \sigma_j^{-1}}\Big)^{-\frac 12} \Im\upsilon_j, \quad j=1,\dots,d-1.
\end{equation*}
Note that the determinant of this transformation cancels the product $\prod_{j=1}^{\d-1} \Big(\sigma_j + \frac{1}{\sigma_j}\Big)^{-1}$ in \eqref{upsilswitch} while gaining $h^{\frac d2}$.  It now follows that since $\mathbf{P}_{\gamma,k,\nu,h}$ is homogeneous of degree $3k+|\gamma|-|\nu|$, the right hand side of \eqref{upsilswitch} is bounded by 
\begin{equation*}
	h^{\frac{|\nu|}{2}-k}(h^{\frac 12} \mu)^{3k+|\gamma| -|\nu|} h^{-\frac d2}= (h^{\frac 12}\mu^3)^k (h^{\frac 12} \mu)^{|\gamma| }\mu^{-|\nu|}h^{-\frac d2} \lesssim (h^{\frac 12} \mu)^{|\gamma| }h^{-\frac d2}.
\end{equation*}  The bound \eqref{matrixL2} now follows by taking $N$ sufficiently large in \eqref{matrixeltasymp}.
\end{proof}

\subsection{Error estimates}\label{SS:errorestimates} We now prove Theorem \ref{T:paramerror} by showing \eqref{paramerror} as the considerations for \eqref{paramerror2} will follow similarly. Given \eqref{etdef}, write can the integral kernel of $\Sscwt \circ \E_t$ (mapping functions of $(x,\xi)$ to $(z,\zeta)$) as 
\begin{equation*}
		\sum_{|\gamma| \leq N_1}\Gsc_\gamma(z,\zeta;t,x,\xi)f_\gamma(\omega_t,h),
\end{equation*}
where $\Gsc_\gamma$ is of the form \eqref{matrixeltgamma}.
It suffices to show that each term here yields an operator which maps $L^2(T^*M) \to L^2(T^*M)$ with norm which is $\mathcal{O}(h^{\frac 32}\mu^{2d+9})$.  We show the details for the cases with $|\gamma|\geq 3$;  the cases with $0 \leq |\gamma| \leq 2$ are treated by similar means and satisfy better bounds as the coefficients in \eqref{fcoeffs} are more favorable in these cases.  

By \eqref{matrixL2} and the coefficient bounds \eqref{fcoeffs} we have
\begin{equation*}
	\Big(\int |\Gsc_\gamma(z,\zeta,t,x,\xi)f_\gamma(\omega_t,h)|^2 \,dz d\zeta \Big)^{\frac 12} \lesssim_\gamma h^{-\frac d2}  (h^{\frac 12}\mu)^{|\gamma|} \mu^{2|\gamma|} = h^{-\frac d2} (h^{\frac 12}\mu^3)^{|\gamma|} \quad \text{ if } |\gamma| \geq 3.
\end{equation*}
Moreover, since $\mu \leq h^{-\frac{1}{26}}$ (cf. \eqref{mubddh}), we have by \eqref{crudemearbcoord}
\begin{equation}\label{TTstarsmall}
	\left| \int \Gsc_\gamma (z,\zeta,t,x,\xi) \,\overline{\Gsc_\gamma (z,\zeta;\tilde x,\tilde \xi) }\,dz d\zeta \right| \lesssim_N h^N \quad \text{ if } |x_t -\tilde x_t|+ |\xi_t - \tilde\xi_t| \gg h^{\frac 12}\mu
\end{equation}
But since $|x-\tilde x| +|\xi-\tilde \xi| \lesssim \mu(|x_t -\tilde x_t|+ |\xi_t - \tilde\xi_t| )$ since $\|d\tilde\kappa_{-t}\| \lesssim \mu(-t)=\mu(t)$, we now see that the bound \eqref{TTstarsmall} holds if $|x-\tilde x| +|\xi-\tilde \xi| \gg h^{\frac 12}\mu^2$.  It now follows that
\begin{equation}\label{crudeyoungs}
	\left| \int\Gsc_\gamma  (z,\zeta,t,x,\xi) \overline{\Gsc_\gamma  (z,\zeta,t,\tilde x,\tilde \xi) }\,dzd\zeta \right|  
	\lesssim
	\begin{cases}
		h^{N} & \text{ if }  |x-\tilde x| +|\xi-\tilde \xi| \gg h^{\frac 12}\mu^2\\
		h^{-d} (h^{\frac 12}\mu^3)^{2|\gamma|} &\text{ if } |x-\tilde x| +|\xi-\tilde \xi| \lesssim h^{\frac 12}\mu^2
	\end{cases}
\end{equation}
Consequently,
\begin{multline*}
	\sup_{\tilde x,\tilde\xi} \int \left| \int \Gsc_\gamma (z,\zeta,t,x,\xi) \overline{ \Gsc_\gamma (z,\zeta;\tilde x,\tilde \xi) }\,dzd\zeta \right|  \,dxd\xi \\
	\lesssim h^N+ \sup_{\tilde x,\tilde\xi} h^{-d}  (h^{\frac 12}\mu^3)^{2|\gamma|} \text{Vol}\Big( |x-\tilde x| +|\xi-\tilde \xi| \lesssim h^{\frac 12}\mu^2\Big) \lesssim 
 (h^{\frac 12}\mu^3)^{2|\gamma|} \mu^{4d}
\end{multline*}
and the same holds if we reverse the roles of $(x,\xi)$, $(\tilde x,\tilde \xi)$.
By Young's inequality applied to  $T^*T$ it now follows that 
\begin{equation*}
\left( \int \left| \int \Gsc_\gamma (z,\zeta,t,x,\xi) G (x,\xi) \,dxd\xi \right|^2 dzd\zeta \right)^{\frac 12}\lesssim \mu^{2d}(h^{\frac 12}\mu^3)^{|\gamma|} 
\|G\|_{L^2_{x,\xi}}.
\end{equation*}
The right hand side here is maximized when $|\gamma| =3$, yielding the desired $\mathcal{O}(h^{\frac 32}\mu^{2d+9})$ bound.

\begin{remark}
	The proof here reveals why we have the right hand sides of $h^{\frac 32}\mu^{2d+9}$ in Theorem \ref{T:paramerror}.  Admittedly the bounds in \eqref{crudeyoungs} are somewhat crude, though improving them appears to be a subtle matter which we do not address in the present work.  Even then, a better approach might be to find approximate solutions $u$ in \S\ref{S:WPansatz} with higher order accuracy as noted in Remark \ref{R:baderror}.
\end{remark}

\subsection{Microlocal Kakeya-Nikodym bounds}
Here we prove Theorem \ref{tdepreduction}, namely
\begin{equation*} 
	\|Q_{h,\tilde\gamma} \circ P^* \circ \Ssc^* \circ (\Ssc \circ \Vsc_t) \circ \Ssc \circ P\circ Q_{h,\tilde\gamma}^*\|_{L^2(M) \to L^2(M)} \lesssim \frac{1}{\sqrt{\vartheta(t)}} .
\end{equation*}
As we have before, we work in Fermi coordinates which flatten the geodesic segment about which $Q_{h,\tilde\gamma}$ is adapted so that its symbol satsfies \eqref{hsymbolhypintro}.  For $N_0\gg2d$ sufficiently large, we use the multipliers $\mathscr{M}_{\pm N_0}$ defined in Theorem \ref{T:multipliers}.  
 Recall from \eqref{mnodef} that $\Mnominus$ is the multiplier operator given by (in a slight abuse of notation, we use $\Mnominus$ to denote both the multiplier and the operator)
\begin{equation*}
	\Mnominus (x,\xi) =  \big(1+h^{-1}|x'| ^2+ h^{-1}|\xi'|^2 + h^{-1}d^2\big((x,\xi);\supp(q_{h,\tilde\gamma})\big) \big)^{-N_0/2}.
\end{equation*}
Given the $L^2$ bounds \eqref{multbds} there, it suffices to show that
\begin{equation}\label{tdepreductionmat}
	\|\Mnominus \circ (\Ssc \circ \Vsc_t) \circ \Mnominus\|_{L^2(T^*M) \to L^2(T^*M)} \lesssim \frac{1}{\sqrt{\vartheta(t)}} .
\end{equation}

To show \eqref{tdepreductionmat}, we use that the kernel of $\Ssc \circ \Vsc_t$ is expressed by $\Gsc_0$.
By Young's inequality, it follows from these bounds in coordinates
\begin{align}
		\int_{T^*M} 	|\Gsc_0(z,\zeta,t,x,\xi)| 	\Mnominus(z,\zeta)\Mnominus(x,\xi)\,dz\,d\zeta \lesssim \vartheta^{-\frac 12}(t)\label{zzeta}\\  
		\int_{T^*M} 	|\Gsc_0(z,\zeta,t,x,\xi)| 	 \Mnominus(z,\zeta)\Mnominus(x,\xi) \,dx\,d\xi \lesssim \vartheta^{-\frac 12}(t)\label{xxi}
\end{align}
Throughout the treatment below, we can assume that $|z'|+|\zeta'| \leq h^{\frac 38}$, $|\zeta-e_d| \ll 1$ and $|x'|+|\xi'| \leq h^{\frac 38}$, $|\xi -e_d| \ll 1$, in particular restricting the domain of integration in \eqref{zzeta} and \eqref{xxi} to these respective regions.  Otherwise we obtain an integral which is $\mathcal{O}(h^{\frac{N_0}{8}-d}\vartheta^{-\frac 12})$ which is stronger than needed if $N_0$ sufficiently large.

The assumption $|z'|+|\zeta'| \leq h^{\frac 38}$ means it also suffices to consider $(x_t,\xi_t)$ such that $x_t $ lies in the Fermi coordinate chart and $|x_t'|+|\xi_t'| \leq h^{\frac 14}$.  Otherwise if $|x_t'|+|\xi_t'| \geq h^{\frac 14}$, then \eqref{mubddh} implies
\begin{equation}\label{crudeme2}
h^{-\frac 14}\leq h^{-\frac 12} \mu^{-1}\big(|x_t'| + |\xi_t'| \big) \lesssim h^{-\frac 12} \mu^{-1}\big(|x_t'-z'| + |\xi_t'-\zeta'|\big) +h^{-\frac 18}\mu^{-1},
\end{equation}
meaning \eqref{crudemearbcoord} is $\mathcal{O}(h^{\frac{N}{4}})$ and hence its contribution to \eqref{zzeta}, \eqref{xxi} similarly satisfies stronger bounds.  Similarly, we can assume that $|(\xi_t)_d-1| \ll 1$ since $\supp(q_{h,\tilde\gamma})$ confines $\zeta$ to $|\zeta_d-1| \ll 1$, so we obtain stronger bounds in the complementary region by taking $N_0$ sufficiently large.

The following gives a coordinate approximation to the second part in parentheses in \eqref{crudemearbcoord}.
\begin{proposition}\label{P:edtransfers}
	Suppose $|x_t'|+|\xi_t'| \leq h^{\frac 14} $, $|z'|+|\zeta'|  \leq h^{\frac 38}$, and $|(\xi_t)_d-1| \ll 1$.  We then have the following bounds
\begin{equation}\label{zxtapprox}
	\left| \xi_t\big(\expz\big)-|\xi_t |_{g(x_t)}(z-x_t)_d\right|  \lesssim h^{\frac 14}|z-x_t| + h^{\frac 14} + \big|(z-x_t)_d\big|^2,
\end{equation}
\begin{equation}\label{zetaxitapprox}
\left| \langle \xi_t, \zeta-\xi_t \rangle_{g(x_t)} -  |\xi_t|_{g(x_t)} \big(\zeta_d -(\xi_t)_d\big)\right| \lesssim  h^{\frac 14} +  h^{\frac 14} |\zeta-\xi_t |.
\end{equation}
\end{proposition}
\begin{proof}
	Two key estimates we use are 
	\begin{equation*}
		|g^{jk}(x_t)-\delta^{jk}| \lesssim |x_t'|  \quad \text{ and } \quad \big|\xi_t-|\xi_t|e_d\big| \lesssim |\xi_t'|.
	\end{equation*}
	The first of these follows since $g^{jk}(0,(x_t)_d) = \delta^{jk}$ in the Fermi coordinate system  while the second just uses that
	\begin{equation}\label{edtransfer}
		\big|\xi_t-|\xi_t|e_d\big|^2 = |\xi_t'|^2 +\big(|\xi_t|-(\xi_t)_d\big)^2 =  |\xi_t'|^2 + \bigg[\frac{|\xi_t|^2 -(\xi_t)_\d^2}{ |(\xi_t)_d+|\xi_t|\,| } \bigg]^2 \lesssim |\xi_t'|^2 
	\end{equation}
	The bound \eqref{zxtapprox} then follows by Proposition \ref{P:covariantlogexpn}:
	\begin{equation*}
		\begin{aligned}
					 \big|\xi_t\big(\expz\big)-|\xi_t |_g (z-x_t)_d \big| &\lesssim |(\xi_t-|\xi_t|e_\d)\cdot (z-x_t)|  +|z-x_t|^2 \\
					 &\lesssim |\xi_t'|\big|(z-x_t)_d\big| + \big|(z-x_t)_d\big|^2 + |z'|^2+ |x_t'|^2
		\end{aligned}
	\end{equation*}
	Similarly,
	\begin{equation*}
		\left| \langle \xi_t, \zeta-\xi_t \rangle_{g(x_t)}  -  |\xi_t | \big(\zeta_d -(\xi_t)_d\big)\right| \lesssim |x_t'| + \left| (\xi_t -|\xi_t|e_d)\cdot( \zeta-\xi_t )  \right| \lesssim |x_t'| + |\xi_t'||\zeta-\xi_t| .
	\end{equation*}
\end{proof}

Given \eqref{zxtapprox}, and \eqref{zetaxitapprox}, we now have that in the region where we integrate \eqref{zzeta} or \eqref{xxi}
\begin{equation*}
	|\Gsc_0(z,\zeta,t,x,\xi)| 	\lesssim
	  h^{-\d}\vartheta^{-\frac 12}\big(1+h^{-\frac 12}\big(|z_d-(x_t(x,\xi))_d| + |\zeta_d - (\xi_t(x,\xi))_d|\big)\big)^{-N}
\end{equation*}
To obtain \eqref{zzeta}, we now integrate this bound with respect to $(z',\zeta')$ to get that the left hand side of \eqref{zzeta} is dominated by
\begin{multline*}
			\int 	|\Gsc_0(z,\zeta,t,x,\xi)| 	\big(1+h^{-\frac 12}(|z'|+|\zeta'| )\big)^{-N_0} \,dz\,d\zeta \lesssim \\
			h^{-1}\vartheta^{-\frac 12}\int\big(1+h^{-\frac 12}\big(|z_d-(x_t)_d| + |\zeta_d - (\xi_t)_d|\big)\big)^{-N}dz_dd\zeta_d  \lesssim \vartheta^{-\frac 12}.
\end{multline*}

We now turn to \eqref{xxi}.  As before, left hand side of \eqref{xxi} is dominated by
\begin{multline*}
	\int 	|\Gsc_0(z,\zeta,t,x,\xi)| 	\big(1+h^{-\frac 12}(|x'|+|\xi'| )\big)^{-N_0} \,dx\,d\xi \lesssim \\
	\sup_{x',\xi'} h^{-1}\vartheta^{-\frac 12} \int \big(1+h^{-\frac 12}|z_d-(x_t)_d| + h^{-\frac 12}|\zeta_d - (\xi_t)_d|\big)^{-N}\,dx_d\,d\xi_d 
	\lesssim \vartheta^{-\frac 12}. 
\end{multline*}
The first inequality here follows as before, but for the second, we make a change of variables $(x_d,\xi_d) \mapsto ((x_t(x,\xi))_d, (\xi_t(x,\xi))_d)$ for each $(x',\xi')$ with $|x'|+|\xi'| \leq h^{\frac 38}$.  The second bound follows once we see that
\begin{equation}\label{dmatrixbd}
\det
\begin{bmatrix}
\frac{\prtl (x_t)_d}{\prtl x_d} & \frac{\prtl (x_t)_d}{\prtl \xi_d}\\
\frac{\prtl (\xi_t)_d}{\prtl x_d} & \frac{\prtl (\xi_t)_d}{\prtl \xi_d}\\
\end{bmatrix} 
\gtrsim 1 .
\end{equation}

To see \eqref{dmatrixbd}, we begin by observing the following $2 \times 2$ matrix identity
\begin{equation}\label{detprep1}
\begin{bmatrix}
\xi_t^T\frac{\prtl x_t}{\prtl x}\xi & \xi_t^T\frac{\prtl x_t}{\prtl \xi}\xi \\
\xi_t^T\frac{\prtl \xi_t}{\prtl x}\xi &\xi_t^T\frac{\prtl \xi_t}{\prtl \xi}\xi
\end{bmatrix}
=
\begin{bmatrix}
|\xi|^2 & 0 \\
\xi_t^T\frac{\prtl \xi_t}{\prtl x}\xi & |\xi_t|^2
\end{bmatrix},
\end{equation}
where $\xi,\xi_t$ are treated as column vectors, so their transposes are row vectors.  To see this, first observe that since $x_t,\xi_t$ are homogeneous of degree 0,1 in $\xi$ respectively, it follows that
\begin{equation*}
	\frac{\prtl x_t}{\prtl \xi}\xi=0, \; \frac{\prtl \xi_t}{\prtl \xi}\xi = \xi_t, \;\text{ and }\;\xi = \left(\frac{\prtl x_t}{\prtl x}\right)^T\xi_t.
\end{equation*}
Here the last identity follows from the first two along with the following, a consequence of differential of $(x,\xi) \mapsto (x_t,\xi_t)$ being a symplectic matrix (cf. \eqref{sympblockprop})
\begin{equation*}
	I = \left(\frac{\prtl x_t}{\prtl x}\right)^T\left(\frac{\prtl \xi_t}{\prtl \xi} \right) - \left(\frac{\prtl \xi_t}{\prtl x}\right)^T\left(\frac{\prtl x_t}{\prtl \xi} \right)
\end{equation*}

We now replace $\xi, \xi_t$ in \eqref{detprep1}, by unit vectors $\xi/|\xi|, \xi_t/|\xi_t|$
\begin{equation}\label{detprep2}
	\begin{bmatrix}
		\xi_t^T/|\xi_t|& 0 \\
		0 & \xi_t^T/|\xi_t|
	\end{bmatrix}
	\begin{bmatrix}
		\frac{\prtl x_t}{\prtl x} & \frac{\prtl x_t}{\prtl \xi} \\
		\frac{\prtl \xi_t}{\prtl x} & \frac{\prtl \xi_t}{\prtl \xi} 
	\end{bmatrix}
	\begin{bmatrix}
		\xi/|\xi| & 0 \\
		0 & \xi/|\xi|
	\end{bmatrix}
	=
	\begin{bmatrix}
		|\xi|/|\xi_t| & 0 \\
		\xi_t^T\frac{\prtl \xi_t}{\prtl x}\xi/(|\xi_t||\xi|) & |\xi_t|/|\xi|
	\end{bmatrix}.
\end{equation}
The right hand side here has determinant 1, so \eqref{dmatrixbd} follows once we see that the two matrices are sufficiently close.  Indeed, since  $|\xi'| \leq h^{\frac 38}$ and $ |\xi_t'| \leq h^{\frac 14} $ the argument in \eqref{edtransfer} and the similar bound $|\xi/|\xi| -e_d| \lesssim h^{\frac 14}$ gives that the matrix norm of the difference satisfies
\begin{multline*}
		\left\|\begin{bmatrix}
		\xi_t^T/|\xi_t|& 0 \\
		0 & \xi_t^T/|\xi_t|
	\end{bmatrix}
	\begin{bmatrix}
		\frac{\prtl x_t}{\prtl x} & \frac{\prtl x_t}{\prtl \xi} \\
		\frac{\prtl \xi_t}{\prtl x} & \frac{\prtl \xi_t}{\prtl \xi} 
	\end{bmatrix}
	\begin{bmatrix}
		\xi/|\xi| & 0 \\
		0 & \xi/|\xi|
	\end{bmatrix} - 	
	\begin{bmatrix}
	e_d& 0 \\
	0 & e_d
	\end{bmatrix}
	\begin{bmatrix}
	\frac{\prtl x_t}{\prtl x} & \frac{\prtl x_t}{\prtl \xi} \\
	\frac{\prtl \xi_t}{\prtl x} & \frac{\prtl \xi_t}{\prtl \xi} 
	\end{bmatrix}
	\begin{bmatrix}
	e_d & 0 \\
	0 & e_d
	\end{bmatrix}\right\| 
	\\
	\lesssim 
	h^{\frac 14} \left\|\begin{bmatrix}
	\frac{\prtl x_t}{\prtl x} & \frac{\prtl x_t}{\prtl \xi} \\
	\frac{\prtl \xi_t}{\prtl x} & \frac{\prtl \xi_t}{\prtl \xi} 
	\end{bmatrix}\right\| 
	\lesssim h^{\frac 14}\mu \ll 1.
\end{multline*}

\appendix\section{A proof of Proposition \ref{P:wptpdo}}\label{A:wptpdo}
In this section we provide an alternate proof of Proposition \ref{P:wptpdo}.  Given that our integrals are compactly supported in the fiber variables $\xi$, we are able simplify the proof somewhat.  

	We express the kernel of $\Ssc^* \Ssc$ in a common coordinate system for $x,y,z$ as
\begin{equation}\label{xint}
	\begin{aligned}
		&	\Ssc^* \Ssc(y,z) := (2\pi h)^{-d} \int K(y,z,\xi)\,d\xi\\
		&K(y,z,\xi) := \detexp^{\frac 12}(g_{jk}(y)) \detexp^{\frac 12}(g_{jk}(z)) (\pi h)^{-\frac d2}\int e^{\frac ih \Phi(x,y,z,\xi)} \frac{\tilde\beta^2\big(|\xi|_{g(x)}\big) }{\detexp^{\frac 12}(g_{jk}(x))} \psi(x,y)\psi(x,z)\;dx,\\
		&\Phi(x,y,z,\xi) :=  \xi(\expyno - \expynoz) + \frac i2 d^2(x,y) + \frac i2 d^2(x,z).
	\end{aligned}	
\end{equation}
The factors $ \detexp^{\frac 12}(g_{jk}(y)) \detexp^{\frac 12}(g_{jk}(z)) $ are included so that the coordinate expression is $(\Ssc^* \Ssc f)(y) = \int \Ssc^* \Ssc(y,z)f(z)\,dz$ and Lebesgue measure is used in both $y,z$.

We want to use \eqref{pureasympexpn} to obtain asymptotics for $K(x,y,\xi)$. 
To this end, we use \eqref{logexpn}, and write $d^2(x,y) = |\expyno|_{g(x)}^2$, $d^2(x,z) = |\expynoz|_{g(x)}^2$to obtain from Proposition \ref{P:covariantlogexpn} 
\begin{multline*}
	\Phi = \xi_j(y-z)^j + \frac i2 g_{jk}(x)(y-x)^j (y-x)^k + \frac i2 g_{jk}(x)(z-x)^j (z-x)^k \\ 
	+\frac 12 \xi_\ell \Gamma_{jk}^\ell (x)(y-x)^j (y-x)^k - \frac 12 \xi_\ell \Gamma_{jk}^\ell(x) (z-x)^j (z-x)^k +
	r_1(x,y,z,\xi), \\
	\text{ where } \quad \big| \prtl_{\xi}^\alpha \prtl_{x}^{\beta}r_1 \big| \lesssim_{\alpha, \beta} |x-y|^{\max(0,3-|\beta|)}.
\end{multline*}
Calculating first derivatives in $x$, we obtain
\begin{equation}\label{Phicrit}
	\frac{\prtl \Phi}{\prtl x_k} = ig_{jk}(x)(2x-y-z)^j +  \xi_\ell \Gamma_{jk}^\ell(x) (y-z)^j +\mathcal{O}(|x-y|^2 + |x-z|^2)
\end{equation}
This illustrates the challenge in treating complex phases:  there is a critical point when $x=y=z$, but otherwise one may not exist in the real domain.  We now calculate the second derivatives as 
\begin{equation}\label{Phihessian}
	\frac{\prtl^2 \Phi}{\prtl x_j \prtl x_k} = 2ig_{jk}(x) + \mathcal{O}(|y-z|+ |x-y| + |x-z|),
\end{equation}	
which defines nonsingular Hessian when $x=y=z$.  

We now take an almost analytic extension of the phase and amplitude in \eqref{xint} and in an abuse of notation, we continue to use $x$ to denote the complexified variable and $\frac{\prtl}{\prtl x_j}$ to denote the complex derivative in $j$.  For $y,z$ sufficiently close we let $X(y,z,\xi)$ parameterize the critical locus
\begin{equation*}
	\frac{\prtl \Phi}{\prtl x_j}\big(X(y,z,\xi),y,z,\xi\big) =0, \qquad j = 1,\dots, d.
\end{equation*}
The observation following \eqref{Phicrit} implies that $X|_{y=z} = y$ and hence
\begin{equation}\label{Phitaytay}
	X_k(y,z,\xi) = y_k +\mathcal{O}(|y-z|) = z_k +\mathcal{O}(|y-z|) 
\end{equation}
Consequently, for some functions $r_{2,\alpha}$ in a bounded subset of $C^\infty$
\begin{equation*}
	\Phi(X(y,z,\xi),y,z,\xi) = (y-z) \cdot \xi+ r_2(y,z,\xi), \quad \text{ where }\quad r_2 = \sum_{|\alpha|=2} (y-z)^\alpha r_{2,\alpha}(y,z,\xi).
\end{equation*}

We now have that \eqref{pureasympexpn} implies that
\begin{equation*}
	K(y,z,\xi) \sim \detexp^{\frac 12}(g_{jk}(y)) \detexp^{\frac 12}(g_{jk}(z)) \detexp^{-\frac 12} \Big(\frac{1}{2i} d_x^2 \Phi(X,y,z,\xi)\Big)e^{\frac ih \Phi(X,y,z,\xi)} 
	\Big( \sum_{j \geq 0} h^j \tilde A_j(y,z,\xi) \Big).
\end{equation*}
Here each $\tilde A_j$ lies in a bounded subset of $C^\infty$ and is supported in a small neighborhood of the diagonal $y=z$.  Moreover, 
\begin{equation}\label{A0def}
		\tilde A_0(y,z,\xi) = \frac{\tilde\beta^2\big(|\xi|_{g(x)}\big) }{\detexp^{\frac 12} (g_{jk}(x))} \psi(x,y)\psi(x,z)\bigg|_{x=X(y,z,\xi)}\\
		 = \frac{\tilde\beta^2\big(|\xi|_{g(x)}\big) }{\detexp^{\frac 12} (g_{jk}(x))}\bigg|_{x=y} \psi^2(y,z) + \mathcal{O}(|y-z|^2). 
\end{equation}
where we used \eqref{Phitaytay} in the last identity.  Similarly, by \eqref{Phihessian} (and \eqref{Phitaytay} as before),
\begin{equation*}
	\detexp^{-\frac 12} \Big(\frac{1}{2i} d_x^2 \Phi(X,y,z,\xi)\Big)  = \detexp^{-\frac 12} (g_{jk}(z)) + \mathcal{O}(|y-z|).
\end{equation*}
Now let
\begin{equation*}
	A_j(y,z,\xi) := \detexp^{\frac 12}(g_{jk}(y)) \detexp^{\frac 12}(g_{jk}(z)) \detexp^{-\frac 12} \Big(\frac{1}{2i} d_x^2 \Phi(X,y,z,\xi)\Big)\tilde A_j(y,z,\xi) ,
\end{equation*}
hence
\begin{equation}\label{prinsymbapp}
	A_0(y,z,\xi) = \tilde\beta^2\big(|\xi|_{g(y)}\big) + \mathcal{O}(|y-z|).
\end{equation}

We now use a routine argument (see e.g. \cite[Theorem 3.2.1]{soggefica}) to replace $\Phi(X,y,z,\xi)$ by the usual pseudodifferential phase; it is similar to the one in \eqref{kthderivI}.  Define
\begin{equation*}
\tilde{K}(t,y,z) = \frac{1}{(2\pi h)^d} \int e^{\frac ih \Phi_t(y,z,\xi)}A(t,y,\xi)\,d\xi \quad \text{ where } \quad	\Phi_t(y,z,\xi) := \xi\cdot (z-y) + tr_2(y,z,\xi) ,
\end{equation*}
so that $\tilde K (1,y,z) = (2\pi h)^{-d}\int K(y,z,\xi)\,d\xi $ and $\tilde K(0,y,z)$ is the kernel of a pseudodifferential operator with standard phase.  Note that $\Phi_t$ satisfies bounds
\begin{equation}\label{Phitbds}
	|d_\xi \Phi_t| \approx |y-z|  \quad \text{ and } \quad |\prtl_\xi^\alpha \Phi_t| \lesssim_\alpha |y-z|^2 \text{ if } |\alpha| \geq 2.
\end{equation}

We want to calculate $\tilde K (1,y,z)$ by taking a Taylor expansion in $t$ centered at $0$.  The terms in the Taylor expansion are $\frac{1}{k!}\prtl_t^k\tilde K (0,y,z)$ where
\begin{equation}\label{tildeKint}
	\prtl_t^k \tilde K|_{t=0} =\frac{h^{-k}}{(2\pi h)^d} \int e^{\frac ih (y-z)\cdot \xi}(i r_2(y,z,\xi))^k A(t,y,\xi)\,d\xi = \frac{h^k}{(2\pi h)^d} \int e^{\frac ih (y-z)\cdot \xi} B_k(t,y,\xi)\,d\xi 
\end{equation}
for some smooth, bounded, $B_k$ supported in the same set as $A$ and $B_k$ admitting an asymptotic series $\sum_{j\geq 0}h^j b_{k,j}$.  Indeed, the factor of $(ir_2)^k$ in the middle expression vanishes to order $2k$ along the diagonal $y=z$ and hence a routine integration by parts in $\xi$ establishes the last expression.  

Next, we claim that for any $k \in\mathbb{N}$ and $t \in [0,1]$, 
\begin{equation}\label{tderivkernel}
	|\prtl_t^k \tilde K (t,y,z)| \lesssim_k 	h^{-k}|y-z|^{2k}(1+h^{-2}|y-z|^2)^{-(d+k)}.
\end{equation}
By Young's inequality this kernel gives then rise to an operator on $L^2(M)$ whose norm is $\mathcal{O}(h^k)$.  To see this, define the differential operator
\begin{equation*}
	\mathcal{L} = \frac{1 - ih^{-1}d_\xi \overline{\Phi_t} \cdot d_\xi}{1+h^{-2} |d_\xi \Phi_t|^2}
	\quad \text{ so that } \quad \mathcal{L} e^{\frac ih \Phi_t} = e^{\frac ih \Phi_t} .
\end{equation*}
Using \eqref{Phitbds} we have for $|\alpha| \geq 1$,
\begin{equation*}
	\left| \prtl_\xi ^\alpha\Big( \frac{1 - ih^{-1}d_\xi \Phi_t }{1+h^{-2} |d_\xi \Phi_t|^2} \Big) \right| 
	\lesssim_\alpha
	(1+h^{-2}|y-z|^2)^{-1} |y-z|^2 .
\end{equation*}
The bound \eqref{tderivkernel} now follows by integrating by parts sufficiently many times with respect to $\mathcal L$ in the integral defining $\prtl_t^k \tilde K (t,y,z) $, namely, \eqref{tildeKint} but with the phase $(y-z)\cdot\xi$ replaced by $\Phi_t$.

The preceding means we can apply the Borel lemma to the formal series $\sum_{k \geq 0} \frac{1}{k!} \prtl_t^k \tilde K (0,y,z)$ (and the corresponding asymptotic series for each $B_k$) to obtain a symbol $s\sim \sum_{j\geq 0} h^js_j$ which determines $\Ssc^* \circ \Ssc$ as a semiclassical PDO.  Its principal symbol is given by $A_0$, but given \ref{prinsymbapp}, the integration by parts argument above means that it can be replaced by $\tilde\beta^2(|\xi|_{g(y)})$.

\bibliographystyle{amsalpha}
\bibliography{bibtexdata}
\end{document}